\documentclass[a4paper,10pt]{article}

\usepackage{lmodern}
\usepackage[T1]{fontenc}
\usepackage[english]{babel}
\usepackage[utf8]{inputenc}

\usepackage{amsmath}
\usepackage{amssymb}
\usepackage{amsthm}
\usepackage{mathtools}
\usepackage{mathrsfs}
\usepackage{abraces}
\usepackage{bbm}

\usepackage{tikz,tikz-3dplot}
\usetikzlibrary{cd,patterns,positioning,decorations.markings,decorations.pathmorphing,decorations.pathreplacing}
\usepackage{xstring}
\usepackage{ifthen}
\usepackage{pdflscape}
\usepackage{diagbox}
\usepackage[labelfont=bf,font=small]{caption}

\usepackage{titlesec}
\titleformat{\section}{\normalfont\fillast\scshape}{\thesection.}{.5em}{}
\titleformat{\subsection}[runin]{\normalfont\bfseries}{\thesubsection.}{.5em}{}[.]

\usepackage{enumitem}

\setlist[description]{%
	font={\normalfont\itshape},
	}
\renewcommand\labelenumi{(\textit{\roman{enumi}})}
\renewcommand\theenumi\labelenumi

\usepackage[autostyle]{csquotes}
\usepackage[backend=biber,style=trad-abbrv,doi=false,isbn=false,url=false]{biblatex}
\newtheorem{thm}{Theorem}[section]

\newtheorem{pro}[thm]{Proposition}
\newtheorem{cor}[thm]{Corollary}
\newtheorem*{thm*}{Theorem}
\newtheorem*{lem*}{Lemma}
\newtheorem*{pro*}{Proposition}
\newtheorem*{cor*}{Corollary}

\newtheorem{cla*}{Claim}
\theoremstyle{definition}
\newtheorem{dfn}[thm]{Definition}
\newtheorem*{dfn*}{Definition}
\theoremstyle{remark}
\newtheorem{rmk}[thm]{Remark}
\newtheorem{exa}[thm]{Example}
\newtheorem*{rmk*}{Remark}
\newtheorem*{rmks*}{Remarks}
\newtheorem*{exa*}{Example}

\DeclareMathOperator{\Hom}{Hom}

\DeclareMathOperator{\id}{id}
\DeclareMathOperator{\JW}{JW}

\newcommand{\cohdeg}[1]{\scriptstyle\textcolor{green}{#1}}

\newcommand{\Adg}{\mathscr{B}^{\mathrm{dg}}_W}
\newcommand{\Ao}{A^\bullet}
\newcommand{\Bo}{B^\bullet}
\newcommand{\C}{\mathcal{C}}
\newcommand{\Cb}{\mathcal{C}^{\mathrm{b}}}
\newcommand{\CC}{\mathscr{C}}
\newcommand{\Cdg}{\mathcal{C}_{\mathrm{dg}}}
\newcommand{\Cdgb}{\mathcal{C}_{\mathrm{dg}}^{\mathrm{b}}}
\newcommand{\hh}{\mathfrak{h}}
\newcommand{\HH}{\mathbf{H}}
\newcommand{\ZZ}{\mathbb{Z}}

\newcommand{\D}{\mathscr{H}}
\newcommand{\DBS}{\D_{\mathrm{BS}}}

\newcommand{\Kb}{\mathcal{K}^{\mathrm{b}}}

\newcommand{\Gb}{\mathcal{G}^{\mathrm{b}}}
\newcommand{\Homb}{\Hom^\bullet}
\newcommand{\un}{\mathbbm{1}}
\newcommand{\boe}{\mathbf{e}}
\newcommand{\bof}{\mathbf{f}}
\newcommand{\boi}{\mathbf{i}}

\newcommand{\uw}{\underline{w}}
\newcommand{\uz}{\underline{z}}

\newcommand{\uv}{\underline{v}}
\newcommand{\ux}{\underline{x}}

\newcommand{\uom}{\underline{\omega}}

\newcommand{\kk}{\Bbbk}

\newcommand{\qn}[2]{[#1]_{#2}}

\newcommand{\Roudg}[1]{F_{#1}^\bullet}

\newcommand{\Laur}{\ZZ[v,v^{-1}]}

\newcommand{\cword}{\mathfrak{W}_S}
\newcommand{\bword}{\mathfrak{W}_{\Sigma}}

\newcommand{\gr}[1]{\textcolor{green}{#1}}
\newcommand{\re}[1]{\textcolor{red}{#1}}
\newcommand{\bl}[1]{\textcolor{blue}{#1}}
\newcommand{\grad}[1]{{#1}^{\mathrm{gr}}}

\newcommand{\Tate}[1]{\langle #1 \rangle}

\newcommand{\writesummand}[1]{%
	\def\tempstring{#1}%
	\StrSubstitute{\tempstring}{s}{B_s}[\tempstring]%
	\StrSubstitute{\tempstring}{t}{B_t}[\tempstring]%
	\StrSubstitute{\tempstring}{u}{B_u}[\tempstring]%
	\StrSubstitute{\tempstring}{v}{B_v}[\tempstring]%
	\StrSubstitute{\tempstring}{O}{\Omega^+}[\tempstring]%
	\StrSubstitute{\tempstring}{o}{\Omega^-}[\tempstring]%
	\StrSubstitute{\tempstring}{l}{(1)}[\tempstring]%
	\StrSubstitute{\tempstring}{2}{(2)}[\tempstring]%
	\StrSubstitute{\tempstring}{9}{(-1)}[\tempstring]%
	\StrSubstitute{\tempstring}{8}{(-2)}[\tempstring]%
	\StrSubstitute{\tempstring}{n}{R}[\tempstring]%
	\tempstring%
	}
\newcommand{\LL}{\mathbb{L}}

\newcommand{\summandnode}[4]{%
		\node (#4) at (#2,#3) {$\writesummand{#1}$}%
		}
\definecolor{green}{rgb}{0,.5,0}

\tikzset{%
	virtual/.style  = {font=\scriptsize,draw=black,circle,dashed,inner sep=2pt},
	real/.style     = {circle,inner sep=2pt,font=\scriptsize},
	patch/.style    = {gray,pattern=north west lines, pattern color=gray},
	positive/.style = {line width=2pt},
	negative/.style = {line width=.25pt,double distance=1.5pt},
	braid/.style = {thick,double distance=2pt,rounded corners},
	localized/.style = {decorate,decoration={snake,amplitude=0.5mm,segment length=2mm}},
}
\tikzset{%
	pics/dot/.style n args={1}{%
		code={%
			\draw (0,0) -- (0,#1);
			\fill (0,#1) circle (1.5pt); 		
			}
	}
}
\def\d{.5cm}\def\h{1cm}

\colorlet{s}{red}
\colorlet{t}{blue}
\colorlet{u}{green}
\colorlet{v}{yellow}
\colorlet{g}{violet}
\colorlet{dgdegree}{green!40!black}

\tikzset{%
	pics/topdot/.style n args={3}{%
		code= {%
			\StrLen{#1}[\length]%
			\def \d{0.35}%
			\def \h{1}%
			\node[anchor=east,inner sep=0pt] at ({-(\length+1)*\d/2},0) {$#3$};%
			\foreach \i in {1,...,\length}{%
				\StrChar{#1}{\i}[\color]%
				\ifthenelse{\equal{\color}{O}}{%
					\draw[->] ({-(\length+1)*\d/2+\i*\d},-\h/2)--({-(\length+1)*\d/2+\i*\d},{1/15*\h/2});%
					\draw ({-(\length+1)*\d/2+\i*\d},{1/15*\h/2})--({-(\length+1)*\d/2+\i*\d},\h/2);%
				}%
				{%
					\ifthenelse{\equal{\color}{o}}{%
						\draw ({-(\length+1)*\d/2+\i*\d},-\h/2)--({-(\length+1)*\d/2+\i*\d},{-1/15*\h/2});%
						\draw[<-] ({-(\length+1)*\d/2+\i*\d},{-1/15*\h/2})--({-(\length+1)*\d/2+\i*\d},\h/2);%
					}%
					{%
						\ifthenelse{\i=#2}{%
							\draw[draw=\color] ({-(\length+1)*\d/2+\i*\d},-\h/2)--({-(\length+1)*\d/2+\i*\d},0);%
							\fill[fill=\color] ({-(\length+1)*\d/2+\i*\d},0) circle (2pt);%
						}%
						{%
							\draw[draw=\color] ({-(\length+1)*\d/2+\i*\d},-\h/2)--({-(\length+1)*\d/2+\i*\d},\h/2);%
						}%
					}%
				}%
			}%
		}%
	}
}

\tikzset{%
	pics/bottomdot/.style n args={3}{%
		code= {%
			\StrLen{#1}[\length]%
			\def \d{0.35}%
			\def \h{1}%
			\ifthenelse{\equal{\length}{1}}{%
					\node[anchor=east,inner sep=0pt] at ({-(\length)*\d/2},{.2*\h}) {$#3$};%
				}%
				{%
					\node[anchor=east,inner sep=0pt] at ({-(\length+1)*\d/2},0) {$#3$};%
				}
			\foreach \i in {1,...,\length}{%
				\StrChar{#1}{\i}[\color]%
				\ifthenelse{\equal{\color}{O}}{%
					\draw[->] ({-(\length+1)*\d/2+\i*\d},-\h/2)--({-(\length+1)*\d/2+\i*\d},{1/15*\h/2});%
					\draw ({-(\length+1)*\d/2+\i*\d},{1/15*\h/2})--({-(\length+1)*\d/2+\i*\d},\h/2);%
				}%
				{%
					\ifthenelse{\equal{\color}{o}}{%
						\draw ({-(\length+1)*\d/2+\i*\d},-\h/2)--({-(\length+1)*\d/2+\i*\d},{-1/15*\h/2});%
						\draw[<-] ({-(\length+1)*\d/2+\i*\d},{-1/15*\h/2})--({-(\length+1)*\d/2+\i*\d},\h/2);%
					}%
					{%
						\ifthenelse{\i=#2}{%
							\draw[draw=\color] ({-(\length+1)*\d/2+\i*\d},\h/2)--({-(\length+1)*\d/2+\i*\d},0);%
							\fill[fill=\color] ({-(\length+1)*\d/2+\i*\d},0) circle (2pt);%
						}%
						{%
							\draw[draw=\color] ({-(\length+1)*\d/2+\i*\d},-\h/2)--({-(\length+1)*\d/2+\i*\d},\h/2);%
						}%
					}%
				}%
			}%
		}%
	}
}

\tikzset{%
	pics/patch/.style n args={1}{
		code={
			\node[scale=#1,circle,draw,fill=white,inner sep=0,font=\tiny] at (0,0) {$\times$};
			}
	}
}

\tikzset{%
	twoframe/.pic={%
		\draw[positive] (.5,0)--(2.5,0);\draw[negative] (.5,2)--(2.5,2);
	}
}

\tikzset{%
	threeframe/.pic={%
		\draw[positive] (.5,0)--(3.5,0);\draw[negative] (.5,2)--(3.5,2);
	}
}

\title{\bfseries Soergel Calculus with patches}
\author{Leonardo Maltoni}
\begin{document}
	\maketitle
	\begin{abstract}
		We adapt the diagrammatic presentation of the Hecke category to the dg category formed by the
		standard representatives for the Rouquier complexes. We use this description to recover basic results
		about these complexes, namely the categorification of the relations of the braid group and the Rouquier formula.
	\end{abstract}
		\section*{Introduction}				
			The diagrammatic Hecke category is an extremely powerful tool.
			It was introduced as a presentation by generators and relations
			of the category of Soergel bimodules by Elias and Williamson \cite{EW}, 
			based on previous work by 
			Elias-Khovanov \cite{EliKho} and Elias \cite{Elias_two}.
			One reduces to the subcategory formed by certain \emph{Bott-Samelson objects},
			which can be represented as sequences of simple reflections. Morphisms 
			between these objects are described as linear combinations of diagrams identified under some relations.
			This presentation has been useful in several occurrences: see for instance Remark \ref{rmk_usesofdiag}.
						
			On the other hand, Rouquier \cite{Rou_cat} introduced some complexes of Soergel bimodules
			to categorify actions of the braid group on categories. These \emph{Rouquier complexes}
			are all possible tensor products between certain \emph{standard} and \emph{costandard} complexes
			$F_s$ and $F_s^{-1}$ (for each simple reflection $s$). These objects satisfy a categorified 
			version	of the relations of the braid group: see Proposition \ref{pro_braidRou} below.
			This allows us to unambiguously consider the object $F_\omega$ associated to a braid $\omega$.
			An important property of these complexes is the so-called \textit{Rouquier formula}.
			Given $w\in W$, its \emph{positive lift} is obtained by replacing
			each simple reflection in any reduced word for $w$ by the corresponding positive 
			generator of the braid group. This is a well defined element of the braid group.
			Similarly one defines the \emph{negative lift}.
			Now, if $\omega$ is the positive lift of $w\in W$, and $\nu$ is the negative 
			lift of another element $v$, we have
			\[
				\Hom_{\Kb(\D)}(F_{\omega},F_{\nu})=	\begin{cases}
														R	&\text{if $w=v$,}\\
														0	&\text{otherwise.}
													\end{cases}
			\]		
			Rouquier \cite{Rou_der} showed this for Weyl groups, by applying the Soergel functor to
			transfer the problem to category $\mathcal{O}$ and
			using the vanishing properties of the extension groups between standard and costandard objects therein.
			He also conjectured that the formula should hold for general Coxeter systems.
			This was shown by Libedinsky and Williamson \cite{LibWil} and, with other methods, by Makisumi \cite{Maki}.			
			
			In this paper we consider the dg category of complexes $\Cdgb(\D)$ with objects in the Hecke category.
			We adapt the diagrammatic description of $\D$ to one for the subcategory of $\Cdgb(\D)$ generated by
			representatives of the Rouquier complexes of the form $F_{s_1}^{\pm 1} \otimes F_{s_2}^{\pm 1} \otimes \dots \otimes F_{s_n}^{\pm 1}$.
			The graded objects associated with each of them splits into Bott-Samelson objects 
			(one for each subexpression of $s_1s_2\dots s_n$), 
			each of which lies in a certain cohomological degree.
			Roughly speaking, morphisms between two such complexes are again linear 
			combinations of diagrams, but in each diagram one has 
			to specify the starting and ending summand. To do this one covers some of the boundary 
			points by \emph{patches}, thus picking a certain subexpression of the starting (and ending) word. 
			One can then give an easy description of the differential map.
			
			With this language we recover the classical results above: the braid relations and the Rouquier formula.
			We also give a recipe to explicitly find the homotopy equivalences giving braid relations, and 
			write those for $m=2,3$.
		\section{Review of the diagrammatic Hecke category}\label{sec_diagHeck}
		We now recall the construction,
		by Elias and Williamson \cite{EW}, of the \emph{diagrammatic Hecke category} associated with 
		(a realization of) an arbitrary Coxeter system.	First we fix some notation.
		\subsection{Coxeter systems, braid groups and Hecke algebras}\label{subs_CoxbraHeck}
		Let $(W,S)$ be a Coxeter system, with $|S|<\infty$.
		Let $\cword$ be the free monoid generated by $S$. Its elements are called \emph{Coxeter words}
		and will be denoted by underlined letters. 
		If $\uw$ is a Coxeter word, we say that it \emph{expresses}
		the element $w\in W$ if $w$ is its image via the natural morphism $\cword\rightarrow W$.
			
		Let $\ell(\uw)$ denote the length of $\uw$, namely
		the number of its letters.
		The word $\uw$ is called \emph{reduced} if there is no shorter word expressing $w$. This defines
		a \emph{length function} $\ell: W\rightarrow \mathbb{N}$, where $\ell(w)$ is the length
		of any reduced word for $w$.

		The \emph{Bruhat graph} is the directed graph whose vertices are the elements $w\in W$ and 
		there is a directed edge from $w$ to all elements of the form $wt$, with $t$ a reflection 
		(i.e., the conjugate of an element in $S$) 
		such that $\ell(wt)>\ell(w)$. Then we define the \emph{Bruhat order} as the transitive closure
		of the relation defined by the Bruhat graph.

		For $\uw, \ux \in \cword$, we say that $\ux$ is a \emph{subword} of $\uw$, and we 
		write $\ux\preceq \uw$, 
		if $\ux$ is obtained from $\uw$ by erasing some letters. In other words,
		if $\uw=s_1s_2\cdots s_n$ with $s_i\in S$, then 
		$\ux$ is of the form $s_{i_1}s_{i_2}\cdots s_{i_r}$, 
		with $1\le i_1 < i_2 < \dots < i_r\le n$. A \emph{subexpression} is instead the datum of a subword
		together with the information of the precise positions of the letters in the original word. This 
		can be encoded in a \emph{01-sequence} $\boe\in\{0,1\}^n$ where the ones correspond to the letters
		that we pick and the zeros to those we do not pick. 

		For a Coxeter word $\uw=s_1\cdots s_k$, and a subexpression (or 01-sequence)
		$\boe=(e_1,\dots, e_k)\in \{0,1\}^k$, we say that $\boe$ \textit{expresses} the element $x\in W$
		if $x=s_1^{e_1}\cdots s_k^{e_k}$ (where $s_i^0:=1$): in this case we write $x=\uw^{\boe}$.
		Then a 01-sequence determines the elements
		$x_i:=\uw_{\le i}^{\boe_{\le i}}\in W$, where $\uw_{\le i}$ is the word consisting of the 
		first $i$ letters of $\uw$ and $\boe_{\le i}$ is the 01-sequence consisting of the first 
		$i$ symbols of $\boe$. This describes a path along the Bruhat graph, called \textit{Bruhat stroll}.	

		We define the \emph{path dominance} partial order on subexpressions by putting $\boe\le\boe'$ if
		$x_i\le x_i'$ for all $i$, where $x_i$ and $x_i'$ are the elements of the Bruhat strolls for $\boe$ and $\boe'$ respectively.		

%
%
		Let $B_W$ be the \emph{Artin-Tits group} (or \emph{generalized braid group}) associated to $(W,S)$.
		Let $\sigma_s$ denote the generator of $B_W$ associated to $s\in S$.
		We have a natural surjection $B_W\twoheadrightarrow W$ sending both $\sigma_s$ and 
		$\sigma_s^{-1}$ to $s$. Let $\Sigma^+=\{\sigma_s\}_{s\in S}$,
		$\Sigma^-=\{\sigma_s^{-1}\}_{s\in S}$ and $\Sigma=\Sigma^+ \sqcup \Sigma^-$. 
		Let $\bword$ be the free monoid generated by $\Sigma$. Its elements are called 
		\emph{braid words}. We call the elements of $\Sigma^+$ (and $\Sigma^-$) \emph{positive letters}
		(respectively \emph{negative letters}).
		Let $\varpi:\bword\rightarrow (\ZZ,+)$ be the morphism of monoids defined by
		$\varpi(\sigma_s)=1$ and $\varpi(\sigma_s^{-1})=-1$.
		In other words $\varpi$ counts the difference between the number of positive and negative letters.

		A braid word is called \emph{positive} (or \emph{negative}) if all its letters are.
		The notion of subword is defined in $\bword$ in the same way as in $\cword$.

		We call \emph{Coxeter projection} the map $\bword\rightarrow\cword$, sending both 
		$\sigma_s$ and $\sigma_s^{-1}$ to $s$. This has two distinguished sections: any Coxeter word $\uw=s_1s_2\dots s_k$
		has a \emph{positive word lift} $\sigma_{s_1}\sigma_{s_2}\dots \sigma_{s_k}$
		and a \emph{negative word lift} $\sigma_{s_1}^{-1}\sigma_{s_2}^{-1}\dots \sigma_{s_k}^{-1}$. 
		These also descend to well defined sections $W\rightarrow B_W$ sending $w$ to 
		the braid expressed by the positive (or negative) word lift of any reduced word for $w$. 
		This is called \emph{positive} (or \emph{negative}) \emph{lift} of $w$.
		
		Consider the ring $\Laur$ of Laurent polynomials in one 
		variable with integer coefficients. Let $\HH_{(W,S)}$ (or simply $\HH$) denote the \emph{Hecke algebra}
		associated to $(W,S)$. Recall that this is the
		$\Laur$-algebra generated by $\{H_s\}_{s\in S}$ with relations 
		\[
			\begin{cases}
				(H_s+v)(H_s-v^{-1})=0& 	\text{for all $s\in S$}\\
				\underbrace{H_sH_t\dots}_{m_{st}}=\underbrace{H_tH_s\dots}_{m_{st}}& 		\text{if $m_{st}<\infty$}
			\end{cases}
		\]			
		\subsection{Realizations of Coxeter systems}\label{sec_realCox} For a Coxeter system
		$(W,S)$ and a commutative ring $\kk$, we consider a \emph{realization} $\hh$
		of $W$, in the sense of \cite[\S\,3.1]{EW}. Recall that this is a free,
		finite rank $\kk$-module $\mathfrak{h}$ with the following properties:
		\begin{enumerate}
			\item there are elements
				$\{\alpha_s^\vee\}_{s\in S} \subset \mathfrak{h}$ and 
				$\{ \alpha_s \}_{s\in S} \subset \mathfrak{h}^*=\Hom(\mathfrak{h},k)$,
				called respectively \emph{simple coroots} and \emph{simple roots}, such that
				$\langle \alpha_s,\alpha_s^\vee\rangle =2$, for each $s\in S$;
			\item the group $W$ acts linearly on $\hh$ via 
				$s(v)=v-\langle \alpha_s, v\rangle \alpha_s^{\vee}$ and on $\hh^*$
				via $s(\lambda)=\lambda-\langle \lambda,\alpha_s^\vee\rangle\alpha_s$;
			\item the condition \cite[\S\ 3.1, (3.3)]{EW} holds. 
				\begin{rmk}
					This is a technical condition to ensure
					\emph{rotation invariance} of \emph{Jones-Wenzl morphisms} which will be mentioned later.
					One could prefer the more precise condition found by Hazi \cite{Haz}
					which also strengthen the equivalence between the diagrammatic category and 
					Abe's version of Soergel bimodules \cite{Abe}.		
				\end{rmk}
		\end{enumerate}
%

		Let $R=S(\mathfrak{h}^*)$, with $\mathfrak{h}^*$ in degree 2. The above action of $W$ on 
		$\mathfrak{h}^*$ extends naturally to $R$. 
		We define, for each $s\in S$, the \emph{Demazure operator} $\partial_s : R \rightarrow R$, via:
		\[
						f \mapsto \frac{f-s(f)}{\alpha_s}.
		\]
		Then, in particular, we have
		$\partial_s(\alpha_t)=\langle \alpha_t,\alpha_s^\vee\rangle$.
		
		Examples of realizations can be found in \cite[Ex.\ 3.3]{EW}.
		In this paper we will only consider \emph{balanced} realizations (see \cite[Definition 3.7]{EW})
		that satisfy \emph{Demazure surjectivity} (see \cite[Assumption 3.9]{EW}).
%
		\subsection{The diagrammatic Hecke category}\label{subs_defcatHeck}
		Given a realization $\mathfrak{h}$ over $\kk$ of a Coxeter system $(W,S)$, 
		one constructs the corresponding diagrammatic Hecke category $\D=\D(\mathfrak{h},\kk)$.
		This is a $\kk$-linear monoidal category enriched in graded $R$-bimodules.
		First one defines the Bott-Samelson category $\D_{\mathrm{BS}}$ by generators and relations,
		then one gets $\D$ as the Karoubi envelope of the closure of $\D_{\mathrm{BS}}$ by direct 
		sums and grade shifts.	
		\begin{enumerate}
			\item The objects of $\DBS$ are generated monoidally by objects $B_s$ for $s\in S$. So a general object
				corresponds to a Coxeter word: if
				$\uw=s_1\dots s_n$, let $B_{\uw}$ denote the object
				$B_{s_1}\otimes \cdots \otimes B_{s_n}$. We also let $\un$ denote the monoidal unit, corresponding to the 
				empty word. 
			\item Morphisms in $\Hom_{\DBS}(B_{\uw_1},B_{\uw_2})$
				are $\kk$-linear combinations of \emph{Soergel graphs}, which are defined as follows.
				\begin{itemize}
					\item We associate a color to each simple reflection.
					\item A Soergel graph is then a colored, \emph{decorated} planar graph 
						contained in the strip $\mathbb{R}\times [0,1]$, with boundary in 
						$\mathbb{R}\times\{0,1\}$.
					\item The bottom (and top) boundary is the arrangement
						of boundary points colored according to the letters of the source word $\uw_1$ 
						(and the target word $\uw_2$, respectively).
					\item The edges of the graph are colored in such a way that those connected with the 
						boundary have colors consistent with the boundary points.
					\item The other vertices of the graph are either: 
						\begin{enumerate}
							\item univalent (called \emph{dots}), which are declared of degree 1, or;
							\item trivalent	with three edges of the same color, of degree $-1$, or;
							\item $2m_{st}$-valent with edges of alternating colors corresponding to $s$ and $t$,
							 	if $m_{st}<\infty$, of degree 0.
								We also call them $(s,t)$-ars.
						\end{enumerate}
						
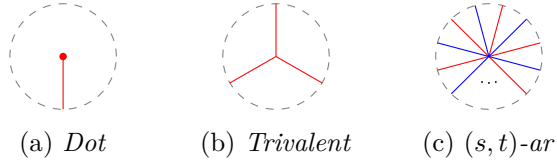
\begin{center}
							\begin{tikzpicture}[baseline=0,scale=.7]
								\def\shift{4cm}
								\draw[gray,dashed] (0,0) circle (1cm); %
									\draw[red] (0,-1)--(0,0); \fill[red] (0,0) circle (2pt);
								\node[anchor=north,inner sep=8pt] at (0,-1) {(a) \textit{Dot}};
								\begin{scope}[xshift=\shift]
									\draw[gray,dashed] (0,0) circle (1cm); %
										\draw[red] (0,1)--(0,0); %
											\draw[red] (-30:1cm)--(0,0)--(-150:1cm);
									\node[anchor=north,inner sep=8pt] at (0,-1) {(b) \textit{Trivalent}};
								\end{scope}
								\begin{scope}[xshift=2*\shift,rotate=-15]
									\draw[gray,dashed] (0,0) circle (1cm); %
										\draw[red] (0,1)--(0,0); %
											\draw[red] (-30:1cm)--(150:1cm);%
												\draw[red](-150:1cm)--(30:1cm);
										\draw[blue,rotate=30] (0,1)--(0,0); %
											\draw[blue,rotate=30] (-30:1cm)--(150:1cm);%
												\draw[blue,rotate=30](-150:1cm)--(30:1cm);
									\foreach \i in {-60,-75,-90}{%
											\fill (\i:.5cm) circle (.5pt);
										}
								\end{scope}
								\node[anchor=north,xshift=1.4*\shift,inner sep=8pt] at (0,-1) {(c) \textit{$(s,t)$-ar}};
							\end{tikzpicture}
						\end{center}
					\item Decorations are boxes labeled by homogeneous elements in 
						$R$ that can appear in
						any region (i.e., connected component of the complement of the graph): 
						we will usually omit the boxes and just write the polynomials.
				\end{itemize}
				Then, composition of morphisms is given by gluing diagrams vertically, whereas tensor product
				is given by gluing them horizontally. The identity morphism of the object $B_{\uw}$ is
				the diagram with parallel vertical strands colored according to the word $\uw$.

			\item These diagrams are identified via some relations:
				\begin{description}
					\item[Polynomial relations.] 
						The first relations impose additivity and multiplicativity of polynomial boxes, more precisely:
						\begin{align}
							& \text{\emph{Addition}:}  &
							\begin{tikzpicture}[baseline=-0.1cm,scale=.7, transform shape]
								\draw[gray,dashed] (0,0) circle (1cm); \node[draw] at (0,0) {$f$};
							\end{tikzpicture}
							+
							\begin{tikzpicture}[baseline=-0.1cm,scale=.7, transform shape]
								\draw[gray,dashed] (0,0) circle (1cm); \node[draw] at (0,0) {$g$};
							\end{tikzpicture}
							&=
							\begin{tikzpicture}[baseline=-0.1cm,scale=.7, transform shape]
								\draw[gray,dashed] (0,0) circle (1cm); \node[draw] at (0,0) {$f+g$};
							\end{tikzpicture}
							\\
							& \text{\emph{Multiplication}:}  &
							\begin{tikzpicture}[baseline=-0.1cm,scale=.7, transform shape]
								\draw[gray,dashed] (0,0) circle (1cm); \node[draw] at (-0.4,0.4) {$f$};\node[draw] at (0.4,-0.4) {$g$};
							\end{tikzpicture}
							&=
							\begin{tikzpicture}[baseline=-0.1cm,scale=.7, transform shape]
								\draw[gray,dashed] (0,0) circle (1cm); \node[draw] at (0,0) {$fg$};
							\end{tikzpicture}
						\end{align}					
						This gives morphism spaces the structure of $R$-bimodules, by acting on the leftmost or the 
						rightmost region. 
						
						The other polynomial relations are:
						\begin{align}
							& \text{\emph{Barbell} relation:}  &
							\begin{tikzpicture}[baseline=-0.1cm,scale=.7, transform shape] \label{barbell}
								\draw[gray,dashed] (0,0) circle (1cm); \draw[red] (0,.5)--(0,-.5); \fill[red] (0,.5) circle (2pt);\fill[red] (0,-.5) circle (2pt);
							\end{tikzpicture}
							&=
							\begin{tikzpicture}[baseline=-0.1cm,scale=.7, transform shape]
								\draw[gray,dashed] (0,0) circle (1cm); \node[draw] at (0,0) {$\alpha_{\re{s}}$};
							\end{tikzpicture}			\\
							& \text{\emph{Sliding} relation:} & 
							\begin{tikzpicture}[baseline=-0.1cm,scale=.7, transform shape]\label{sliding}
								\draw[gray,dashed] (0,0) circle (1cm); \draw[red] (0,1)--(0,-1); 
								\node[draw] at (.5,0) {$f$};
							\end{tikzpicture}			
							&=
							\begin{tikzpicture}[baseline=-0.1cm,scale=.7, transform shape]
								\draw[gray,dashed] (0,0) circle (1cm); \draw[red] (0,1)--(0,-1); 
								\node[draw, inner sep =.08cm] at (-.5,0) {$\re{s}(f)$};
							\end{tikzpicture}			
							+
							\begin{tikzpicture}[baseline=-0.1cm,scale=.7, transform shape]
								\draw[gray,dashed] (0,0) circle (1cm); \draw[red] (0,1)--(0,.5); \fill[red] (0,.5) circle (2pt);
								\draw[red] (0,-1)--(0,-.5); \fill[red] (0,-.5) circle (2pt); 
								\node[draw] at (0,0) {$\partial_{\re{s}}(f)$};
							\end{tikzpicture}
							\end{align}
					\item[One color relations.] These are the following:
						\begin{align}
							& \text{\emph{Frobenius associativity}:} &
							\begin{tikzpicture}[baseline=-0.1cm,color=red,scale=sqrt(2)/3,scale=.7, transform shape]\label{associativity}
								\draw (-1.5,-1.5)--(-0.6,0)--(0.6,0)--(1.5,-1.5);
								\draw (-1.5,1.5)--(-0.6,0); \draw(0.6,0)--(1.5,1.5);
								\draw[dashed,gray] (0,0) circle ({3/sqrt(2)});
							\end{tikzpicture}
							&=
							\begin{tikzpicture}[baseline=-0.1cm,color=red,scale=sqrt(2)/3,rotate=90,scale=.7, transform shape]
								\draw (-1.5,-1.5)--(-0.6,0)--(0.6,0)--(1.5,-1.5);
								\draw (-1.5,1.5)--(-0.6,0); \draw(0.6,0)--(1.5,1.5);
								\draw[dashed,gray] (0,0) circle ({3/sqrt(2)});
							\end{tikzpicture} \\
							& \text{\emph{Frobenius unit}:} &
							\begin{tikzpicture}[color=red,baseline=-0.1cm,scale=.7, transform shape]\label{dotline}
								\draw[gray,dashed] (0,0) circle (1cm);
								\draw (0,-1) -- (0,1); \draw (0,0)--(.3,0);\fill (0.3,0) circle (2pt);
							\end{tikzpicture}
							&=
							\begin{tikzpicture}[baseline=-0.1cm,scale=.7, transform shape]
								\draw[gray,dashed] (0,0) circle (1cm);						
								\draw[red] (0,-1) -- (0,1);
							\end{tikzpicture}
							\\
							&\text{\emph{Needle} relation:} &
							\begin{tikzpicture}[baseline=0,scale=.7, transform shape]\label{eq_needle}
								\draw[gray,dashed] (0,0) circle (1cm);
								\draw[red] (0,-1)--(0,-0.3) arc (-90:270:0.3cm);
							\end{tikzpicture}
							&= 0
						\end{align}
					\item[Two color relations.] These allow us to move dots,
						or trivalent vertices, across $(s,t)$-ars.
						We give the two versions, according to the parity of $m_{st}$:
						\begin{equation}\label{2mtriv}
							\begin{tikzpicture}[baseline=0,scale=.8,transform shape]
								\draw[dashed,gray] (0,0) circle (1cm);
								\coordinate (a) at (-.3,0); \coordinate (b) at (0.5,0);
								\draw[red] (180:1cm)--(a); \draw[blue] ({180-360/11}:1cm)--(a);
								\draw[red] ({180-2*360/11}:1cm)--(a); \draw[blue] ({180-5*360/11}:1cm)--(b);
								\draw[red] ({180-4*360/11}:1cm)--(a);
								\draw[blue] ({180+360/11}:1cm)--(a);
								\draw[red] ({180+2*360/11}:1cm)--(a); \draw[blue] ({180+5*360/11}:1cm)--(b);
								\draw[red] ({180+4*360/11}:1cm)--(a);
								\draw[blue](a)--(b);
								\node at (-0.05,.5) {...};\node at (-0.05,-.5) {...};
							\end{tikzpicture}
							=
							\begin{tikzpicture}[baseline=0,scale=.8,transform shape]
								\draw[dashed,gray] (0,0) circle (1cm);
								\coordinate (a) at (-.5,0); \coordinate (b) at (0.15,0.35); \coordinate (c) at (0.15,-0.35);
								\draw[red] (180:1cm)--(a); \draw[blue] ({180-360/11}:1cm)--(b);
								\draw[red] ({180-2*360/11}:1cm)--(b); \draw[blue] ({180-5*360/11}:1cm)--(b);
								\draw[red] ({180-4*360/11}:1cm)--(b);
								\draw[blue] ({180+360/11}:1cm)--(c);
								\draw[red] ({180+2*360/11}:1cm)--(c); \draw[blue] ({180+5*360/11}:1cm)--(c);
								\draw[red] ({180+4*360/11}:1cm)--(c);
								\draw[blue](c)..controls (-0.15,0)..(b); \draw[red] (b) -- (c);\draw[red] (b)--(a)--(c);
								\draw[red] (b) ..controls (0.7,0).. (c);
								\node at (.2,.6) {...};\node at (.2,-.6) {...};\node at (.35,0) {...};
							\end{tikzpicture}\quad\text{or}\quad
							\begin{tikzpicture}[baseline=0,scale=.8,transform shape]
								\draw[dashed,gray] (0,0) circle (1cm);
								\coordinate (a) at (-.3,0); \coordinate (b) at (0.5,0);
								\draw[red] (180:1cm)--(a); \draw[blue] ({180-360/11}:1cm)--(a);
								\draw[red] ({180-2*360/11}:1cm)--(a); \draw[red] ({180-5*360/11}:1cm)--(b);
								\draw[blue] ({180-4*360/11}:1cm)--(a);
								\draw[blue] ({180+360/11}:1cm)--(a);
								\draw[red] ({180+2*360/11}:1cm)--(a); \draw[red] ({180+5*360/11}:1cm)--(b);
								\draw[blue] ({180+4*360/11}:1cm)--(a);
								\draw[red](a)--(b);
								\node at (-0.05,.5) {...};\node at (-0.05,-.5) {...};
							\end{tikzpicture}
							=
							\begin{tikzpicture}[baseline=0,scale=.8,transform shape]
								\draw[dashed,gray] (0,0) circle (1cm);
								\coordinate (a) at (-.5,0); \coordinate (b) at (0.15,0.35); \coordinate (c) at (0.15,-0.35);
								\draw[red] (180:1cm)--(a); \draw[blue] ({180-360/11}:1cm)--(b);
								\draw[red] ({180-2*360/11}:1cm)--(b); \draw[red] ({180-5*360/11}:1cm)--(b);
								\draw[blue] ({180-4*360/11}:1cm)--(b);
								\draw[blue] ({180+360/11}:1cm)--(c);
								\draw[red] ({180+2*360/11}:1cm)--(c); \draw[red] ({180+5*360/11}:1cm)--(c);
								\draw[blue] ({180+4*360/11}:1cm)--(c);
								\draw[blue](c)..controls (-0.15,0)..(b); \draw[red] (b) -- (c);\draw[red] (b)--(a)--(c);
								\draw[blue] (b) ..controls (0.7,0).. (c);
								\node at (.2,.6) {...};\node at (.2,-.6) {...};\node at (.35,0) {...};
							\end{tikzpicture}
						\end{equation}
						and
						\begin{equation}\label{eq_2mdot}
							\begin{tikzpicture}[baseline=0,scale=.8,transform shape]
								\draw[dashed,gray] (0,0) circle (1cm);
								\coordinate (a) at (-.2,0); \coordinate (b) at (0.5,0);
								\draw[red] (180:1cm)--(a); \draw[blue] ({180-360/11}:1cm)--(a);
								\draw[blue] ({180-3*360/11}:1cm)--(a); 
								\draw[red] ({180-4*360/11}:1cm)--(a);
								\draw[blue] ({180+360/11}:1cm)--(a);
								\draw[blue] ({180+3*360/11}:1cm)--(a); \fill[blue] (b) circle (2pt);
								\draw[red] ({180+4*360/11}:1cm)--(a);
								\draw[blue](a)--(b);
								\node at (-0.4,.5) {...};\node at (-0.4,-.5) {...};
							\end{tikzpicture}
							=
							\begin{tikzpicture}[baseline=0,scale=.8,transform shape]
								\draw[dashed,gray] (0,0) circle (1cm);
								\coordinate (a) at (-.2,0); \coordinate (b) at (0.5,0);
								\draw[red] (180:1cm)--(a); \draw[blue] ({180-360/11}:1cm)--(a);
								\draw[blue] ({180-3*360/11}:1cm)--(a); 
								\draw[blue] ({180+360/11}:1cm)--(a);
								\draw[blue] ({180+3*360/11}:1cm)--(a);
								\draw[red] ({180+4*360/11}:1cm) -- (b) -- ({180-4*360/11}:1cm); \draw[red](a)--(b);
								\node at (-0.4,.6) {...};\node at (-0.4,-.6) {...};
								\node[draw,circle, inner sep=.1cm,fill=white] at (a) {$\JW$};
							\end{tikzpicture}
							\quad\text{or}\quad
							\begin{tikzpicture}[baseline=0,scale=.8,transform shape]
								\draw[dashed,gray] (0,0) circle (1cm);
								\coordinate (a) at (-.2,0); \coordinate (b) at (0.5,0);
								\draw[red] (180:1cm)--(a); \draw[blue] ({180-360/11}:1cm)--(a);
								\draw[red] ({180-3*360/11}:1cm)--(a); 
								\draw[blue] ({180-4*360/11}:1cm)--(a);
								\draw[blue] ({180+360/11}:1cm)--(a);
								\draw[red] ({180+3*360/11}:1cm)--(a); \fill[red] (b) circle (2pt);
								\draw[blue] ({180+4*360/11}:1cm)--(a);
								\draw[red](a)--(b);
								\node at (-0.4,.5) {...};\node at (-0.4,-.5) {...};
							\end{tikzpicture}
							=
							\begin{tikzpicture}[baseline=0,scale=.8,transform shape]
								\draw[dashed,gray] (0,0) circle (1cm);
								\coordinate (a) at (-.2,0); \coordinate (b) at (0.5,0);
								\draw[red] (180:1cm)--(a); \draw[blue] ({180-360/11}:1cm)--(a);
								\draw[red] ({180-3*360/11}:1cm)--(a); 
								\draw[blue] ({180+360/11}:1cm)--(a);
								\draw[red] ({180+3*360/11}:1cm)--(a);
								\draw[blue] ({180+4*360/11}:1cm) -- (b) -- ({180-4*360/11}:1cm); \draw[blue](a)--(b);
								\node at (-0.4,.6) {...};\node at (-0.4,-.6) {...};
								\node[draw,circle, inner sep=.1cm,fill=white] at (a) {$\JW$};
							\end{tikzpicture}
						\end{equation}
						where the circles labeled $\JW$ are the \emph{Jones-Wenzel morphisms}. 
						These are certain $\kk$-linear combinations of diagrams (with circular boundary
						and $2m_{st}-2$ boundary points around it)
						that can be described in terms of the 2-colored Temperley-Lieb category\footnote{%
							More precisely, one can obtain them by deformation retraction from a colored 
							version of certain idempotents in the Temperley-Lieb algebra, which correspond
							to the highest weight irreducible summands inside tensor powers of the standard representation
							of $U_q(\mathfrak{sl}_2)$.}.
						
						Here are the examples for the first few values of $m_{st}$:
						\begin{align*}
							&m_{st}=2& %
							&\begin{tikzpicture}[baseline=0]
								\def\radius{.7cm}
								\draw[dashed,gray] (0,0) circle (\radius);
								\node[draw,circle, inner sep=.1cm,fill=white] (a) at (0,0) {$\JW$};
								\draw[red] (0:-\radius)--(a);
								\draw[blue] (0:\radius)--(a);
							\end{tikzpicture}
							=
							\begin{tikzpicture}[baseline=0]
								\def\radius{.7cm}
								\draw[dashed,gray] (0,0) circle (\radius);
								\draw[red] (0:-\radius)--(0:{-\radius/3});\fill[red] (0:{-\radius/3}) circle (1.5pt);
								\draw[blue] (0:\radius)--(0:{\radius/3});\fill[blue] (0:{\radius/3}) circle (1.5pt);
							\end{tikzpicture}
							\\
							&m_{st}=3& %
							&\begin{tikzpicture}[baseline=0]
								\def\radius{.7cm}
								\draw[dashed,gray] (0,0) circle (\radius);
								\node[draw,circle, inner sep=.1cm,fill=white] (a) at (0,0) {$\JW$};
								\foreach \i in {0,180}{%
									\draw[red] (\i:\radius)--(a);
									}
								\foreach \i in {90,270}{%
									\draw[blue] (\i:\radius)--(a);
									}									
							\end{tikzpicture}
							=
							\begin{tikzpicture}[baseline=0]
								\def\radius{.7cm}
								\draw[dashed,gray] (0,0) circle (\radius);
								\foreach \i in {0,180}{%
									\draw[red] (\i:\radius)--(\i:{\radius/3});%
										\fill[red] (\i:{\radius/3}) circle (1.5pt);
									}
								\draw[blue] (0,\radius)--(0,-\radius);
							\end{tikzpicture}
							+
							\begin{tikzpicture}[baseline=0]
								\def\radius{.7cm}
								\draw[dashed,gray] (0,0) circle (\radius);
								\draw[red] (\radius,0)--(-\radius,0);
								\foreach \i in {90,270}{%
									\draw[blue] (\i:\radius)--(\i:{\radius/3});%
										\fill[blue] (\i:{\radius/3}) circle (1.5pt);
									}									
							\end{tikzpicture}
							\\
							&m_{st}=4\qquad& %
							&\begin{multlined}[t]
								\begin{tikzpicture}[baseline=0]
									\def\radius{.7cm}
									\draw[dashed,gray] (0,0) circle (\radius);
									\node[draw,circle, inner sep=.1cm,fill=white] (a) at (0,0) {$\JW$};
									\foreach \i in {0,120,240}{%
										\draw[blue] (\i:\radius)--(a);
										}
									\foreach \i in {60,180,300}{%
										\draw[red] (\i:\radius)--(a);
										}									
								\end{tikzpicture}
								=
								\begin{tikzpicture}[baseline=0]
									\def\radius{.7cm}
									\draw[dashed,gray] (0,0) circle (\radius);
									\draw[red] (-\radius,0)--({-\radius/2},0);\fill[red] ({-\radius/2},0) circle (1.5pt);
									\draw[blue] (\radius,0)--({\radius/2},0);\fill[blue] ({\radius/2},0) circle (1.5pt);
									\draw[blue] (120:\radius) ..controls (120:{\radius/8}) and (-120:{\radius/8}).. (-120:\radius);
									\draw[red] (60:\radius) ..controls (60:{\radius/8}) and (-60:{\radius/8}).. (-60:\radius);
								\end{tikzpicture}
								+
								\begin{tikzpicture}[baseline=0,rotate=120]
									\def\radius{.7cm}
									\draw[dashed,gray] (0,0) circle (\radius);
									\draw[red] (-\radius,0)--({-\radius/2},0);\fill[red] ({-\radius/2},0) circle (1.5pt);
									\draw[blue] (\radius,0)--({\radius/2},0);\fill[blue] ({\radius/2},0) circle (1.5pt);
									\draw[blue] (120:\radius) ..controls (120:{\radius/8}) and (-120:{\radius/8}).. (-120:\radius);
									\draw[red] (60:\radius) ..controls (60:{\radius/8}) and (-60:{\radius/8}).. (-60:\radius);
								\end{tikzpicture}
								+
								\begin{tikzpicture}[baseline=0,rotate=240]
									\def\radius{.7cm}
									\draw[dashed,gray] (0,0) circle (\radius);
									\draw[red] (-\radius,0)--({-\radius/2},0);\fill[red] ({-\radius/2},0) circle (1.5pt);
									\draw[blue] (\radius,0)--({\radius/2},0);\fill[blue] ({\radius/2},0) circle (1.5pt);
									\draw[blue] (120:\radius) ..controls (120:{\radius/8}) and (-120:{\radius/8}).. (-120:\radius);
									\draw[red] (60:\radius) ..controls (60:{\radius/8}) and (-60:{\radius/8}).. (-60:\radius);
								\end{tikzpicture}
								+\\+\qn{2}{\bl{t}}\,
								\begin{tikzpicture}[baseline=0]
									\def\radius{.7cm}
									\draw[dashed,gray] (0,0) circle (\radius);
									\foreach \i in {60,180,300}{%
										\draw[red] (\i:\radius)--(0,0);
										}
									\foreach \i in {0,120,240}{%
										\draw[blue] (\i:\radius)--(\i:{\radius/2});%
											\fill[blue](\i:{\radius/2})circle(1.5pt);
										}									
								\end{tikzpicture}
								+\qn{2}{\re{s}}\,
								\begin{tikzpicture}[baseline=0]
									\def\radius{.7cm}
									\draw[dashed,gray] (0,0) circle (\radius);
									\foreach \i in {60,180,300}{%
										\draw[red] (\i:\radius)--(\i:{\radius/2});%
											\fill[red](\i:{\radius/2})circle(1.5pt);
										}
									\foreach \i in {0,120,240}{%
										\draw[blue] (\i:\radius)--(0,0);
										}									
								\end{tikzpicture}
							\end{multlined}
						\end{align*}
						Notice that these are all the Jones-Wenzl morphisms that we need to handle with 
						Coxeter groups of types $A$, $B$, $D$, $E$, $F$.
						In these cases, Relation \eqref{eq_2mdot} becomes:
						\begin{align*}
							&m_{st}=2\qquad&
							&\begin{tikzpicture}[baseline=0]
								\def\radius{.7cm}
								\draw[dashed,gray] (0,0) circle (\radius);
								\draw[red] (-\radius,0)--({\radius/2},0);%
									\fill[red] ({\radius/2},0) circle (1.5pt);
								\draw[blue] (0,-\radius)--(0,\radius);
							\end{tikzpicture}
							=
							\begin{tikzpicture}[baseline=0]
								\def\radius{.7cm}
								\draw[dashed,gray] (0,0) circle (\radius);
								\draw[red] (-\radius,0)--({-\radius/2},0);%
									\fill[red] ({-\radius/2},0) circle (1.5pt);
								\draw[blue] (0,-\radius)--(0,\radius);
							\end{tikzpicture}
							\\
							&m_{st}=3&
							&\begin{tikzpicture}[baseline=0]
								\def\radius{.7cm}
								\draw[dashed,gray] (0,0) circle (\radius);
								\coordinate (a) at ({-.1*\radius},0);
								\foreach \i in {70,180,-70}{%
									\draw[red] (a)--(\i:\radius);								
								}
								\foreach \i in {125,-125}{%
									\draw[blue] (a)--(\i:\radius);								
								}
								\coordinate (b) at ({\radius/2},0);
								\draw[blue] (a)--(b);\fill[blue] (b) circle (1.5pt);
							\end{tikzpicture}
							=
							\begin{tikzpicture}[baseline=0]
								\def\radius{.7cm}
								\draw[dashed,gray] (0,0) circle (\radius);
								\coordinate (a) at ({-.1*\radius},0);
								\draw[red] (70:\radius) ..controls (70:{\radius/3}) and (-70:{\radius/3}).. (-70:\radius);
								\draw[blue] (125:\radius) ..controls (125:{\radius/8}) and (-125:{\radius/8}).. (-125:\radius);
								\coordinate (b) at ({-\radius/2},0);
								\draw[red] ({-\radius},0)--(b);\fill[red] (b) circle (1.5pt);
							\end{tikzpicture}
							+
							\begin{tikzpicture}[baseline=0]
								\def\radius{.7cm}
								\draw[dashed,gray] (0,0) circle (\radius);
								\coordinate (a) at (0,0);
								\foreach \i in {70,180,-70}{%
									\draw[red] (a)--(\i:\radius);								
								}
								\foreach \i in {125,-125}{%
									\draw[blue] (\i:\radius)--(\i:{\radius/2});\fill[blue] (\i:{\radius/2}) circle (1.5pt);	
								}
							\end{tikzpicture}\\
							&m_{st}=4&
							&\begin{multlined}[t]
								\begin{tikzpicture}[baseline=0]
									\def\radius{.7cm}
									\draw[dashed,gray] (0,0) circle (\radius);
									\coordinate (a) at (0,0);
									\foreach \i in {90,180,-90}{%
										\draw[red] (a)--(\i:\radius);								
									}
									\coordinate (b) at ({\radius/2},0);
									\draw[red] (a)--(b);\fill[red] (b) circle (1.5pt);
									\foreach \i in {45,135,-45,-135}{%
										\draw[blue] (a)--(\i:\radius);								
									}
								\end{tikzpicture}
								=
								\begin{tikzpicture}[baseline=0]
									\def\radius{.7cm}
									\draw[dashed,gray] (0,0) circle (\radius);
									\draw[red] (0,\radius)--(0,-\radius);
									\draw[blue] (135:\radius) ..controls ({-\radius/5},{\radius/5}) and ({-\radius/5},{-\radius/5})..(-135:\radius);
									\draw[red] (-\radius,0)--({-.6*\radius},0);\fill[red] ({-.6*\radius},0) circle (1.5pt);
									\draw[blue] (45:\radius) ..controls ({\radius/3},{\radius/3}) and ({\radius/3},{-\radius/3})..(-45:\radius);
								\end{tikzpicture}
								+
								\begin{tikzpicture}[baseline=0]
									\def\radius{.7cm}
									\draw[dashed,gray] (0,0) circle (\radius);
									\draw[red] (-\radius,0) ..controls (-{\radius/4},0) and (0,{\radius/4}).. (0,\radius);
									\draw[blue] (135:\radius)--(135:{.6*\radius}); \fill[blue] (135:{.6*\radius}) circle (1.5pt);
									\draw[blue] (45:\radius)--(-135:\radius);\draw[blue] (-45:\radius)--(0,0);
									\draw[red] (0,{-\radius})--(0,{-.5*\radius});\fill[red](0,{-.5*\radius}) circle (1.5pt);
								\end{tikzpicture}
								+
								\begin{tikzpicture}[baseline=0]
									\def\radius{.7cm}
									\draw[dashed,gray] (0,0) circle (\radius);
									\draw[red] (-\radius,0) ..controls ({-\radius/4},0) and (0,{-\radius/4}).. (0,-\radius);
									\draw[blue] (-135:\radius)--(-135:{.6*\radius}); \fill[blue] (-135:{.6*\radius}) circle (1.5pt);
									\draw[blue] (-45:\radius)--(135:\radius);\draw[blue] (45:\radius)--(0,0);
									\draw[red] (0,{\radius})--(0,{.5*\radius});\fill[red](0,{.5*\radius}) circle (1.5pt);
								\end{tikzpicture}
								+ \\
								+\qn{2}{\bl{t}}\,
								\begin{tikzpicture}[baseline=0]
									\def\radius{.7cm}
									\draw[dashed,gray] (0,0) circle (\radius);
									\draw[red] (0,\radius)--(0,-\radius);
									\draw[red] ({-\radius},0)--(0,0);
									\foreach \i in {135,-135}{%
										\draw[blue] (\i:\radius)--(\i:{\radius/2});%
											\fill[blue](\i:{\radius/2})circle(1.5pt);
										}						
									\draw[blue] (45:\radius) ..controls (45:{\radius/2}) and (-45:{\radius/2}).. (-45:\radius);
								\end{tikzpicture}
								+\qn{2}{\re{s}}\,
								\begin{tikzpicture}[baseline=0]
									\def\radius{.7cm}
									\draw[dashed,gray] (0,0) circle (\radius);
									\foreach \i in {90,270}{%
										\draw[red] (\i:\radius)--(\i:{\radius/3});%
											\fill[red](\i:{\radius/3})circle(1.5pt);
										}
									\draw[red] (180:\radius)--(180:{.6*\radius});%
										\fill[red](180:{.6*\radius})circle(1.5pt);
									\draw[blue] (-135:\radius)--({-\radius/4},0)--({\radius/4},0)--(45:\radius);
									\draw[blue] (135:\radius)--({-\radius/4},0);
									\draw[blue] (-45:\radius)--({\radius/4},0);									
								\end{tikzpicture}
							\end{multlined}
						\end{align*}	
						For further details, we refer the reader to 
						\cite[\S\,5.2]{EW}, \cite{Elias_two} or \cite[\S\,8]{EMTW}.
					\item[Three color relations.] For each finite parabolic subgroup $W_I$
						of rank 3, there is a relation ensuring compatibility between the three 
						corresponding $2m$-valent vertices. We will not need these relations so we refer the reader to
						\cite[\S\ 1.4.3]{EW}
%
						
				\end{description}				
				Notice that all the relations are homogeneous, so the morphism spaces are \emph{graded}
				$R$-bimodules.
		\end{enumerate}
		This completes the definition of $\DBS$.
		Now we can finally give the main definition of this section.
		\begin{dfn}
			The \emph{diagrammatic Hecke category} $\D$ is the Karoubi envelope of the closure of $\DBS$
			by shifts and direct sums.
		\end{dfn}
		Let $(\cdot)$ denote the shift in the polynomial grading, so if $B$ is an object of $\D$, then
		$B(1)$ is such that $\Hom_i(-,B(1))=\Hom_{i+1}(-,B)$ where $\Hom_i$ denotes the degree $i$ piece of 
		the morphism space. Let $[\cdot]_{\oplus}$ denote the 
		operation of taking the split Grothendieck group of an additive category. 
		Notice that $[\D]_{\oplus}$ is naturally a $\Laur$-algebra: the ring structure is induced by 
		the tensor product and the action of $v$ corresponds to the shift (more precisely $v[B]:=[B(1)]$).
		Now we are ready to state the following \emph{categorification} result (see \cite[\S\,6.6]{EW}).
		\begin{thm}\label{thm_soergcat}
			If $\kk$ is a complete local ring, then $\D$ is Krull-Schmidt and
			there is a unique isomorphism of $\Laur$-algebras, called \emph{character},
			\begin{equation}\label{eq_Grothgroup}
				\mathrm{ch}:[\D]_\oplus \overset{\sim}{\rightarrow}\HH_{(W,S)}					
			\end{equation}
			sending the class of $B_s$ to $b_s$.
		\end{thm}
		\begin{rmk}
			Under the same assumptions of the theorem one can also classify the indecomposable objects and 
			describe graded ranks of the morphism spaces in terms of the above isomorphism.
		\end{rmk}
		\begin{rmk}
			As mentioned in the introduction, the category $\D$ can be seen as a diagrammatic description
			of other categories. It is in fact equivalent to the category of \emph{Soergel bimodules} 
			(see \cite[Theorem 6.30]{EW}), under certain assumptions that can 
			be weakened if one considers Abe's version of the latter category \cite{Abe}. 
			Among other variants, we mention the following geometric one.
			In the case $W$ is a Weyl group (or, more generally, crystallographic) 
			Riche and Williamson \cite{RW} proved that
			the diagrammatic category is equivalent to the category of equivariant parity 
			sheaves on the (affine) flag variety, in the sense of \cite{JMW}, 
			for an appropriate choice of realization.
		\end{rmk}
		\begin{rmk}\label{rmk_usesofdiag}
			Here are some applications of the Hecke category.
			It was used by Williamson \cite{Wil_tors} to give counterexamples for the expected bound in Lusztig's 
			conjecture; Riche and Williamson \cite{RW} used it to construct the action of the Hecke 
			category over the principal block of the category of rational representations
			of $\mathrm{GL}_n$ (this action was then established for all types by Ciappara \cite{Ciap} and independently by Bezrukavnikov and Riche \cite{BeRi});
			it was also used to explicitly compute the $p$-canonical basis for Weyl groups of small ranks: 
			see the paper by Jensen and Williamson \cite{JenWil}.
%
	\end{rmk}
		\subsection{Localization}\label{subs_localiz}
			Let $Q$ be the field of fractions of $R$. 
			Recall from \cite[\S\ 5.4]{EW} that the Karoubi envelope $\D_Q$ of the category obtained from
			$\DBS$ by extension of scalars from $R$ to $Q$ can be described diagrammatically.
			One considers an additional generating object $Q_s$ for each $s\in S$. For $\ux$ a Coxeter 
			word, let $Q_{\ux}$ denote the corresponding product of generators.
			Morphisms are still described by diagrams, with an additional kind of strand, corresponding 
			to the new generators $Q_s$, that we draw as a squiggly line of the corresponding color.
			Then we add to the generating vertices, new bivalent vertices, 
			as well as squiggly caps and cups, and squiggly $(s,t)$-ars, as follows:
			\begin{center}
				\begin{tikzpicture}
					\draw[red] (0,0)--(0,1);\draw[red,localized] (0,1)--(0,2);
					\draw[red,localized] (1.5,0)--(1.5,1);\draw[red] (1.5,1)--(1.5,2);
					\draw[red,localized] (3,0.5) arc (180:0:1);
					\draw[red,localized] (6,1.5) arc (180:360:1);
					\begin{scope}[xshift=10cm,yshift=1cm,rotate=-15]
						\draw[red,localized] (0,1)--(0,0); %
						\draw[red,localized] (-30:1cm)--(150:1cm);%
								\draw[red,localized](-150:1cm)--(30:1cm);
						\draw[blue,rotate=30,localized] (0,1)--(0,0); %
							\draw[blue,rotate=30,localized] (-30:1cm)--(150:1cm);%
								\draw[blue,rotate=30,localized](-150:1cm)--(30:1cm);
						\foreach \i in {-60,-75,-90}{%
								\fill (\i:.5cm) circle (.5pt);
							}
					\end{scope}
				\end{tikzpicture}
			\end{center}
			We will not give details for the new relations that one needs to impose,
			we only recall that certain isotopy relations involving the new strands are twisted by a 
			sign (the category is not cyclic any more). We will only need the following equalities, see 
			\cite[(5.23)]{EW}:
			\begin{align}\label{eq_localrel}
				&\begin{tikzpicture}[scale=0.7,baseline=0.7cm]
					\draw[red,localized] (0,0)--(0,1);\draw[red] (0,1)--(0,2);					
					\fill[red] (0,2) circle (3pt);
				\end{tikzpicture}
				\,=0&
				&\begin{tikzpicture}[scale=0.7,baseline=0.7cm]
					\draw[red] (0,0)--(0,1);\draw[red,localized] (0,1)--(0,2);
					\fill[red] (0,0) circle (3pt);
				\end{tikzpicture}
				\,=0&
				&
				\begin{tikzpicture}[scale=0.7,baseline=0.7cm]
					\draw[red,localized] (0,0)--(0,2);
					\node at (.5,1.1) {$f$};
				\end{tikzpicture}
				=
				\begin{tikzpicture}[scale=0.7,baseline=0.7cm]
					\node at (-.7,1.1) {$\re{s}(f)$};
					\draw[red,localized] (0,0)--(0,2);
				\end{tikzpicture}
			\end{align}
			The first two are linked with the fact that that in $\D_Q$ the object $B_s$ splits into $\un \oplus Q_s$ (the category $\D_Q$ is not graded any more), 
			and inclusions and projections are given by:
			\begin{align*}
				&\begin{tikzpicture}[baseline=0]
					\node (un) at (0,0) {$\un$};\node (Bs) at (2,0) {$B_{\re{s}}$};
					\draw[-latex] ([yshift=1mm]un.east)--([yshift=1mm]Bs.west);
					\draw[latex-] ([yshift=-1mm]un.east)--([yshift=-1mm]Bs.west);
					\draw[red] (1,.3)--(1,.8);\fill[red] (1,.3) circle (1.5pt);
					\draw[red] (1,-.3)--(1,-.8);\fill[red] (1,-.3) circle (1.5pt);
					\node[anchor=east] at (1,-.5) {$\frac{1}{\alpha_{\re{s}}}$};
				\end{tikzpicture},&
				&\begin{tikzpicture}[baseline=0]
					\node (Qs) at (0,0) {$Q_{\re{s}}$};\node (Bs) at (2,0) {$B_{\re{s}}$};
					\draw[-latex] ([yshift=1mm]Qs.east)--([yshift=1mm]Bs.west);				
					\draw[latex-] ([yshift=-1mm]Qs.east)--([yshift=-1mm]Bs.west);				
					\draw[red,localized] (1,.25)--(1,.7);\draw[red] (1,.7)--(1,1.05);
					\draw[red,localized] (1,-.25)--(1,-.7);\draw[red] (1,-.7)--(1,-1.05);
					\node[anchor=east] at (1,-.65) {$\frac{1}{\alpha_{\re{s}}}$};
				\end{tikzpicture}.
			\end{align*}
			If $\bof=(f_1,\dots,f_k)\in \{0,1\}^{\ell(\uw)}$ is a subexpression, let $Q_{\bof}$ denote the product
			$Q_{s_1}^{f_1}\dots Q_{s_k}^{f_k}$, where $Q_{s_i}^0:=\un$. Then it is easy to see that:
			\[
				B_{\uw}=\bigoplus_{\bof} Q_{\bof},
			\]
			and that the inclusion and projection of each factor is given by an appropriate combination 
			of the above morphisms.
			
			Then, given any $\phi\in \Hom_{\D}(B_{\uw},B_{\uw'})$, and 
			subexpressions $\bof$ of $\uw$ and $\bof'$ of $\uw'$,
			by composing with the the corresponding inclusion and projection, one can 
			compute each component:
			\[
				\phi^{\bof,\bof'}\in \Hom_{\D_Q}(Q_{\bof},Q_{\bof'})\cong 
					\begin{cases}
						Q & \text{if $\uw^{\bof}=(\uw')^{\bof'}$},\\
						0 &	\text{otherwise.}
					\end{cases}
			\]
			For the isomorphism, see \cite[\S\ 3.6]{EW}.
		\subsection{Light leaves and double leaves}\label{subs_lightdouble}
		Let $\uw$ be a Coxeter word and $\ux$ a subword which is reduced.
		One can construct, in the space $\Hom_{\D}(B_{\uw},B_{\ux})$, some special morphisms called the
		\textit{light leaves maps}, that were introduced by Libedinsky in \cite{Lib},
		and described diagrammatically in \cite{EW}.
		Then, one can use these maps to construct an explicit basis of the general morphism space 
		$\Hom_{\D}(B_{\uw_1},B_{\uw_2})$ as a left $R$-module (equivalently, one could consider the 
		right $R$-module structure).
		Let us recall this construction more precisely.

		Let $\ux$ be a reduced word for the element $x$ in $W$. 
		The light leaves maps in $\Hom(B_{\uw},B_{\ux})$ 
		are labeled by 01-sequences expressing $x$ in the following way.
	
		Consider a subexpression $\boe$ of $\uw$ expressing $x$. 
		One can \emph{decorate} $\boe$ according to the corresponding 
		Bruhat stroll (as defined in \S\ \ref{subs_CoxbraHeck}).
		The idea is that, at each step $i$, we consider the position where the next simple 
		reflection \textit{would} bring us,	regardless of the fact that this reflection will actually 
		be taken or not,
		and we decorate the corresponding symbol in $\boe$ in the following way. 
		\begin{itemize}
			\item if $x_{i}s_{i+1}>x_i$ then we decorate $e_{i+1}$ with a $U$ (for ``up'': taking the reflection $s_{i+1}$ \emph{would} bring us up
				in the Bruhat graph);
			\item if $x_{i}s_{i+1}<x_{i}$ then we decorate it with a $D$ (for ``down'').
		\end{itemize}
		We can then construct a morphism $L_{\uw,\boe}:B_{\uw}\rightarrow B_{\ux}$ from this data. 
		This will actually depend on some choices.
		\begin{dfn*}
			First, choose, for each $i<\ell(\uw)$, a reduced word $\ux_i$ for $x_i$, and put $\ux_{\ell(\uw)}:=\ux$. We define 
			$L_{\uw,\boe}$ by induction. Let $\lambda_0$ to be $\id_{\un}$. 
			Then, suppose we have constructed a morphism 
			$\lambda_i:B_{\uw_{\le i}}\rightarrow B_{\ux_i}$, and consider the decorated value of $e_{i+1}$.
			\begin{enumerate}
				\item If $e_{i+1}=U0$, then $x_i=x_{i+1}$ and we choose a \emph{braid move} $\Phi$
				(i.e., a sequence of braid relations) from $\ux_i$ to $\ux_{i+1}$. 
				\item If $e_{i+1}=U1$, then $x_{i+1}=x_is_{i+1}$ and we choose a braid move $\Phi$
					from the reduced word $\ux_i s_{i+1}$ to $\ux_{i+1}$.
				\item If $e_{i+1}=D0$, then there is 
					a reduced expression for $x_{i}$ of the form $\uz s_{i+1}$ where $\uz$ is a reduced expression 
					for $z=x_{i}s_{i+1}$,
					and	we choose a braid move $\Psi$ from $\ux_i$ to $\uz s_{i+1}$.
					Furthermore, we choose another braid 
					move $\Phi$ from the latter to $\ux_{i+1}$.
				\item Finally, if $e_{i+1}=D1$, then we choose, as in the previous case, a braid move $\Psi$
					from $\ux_{i}$ to $\uz s_{i+1}$. This time $\uz$ is a reduced expression for 
					$x_{i+1}$ and we choose another braid move $\Phi$ from $\uz$ to $\ux_{i+1}$.
			\end{enumerate}
			Then $\lambda_{i+1}$ is defined by one of the following diagrams, according to the 
			decorated value of $e_{i+1}$ (here we denote $\beta(\Phi)$ and $\beta(\Psi)$ the compositions 
			of $(s,t)$-ars corresponding to the braid moves $\Phi$ and $\Psi$):
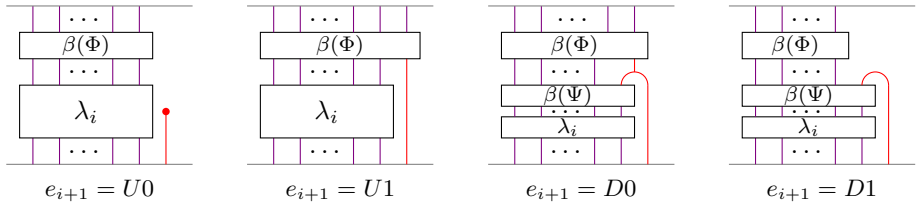
			\[
				\begin{tikzpicture}[scale=.7]
					\draw[gray] (-.5,0)--(3,0); 
					\draw[violet] (0,0)--(0,.5);\draw[violet] (0.5,0)--(0.5,.5);\draw[violet] (1.5,0)--(1.5,.5);\node at (1,.25) {\dots};\draw[violet] (2,0)--(2,.5);
					\draw (-.25,.5)--(2.25,.5)--(2.25,1.5)--(-.25,1.5)--cycle; \node at (1,1) {$\lambda_i$};
					\draw[violet] (0,1.5)--(0,2);\draw[violet] (0.5,1.5)--(0.5,2);\draw[violet] (1.5,1.5)--(1.5,2);\node at (1,1.75) {\dots};\draw[violet] (2,1.5)--(2,2);
					\draw[red] (2.5,0)--(2.5,1);\fill[red](2.5,1)circle (2pt);
					\draw (-.25,2)--(2.25,2)--(2.25,2.5)--(-.25,2.5)--cycle; \node at (1,2.25) {\footnotesize $\beta(\Phi)$};
					\draw[violet] (0,2.5)--(0,3);\draw[violet] (0.5,2.5)--(0.5,3);\draw[violet] (1.5,2.5)--(1.5,3);\node at (1,2.75) {\dots};\draw[violet] (2,2.5)--(2,3);
					\draw[gray] (-.5,3)--(3,3);
					\node[font=\small] at (1.25,-.5){$e_{i+1}=U0$};
				\end{tikzpicture}\qquad
				\begin{tikzpicture}[scale=.7]
					\draw[gray] (-.5,0)--(3,0); 
					\draw[violet] (0,0)--(0,.5);\draw[violet] (0.5,0)--(0.5,.5);\draw[violet] (1.5,0)--(1.5,.5);\node at (1,.25) {\dots};\draw[violet] (2,0)--(2,.5);
					\draw (-.25,.5)--(2.25,.5)--(2.25,1.5)--(-.25,1.5)--cycle; \node at (1,1) {$\lambda_i$};
					\draw[violet] (0,1.5)--(0,2);\draw[violet] (0.5,1.5)--(0.5,2);\draw[violet] (1.5,1.5)--(1.5,2);\node at (1,1.75) {\dots};\draw[violet] (2,1.5)--(2,2);
					\draw[red] (2.5,0)--(2.5,2);
					\draw (-.25,2)--(2.75,2)--(2.75,2.5)--(-.25,2.5)--cycle; \node at (1.25,2.25) {\footnotesize $\beta(\Phi)$};
					\draw[violet] (0,2.5)--(0,3);\draw[violet] (0.5,2.5)--(0.5,3);\draw[violet] (1.5,2.5)--(1.5,3);\node at (1,2.75) {\dots};\draw[violet] (2,2.5)--(2,3);\draw[violet] (2.5,2.5)--(2.5,3);
					\draw[gray] (-.5,3)--(3,3);
					\node[font=\small] at (1.25,-.5){$e_{i+1}=U1$};
				\end{tikzpicture}\qquad
				\begin{tikzpicture}[scale=.7]
					\draw[gray] (-.5,0)--(3,0); 
					\draw[violet] (0,0)--(0,.5);\draw[violet] (0.5,0)--(0.5,.5);\draw[violet] (1.5,0)--(1.5,.5);\node at (1,.25) {\dots};\draw[violet] (2,0)--(2,.5);
					\draw (-.25,.5)--(2.25,.5)--(2.25,.9)--(-.25,.9)--cycle; \node at (1,.7) {\footnotesize $\lambda_i$};
					\draw[violet] (0,.9)--(0,1.1);\draw[violet] (0.5,0.9)--(0.5,1.1);\draw[violet] (1.5,0.9)--(1.5,1.1);\draw[violet] (2,0.9)--(2,1.1);\node at (1,1) {\dots};
					\draw (-.25,1.1)--(2.25,1.1)--(2.25,1.5)--(-.25,1.5)--cycle; \node at (1,1.3) {\footnotesize $\beta(\Psi)$};
					\draw[violet] (0,1.5)--(0,2);\draw[violet] (0.5,1.5)--(0.5,2);\draw[red] (2,1.5) arc (180:0:.25cm);\node at (1,1.75) {\dots};\draw[violet] (1.5,1.5)--(1.5,2);
					\draw[red] (2.5,0)--(2.5,1.5); \draw[red] (2.25,2)--(2.25,1.75);
					\draw (-.25,2)--(2.5,2)--(2.5,2.5)--(-.25,2.5)--cycle; \node at (1.125,2.25) {\footnotesize $\beta(\Phi)$};
					\draw[violet] (0,2.5)--(0,3);\draw[violet] (0.5,2.5)--(0.5,3);\node at (1.125,2.75) {\dots};\draw[violet] (1.75,2.5)--(1.75,3);\draw[violet] (2.25,2.5)--(2.25,3);
					\draw[gray] (-.5,3)--(3,3);
					\node[font=\small] at (1.25,-.5){$e_{i+1}=D0$};
				\end{tikzpicture}\qquad
				\begin{tikzpicture}[scale=.7]
					\draw[gray] (-.5,0)--(3,0); 
					\draw[violet] (0,0)--(0,.5);\draw[violet] (0.5,0)--(0.5,.5);\draw[violet] (1.5,0)--(1.5,.5);\node at (1,.25) {\dots};\draw[violet] (2,0)--(2,.5);
					\draw (-.25,.5)--(2.25,.5)--(2.25,.9)--(-.25,.9)--cycle; \node at (1,.7) {\footnotesize $\lambda_i$};
					\draw[violet] (0,.9)--(0,1.1);\draw[violet] (0.5,0.9)--(0.5,1.1);\draw[violet] (1.5,0.9)--(1.5,1.1);\draw[violet] (2,0.9)--(2,1.1);\node at (1,1) {\dots};
					\draw (-.25,1.1)--(2.25,1.1)--(2.25,1.5)--(-.25,1.5)--cycle; \node at (1,1.3) {\footnotesize $\beta(\Psi)$};
					\draw[violet] (0,1.5)--(0,2);\draw[violet] (0.5,1.5)--(0.5,2);\draw[red] (2,1.5) arc (180:0:.25cm);\node at (1,1.75) {\dots};\draw[violet] (1.5,1.5)--(1.5,2);
					\draw[red] (2.5,0)--(2.5,1.5);
					\draw (-.25,2)--(1.75,2)--(1.75,2.5)--(-.25,2.5)--cycle; \node at (0.75,2.25) {\footnotesize $\beta(\Phi)$};
					\draw[violet] (0,2.5)--(0,3);\draw[violet] (0.5,2.5)--(0.5,3);\node at (1,2.75) {\dots};\draw[violet] (1.5,2.5)--(1.5,3);
					\draw[gray] (-.5,3)--(3,3);
					\node[font=\small] at (1.25,-.5){$e_{i+1}=D1$};
				\end{tikzpicture}
				\]
				Then one can check that the ending word of $\lambda_{i+1}$ is the reduced word 
				$\ux_{i+1}$	and the induction can continue.
		
				Finally we define $L_{\uw,\boe}$ to be $\lambda_{\ell(\uw)}\in \Hom_{\D}(B_{\uw},B_{\ux})$.
		\end{dfn*}
		\begin{exa}
			Let $\uw=\underline{stsuts}$, with $m_{st}=3$, and $\ux=st$. Consider the subexpression
			$\boe=111001$ expressing $x$. One can check that the corresponding decoration of $\boe$ is:
			\[
				U1\,U1\,U1\,U0\,D0\,D1.
			\]
			The following is a construction of $L_{\uw,\boe}$.
			\begin{gather*}
				\lambda_0=
					\begin{tikzpicture}[baseline=.4cm]
						\draw[gray] (0,0)--(1,0);\draw[gray] (0,1)--(1,1);
					\end{tikzpicture}\qquad
				\lambda_1=
					\begin{tikzpicture}[baseline=.4cm]
						\draw[gray] (0.25,0)--(.75,0);\draw[gray] (0.25,1)--(.75,1);
						\draw[red] (0.5,0)--(0.5,1);
					\end{tikzpicture}\qquad
				\lambda_2=
					\begin{tikzpicture}[baseline=.4cm]
						\draw[gray] (0.25,0)--(1.25,0);\draw[gray] (0.25,1)--(1.25,1);
						\draw[red] (0.5,0)--(0.5,1); \draw[blue] (1,0)--(1,1);
					\end{tikzpicture}\qquad
				\lambda_3=
					\begin{tikzpicture}[baseline=.4cm]
						\draw[gray] (0.25,0)--(1.75,0);\draw[gray] (0.25,1)--(1.75,1);
						\draw[red] (0.5,0)--(0.5,1); \draw[blue] (1,0)--(1,1);\draw[red] (1.5,0)--(1.5,1);
					\end{tikzpicture} \\[1em]
				\lambda_4=
					\begin{tikzpicture}[baseline=.4cm,x=.8cm]
						\draw[gray] (0.25,0)--(2.25,0);\draw[gray] (0.25,1)--(2.25,1);
						\draw[red] (0.5,0)--(0.5,1); \draw[blue] (1,0)--(1,1);\draw[red] (1.5,0)--(1.5,1);
						\draw[green] (2,0)--(2,0.5);\fill[green] (2,0.5) circle (1.5pt);
					\end{tikzpicture} \qquad
				\lambda_5=
					\begin{tikzpicture}[baseline=.65cm,x=.7cm]
						\draw[gray] (0.25,0)--(2.75,0);\draw[gray] (0.25,1.5)--(2.75,1.5);
						\draw[red] (0.5,0)--(0.5,0.5); \draw[blue] (1,0)--(1,.5);\draw[red] (1.5,0)--(1.5,.5);
						\draw[red] (0.5,0.5)--(1,0.75)--(1,1.5);\draw[red] (1.5,.5)--(1,0.75);
						\draw[blue] (1,0.5)--(1,0.75)--(0.5,1)--(0.5,1.5); \draw[blue] (1,0.75)--(1.5,1);
						\draw[green] (2,0)--(2,0.4);\fill[green] (2,0.4) circle (1.5pt);
						\draw[blue] (2.5,0)--(2.5,1);
						\draw[blue] (1.5,1)--(2,1.25)--(2,1.5);\draw[blue] (2.5,1)--(2,1.25);
					\end{tikzpicture} \qquad 
				\lambda_6=
					\begin{tikzpicture}[baseline=.9cm,x=.7cm]
						\draw[gray] (0.25,0)--(3.25,0);\draw[gray] (0.25,2)--(3.25,2);
						\draw[red] (0.5,0)--(0.5,0.5); \draw[blue] (1,0)--(1,.5);\draw[red] (1.5,0)--(1.5,.5);
						\draw[red] (0.5,0.5)--(1,0.75)--(1,2);\draw[red] (1.5,.5)--(1,0.75);
						\draw[blue] (1,0.5)--(1,0.75)--(0.5,1)--(0.5,2); \draw[blue] (1,0.75)--(1.5,1);
						\draw[green] (2,0)--(2,0.4);\fill[green] (2,0.4) circle (1.5pt);
						\draw[blue] (2.5,0)--(2.5,1);
						\draw[blue] (1.5,1)--(2,1.25)--(2,1.5);\draw[blue] (2.5,1)--(2,1.25);
						\draw[blue] (3,0)--(3,1.5) ..controls (3,2) and (2,2).. (2,1.5);
					\end{tikzpicture}
					=
					\begin{tikzpicture}[x=.35cm,baseline=.9cm]
						\draw[gray] (.5,0)--(6.5,0);\draw[gray] (.5,2)--(6.5,2);
						\draw[red] (1,0)..controls (1,.5).. (2,1) ..controls (3,.5)..(3,0);
						\draw[red] (2,1)--(2,2);
						\draw[blue] (2,0)--(2,1)..controls (1,1.5)..(1,2);
						\draw[blue] (2,1)..controls (3,1.5) and (6,1.5).. (6,0);
						\clip (2,0)--(2,1)..controls (3,1.5) and (6,1.5).. (6,0);
						\draw[blue] (5,0)--(5,1.5);
						\draw[green] (4,0)--(4,0.5); \fill[green] (4,0.5) circle (1.5pt);
					\end{tikzpicture}
			\end{gather*}
		\end{exa}
		Now consider two Coxeter words $\uw$ and $\uw'$. We want to describe an $R$-basis for
		$\Hom(B_{\uw},B_{\uw'})$. 
		Consider the \emph{flip operation} $\overline{(\cdot)}$ on diagrams, consisting of flipping 
		diagrams upside down.
		Now, for each pair of 01-sequences $\boe$ of $\uw$ and $\boe'$ 
		of $\uw'$ such that $\uw^{\boe}=(\uw')^{\boe'}$ we pick \emph{a choice of} morphism
		$L_{\uw,\boe}$ and \emph{a choice of} morphism $L_{\uw',\boe'}$ as defined above. This means 
		that we make a choice for all the needed reduced words and braid moves and consider the 
		corresponding light leaves maps. Now let $\mathbb{L}_{\uw,\uw'}\subset \Hom(B_{\uw},B_{\uw'})$ 
		be the set of the compositions $L_{\boe,\boe'}:=\overline{L_{\uw',\boe'}}\circ L_{\uw,\boe}$. 		
		Morphisms as those in $\mathbb{L}_{\uw,\uw'}$ are called \emph{double leaves maps}.
		Then one has the following results (see \cite[\S\ 6.3]{EW}).
		\begin{pro}\label{pro_uppertriang}
			Let $\bof$ and $\bof'$ be subexpressions of $\uw$ and $\uw'$ respectively.
			The component $(L_{\boe,\boe'})^{\bof,\bof'}\in \Hom(Q_{\bof},Q_{\bof'})$, obtained 
			by localization, satisfies:
			\begin{enumerate}
				\item $(L_{\boe,\boe'})^{\bof,\bof'}=0$ unless $\uw^\bof\neq (\uw')^{\bof'}$, and both
					$\bof\le \boe$ and $\bof'\le\boe'$;
				\item $(L_{\boe,\boe'})^{\bof,\bof'}\neq 0$ if $\boe=\bof$ and $\boe'=\bof'$.
			\end{enumerate}
		\end{pro}
		\begin{thm}\label{thm_doubleleaves}
			The set $\mathbb{L}_{\uw,\uw'}$ is a basis of $\Hom_{\D}(B_{\uw},B_{\uw'})$ as a left $R$-module.	
		\end{thm}
		\section{Diagrammatics for dg Rouquier complexes}
		In this section we introduce the Rouquier complexes in the homotopy category of $\D$. Then we adapt 
		the diagrammatic description to one for the dg monoidal category generated by their standard representatives.
		\subsection{Rouquier complexes}\label{subs_Roucom}
		Consider the homotopy category $\Kb(\D)$. 
		The monoidal structure of $\D$ extends to $\Kb(\D)$ via the usual definition of tensor product
		of complexes. Recall that this is as follows. Let $\Ao,\Bo\in \Kb(\D)$, then:
		\[
			(\Ao\otimes \Bo)^i=\bigoplus_{p+q=i} A^p\otimes B^q,
		\]
		and the differential map, restricted to $A^p\otimes B^q$ is:
		\begin{equation}\label{eq_signconv}
			d\otimes \id + (-1)^p \id \otimes d.
		\end{equation}
		Notice that the unit of $\D$, seen as a complex concentrated in degree 0, and denoted by $\un[0]$,
		is the unit of $\Kb(\D)$.
			
		Consider the following \emph{standard} and \emph{costandard} complexes:
		\[
			\begin{tikzcd}[row sep=small,%
							execute at end picture={%
								\def\hdott{.32cm}
								\pic[red,yshift=.3em] at (A) {dot={\hdott}};
								\pic[red,yshift=1.5em,yscale=-1] at (B) {dot={\hdott}};
								}]
				F_{\re{s}}=\,\cdots\ar[r]		& 0 \ar[r]		& 0	\ar[r]								& B_{\re{s}}\ar[r,""{coordinate,name=A}]	& \un(1)\ar[r]	& 0\ar[r]		& \cdots \\
															& \cohdeg{-2}	& \cohdeg{-1}							& \cohdeg{0}											& \cohdeg{1}	& \cohdeg{2}	& \\
				F_{\re{s}}^{-1}=\,\cdots\ar[r]	& 0\ar[r]		& \un(-1)\ar[r,""{coordinate,name=B}]	& B_{\re{s}}\ar[r]							& 0	\ar[r]		& 0\ar[r]		& \cdots \\
			\end{tikzcd}
		\]
		where the numbers in the middle row denote the 
		cohomological degree in $\Kb(\D)$. If $\sigma=\sigma_{s}^{\pm 1}\in \Sigma$, then let 
		$F_\sigma$ denote $F_s^{\pm 1}$.
		Then, for any braid word $\uom=\sigma_1 \sigma_2 \dots \sigma_n$, with $\sigma_i\in \Sigma$,
		we put:
		\[
			F_{\uom}^\bullet:=F_{\sigma_1}\otimes F_{\sigma_2}\otimes \dots \otimes F_{\sigma_n}.
		\] 
		In the sequel we will often omit tensor products.
		Objects of this form are called
		\emph{Rouquier complexes}. They were introduced in \cite{Rou_cat} in terms of Soergel bimodules 
		to categorify braid group actions on categories and study natural transformations 
		between the induced endofunctors.
		\subsection{Basic properties of Rouquier complexes}
		These properties were first proved by Rouquier \cite{Rou_cat} in the 
		language of Soergel bimodules. 
		\begin{pro}\label{pro_braidRou}
			One has the following.
			\begin{enumerate}
			 	\item\label{item_propinvers} Let $s\in S$, then $F_s F_s^{-1}\cong F_s^{-1}F_s\cong \un[0]$.
				\item\label{item_propbraid} Let $s,t\in S$ with $m_{st}<\infty$, then:
					\[
						\underbrace{F_s F_tF_sF_t \cdots}_{m_{st} \text{ times}} \cong %
							\underbrace{F_t F_sF_tF_s \cdots}_{m_{st} \text{ times}}
					\]
			\end{enumerate}
			Hence, for each pair of braid words $\uom_1$, $\uom_2$ expressing
			the same element $\omega\in B_W$, there is an isomorphism $F_{\uom_1}^\bullet\cong F_{\uom_2}^\bullet$.
			Furthermore:
			\begin{enumerate}[resume]
				\item\label{item_propcanonic}(Rouquier Canonicity)
					for each $\uom_1$ and $\uom_2$ as above, we have:
					\[
						\Hom(F_{\uom_1}^\bullet,F_{\uom_2}^\bullet)\cong R,
					\]
					and one can chose homotopy equivalences
					$\gamma_{\uom_1}^{\uom_2}\in \Hom(F_{\uom_1}^\bullet,F_{\uom_2}^\bullet)$ such that the system 
					$\{\gamma_{\uom_1}^{\uom_2}\}_{\uom_1,\uom_2}$ is transitive, namely:
					\begin{align*}
						\gamma_{\uom}^{\uom}=&\id_{F^\bullet_{\uom}},& 										&\text{for all $\uom$,}\\
						\gamma_{\uom_2}^{\uom_3}\circ\gamma_{\uom_1}^{\uom_2}=&\gamma_{\uom_1}^{\uom_3},&	&\text{for all $\uom_1$, $\uom_2$, $\uom_3$.}						
					\end{align*}
			\end{enumerate}
		\end{pro}
		\begin{rmk}
			Let $\mathscr{B}_W:=\langle F_{\sigma}\mid \sigma\in \Sigma \rangle_\otimes$ inside $\Kb(\D)$. Thanks to
			Proposition \ref{pro_braidRou}, the isomorphism classes of this subcategory form at least a quotient 
			of the braid group. In particular the Rouquier complex $F_{\omega}$ associated to 
			$\omega\in B_W$ is well defined up to a canonical isomorphism.
			Rouquier conjectured that this quotient is just $B_W$. This was proven in all finite types by
			Khovanov and Seidel \cite{KhoSei}, Brav and Thomas \cite{BraTho} and Jensen \cite{Jen}. 
		\end{rmk}
		In what follows, we will consider the representatives $\Roudg{\uom}$ in the 
		dg category of complexes $\Cdg(\D)$ and describe morphism spaces between them using Soergel
		diagrams. Then the morphisms
		in the homotopy category are obtained as the cohomology at zero of the morphism complexes 
		in $\Cdg(\D)$.
		In this setting we will recover the fundamental properties of Rouquier 
		complexes, stated in Proposition \ref{pro_braidRou}. This approach will give an
		algorithm to determine explicitly the homotopy equivalence categorifying the braid relation.
		Furthermore we will give a diagrammatic proof of the Rouquier formula (see Corollary \ref{cor_Roufor}).
			\subsection{Notation for graded objects and complexes}
			Given a $\Bbbk$-linear additive category $\mathscr{C}$, a \emph{graded object} is a family $\{A^q\}_{q\in \ZZ}$,
			where $A^q$ is an object of $\mathscr{C}$ for all $q\in \ZZ$. If $A$ is an object of $\mathscr{C}$, then
			$A[-q]$ denotes the graded object whose only nonzero entry is $A$ in degree $q$.
			The direct sum of graded objects is defined in a natural way, so that we can also write:
			\[
				\{A^q\}_{q\in\ZZ}=\bigoplus_{q\in \ZZ} A^q[-q].
			\] 
			A (degree 0) morphism of graded objects $\Ao=\{A^q\}$ and $\Bo=\{B^q\}$ is a family of morphisms:
			\[
				\{f^q:A^q\rightarrow B^{q}\}_{q\in \ZZ}.
			\]
			This defines a category $\mathcal{G}_0(\mathscr{C})$
			of graded objects. The \emph{shift} functor $[1]$ sends the graded object $\{A^q\}_{q\in \ZZ}$ to
			$\{A^{q+1}\}_{q\in \ZZ}$. Then one can define the \emph{graded category} $\mathcal{G}(\mathscr{C})$ with the 
			same objects as $\mathcal{G}_0(\mathscr{C})$ and morphism spaces:
			\[
				\Hom_{\mathcal{G}(\mathscr{C})}(\Ao,\Bo):=\bigoplus_{p\in \ZZ} \Hom_{\mathcal{G}_0(\mathscr{C})}(\Ao,\Bo[p]).
			\]
			The $p$-th graded piece is denoted by $\Hom^p(\Ao,\Bo)$ and its elements are called 
			\emph{homogeneous} of degree $p$. 
			A complex in $\C(\mathscr{C})$ is then determined by a graded object $\Ao$ endowed with an
			endomorphism $d:\Ao\rightarrow \Ao$ of degree 1, such that $d^2=0$.
			Viceversa, let $\grad{(\cdot)}:\C(\mathscr{C})\rightarrow\mathcal{G}(\mathscr{C})$ denote 
			the forgetful functor sending $(\Ao,d)$ to $\Ao$.
			
			The \emph{dg category} of complexes $\Cdg(\CC)$ has the same objects as $\C(\CC)$
			but the morphism space between the complexes $\Ao$ and $\Bo$, is 
			the dg $\kk$-module whose underlying graded module is $\Hom_{\mathcal{G}(\CC)}(\Ao,\Bo)$, 
			and the differential map is:
			\begin{align}\label{eq_curlydiff}
				\Hom^p(\Ao,\Bo) & \longrightarrow \Hom^{p+1}(\Ao,\Bo) \\
				(f^q)_{q\in\mathbb{Z}} & \longmapsto (d_{B}^{q+p}\circ f^{q} - (-1)^p f^{q+1}\circ d_A^{q}). \notag
			\end{align}
			A homogeneous map $f$ is called \emph{closed} if it lies in the kernel of \eqref{eq_curlydiff}
			and \emph{exact} if it lies in the image. It is easy to see that closed degree zero
			maps are precisely morphisms in $\C(\CC)$, and
			exact degree zero maps are null-homotopic morphisms of complexes.
			Hence, the cohomology at zero of the dg module $\Homb(\Ao,\Bo)$ is the space of 
			morphisms in the category $\mathcal{K}(\D)$.
			
			Finally, let $\Gb(\mathscr{C})$, $\Cb(\mathscr{C})$, $\Cdgb(\CC)$ and $\Kb(\CC)$ 
			denote the \emph{bounded} subcategories, in the obvious sense.
			\subsection{A dg monoidal category of Rouquier complexes}
			Consider the Hecke category $\D$ defined in \S\,\ref{subs_defcatHeck}
			and take $\Cdg(\D)$.
			By the definition of $\D$, this is a dg monoidal category and the space 
			$\Hom^\bullet(A^\bullet,B^\bullet)$ has the structure of dg $R$-bimodule, with an additional
			grading, inherited from	the polynomial grading of $\D$. The shift in the polynomial grading 
			is still denoted by $(1)$ and the shift in the cohomological degree by $[1]$, as in the homotopy
			category.
			
			We consider the dg monoidal subcategory generated by the standard and costandard complexes:
			\begin{equation}
				\Adg=\langle F_\sigma \mid \sigma \in \Sigma \rangle_{\oplus,\otimes,[\cdot]}\subset \Cdg(\D).\\
			\end{equation}
			Its objects will be direct sums of shifts of the $\Roudg{\uom}$'s.
			In this section we will adapt the Elias-Williamson presentation to $\Adg$. 
			\subsection{Diagrammatics}\label{subs_heckcat_diagrams}
			For a given subexpression $\boi$ of a braid word $\uom$, let $\uom_{\boi}$ denote the 
			subword corresponding to it. Then let $B_{\boi}$ be the Bott-Samelson object 
			corresponding to the Coxeter projection of $\uom_{\boi}$. For example, given:
			\[
				\uom=\sigma_s\sigma_t\sigma_t\sigma_u^{-1},
			\]
			and $\boi=0101$, we have 
			$\uom_{\boi}=\sigma_t\sigma_u^{-1}$ and then $B_{\boi}=B_{\underline{tu}}=B_tB_u$.
			
			By definition of tensor product of complexes, we have:
			\begin{equation}
				\grad{(\Roudg{\uom})}=\bigoplus_{\boi\in\{0,1\}^{\ell(\uom)}}B_{\boi}\Tate{q_{\boi}},
			\end{equation}
			where $q_{\boi}=\varpi(\uom)-\varpi(\uom_{\boi})$ and $\Tate{1}$ denotes the \emph{Tate twist}
			 $(1)[-1]$ (see \cite[\S\,2.2]{AchRic_modII}).
			 One can then compute the components of the differential map, according to
			 \eqref{eq_signconv}. The only nonzero components $B_{\boi}\rightarrow B_{\boi'}$ are those 
			 for which $\boi'$ is obtained from $\boi$ by changing one symbol in such a way that 
			 $q_{\boi'}=q_{\boi}+1$.
			 In this case the component is either:
			 \begin{align}\label{eq_diffonsubexpr}
			 	&(-1)^k	 	
					\begin{tikzpicture}[x=0.3cm,y=0.4cm,baseline=0.36cm,every node/.style={font=\scriptsize}]
						\draw[gray] (0.5,0)--(9.5,0);\draw[gray] (0.5,2)--(9.5,2);
						\foreach \i in {1,2,4,6,8,9}{%
								\draw[violet] (\i,0)--(\i,2);
							}
						\draw[violet] (5,0)--(5,.6); \fill[violet] (5,.6) circle (1.5pt);
						\node[font=\normalsize] at (3,1) {\dots}; \node[font=\normalsize] at (7,1) {\dots};
					\end{tikzpicture}&
					&\text{or}&
					&(-1)^k
						\begin{tikzpicture}[x=0.3cm,y=0.4cm,baseline=0.36cm,every node/.style={font=\scriptsize}]
							\draw[gray] (0.5,0)--(9.5,0);\draw[gray] (0.5,2)--(9.5,2);
							\foreach \i in {1,2,4,6,8,9}{%
									\draw[violet] (\i,0)--(\i,2);
								}
							\draw[violet] (5,2)--(5,1.4); \fill[violet] (5,1.4) circle (1.5pt);
							\node[font=\normalsize] at (3,1) {\dots}; \node[font=\normalsize] at (7,1) {\dots};
						\end{tikzpicture},&
			 \end{align}
			 where $k$ is the number of 0's in $\boi$ preceding the changed symbol.
			\begin{exa}\label{exa_cube}
				Let $\re{s},\bl{t}\in S$ and 
				$\uom=\sigma_s\sigma_s\sigma_t^{-1}$. Then the complex $\Roudg{\uom}=F_sF_sF_t^{-1}$ 
				is the following:
				\begin{center}
					\begin{tikzpicture}[every pic/.style={scale=0.4,font=\tiny,-}]
						\def \d{3}%
						\def \h{1}%
						\summandnode{ss9}{0}{0}{ss9};
						\summandnode{sst}{\d}{2*\h}{sst};\summandnode{s}{\d}{0}{s1};\summandnode{s}{\d}{-2*\h}{s2};
						\node (s1t1) at (2*\d,2*\h) {$B_sB_t(1)$};
						\node (s2t1) at (2*\d,0) {$B_sB_t(1)$};
						\node (n1) at (2*\d,-2*\h) {$\un(1)$};
						\summandnode{t2}{3*\d}{0}{t2};
						\draw[-latex] (ss9) to %
							pic[midway,yshift=.5cm]{bottomdot={sst}{3}{+}} (sst);%
						\draw[-latex] (ss9) to%
							pic[midway,yshift=.3cm]{topdot={ss}{2}{+}} (s1);%
						\draw[-latex] (ss9) to%
							pic[midway,yshift=-.5cm]{topdot={ss}{1}{+}} (s2);%
						\draw[-latex] (sst) to%
							pic[midway,yshift=.3cm]{topdot={sst}{2}{+}} (s1t1);%
						\draw[-latex] (sst) to%
							pic[pos=.6,xshift=.8cm]{topdot={sst}{1}{+}} (s2t1);%
						\draw[-latex] (s1) to%
							pic[pos=.4,xshift=-.6cm]{bottomdot={st}{2}{-}} (s1t1);%
						\draw[-latex] (s1) to%
							pic[near start,xshift=-.3cm]{topdot={s}{1}{+}} (n1);%
						\draw[-latex] (s2) to%
							pic[pos=.6,xshift=.8cm]{bottomdot={st}{2}{-}} (s2t1);%
						\draw[-latex] (s2) to%
							pic[midway,yshift=-.3cm]{topdot={s}{1}{-}} (n1);%
						\draw[-latex] (s1t1) to%
							pic[midway,yshift=.5cm]{topdot={st}{1}{+}} (t2);%
						\draw[-latex] (s2t1) to%
							pic[midway,yshift=.3cm]{topdot={st}{1}{-}} (t2);%
						\draw[-latex] (n1) to%
							pic[midway,yshift=-.5cm]{bottomdot={t}{1}{+}} (t2);%
						\node[green,font=\small] at (0,-3*\h){-1}; \node[green,font=\small] at (\d,-3*\h){0};%
						\node[green,font=\small] at (2*\d,-3*\h){1}; \node[green,font=\small] at (3*\d,-3*\h) {2};
						\node at (\d,\h) {$\oplus$}; \node at (\d,-\h) {$\oplus$};
						\node at (2*\d,\h) {$\oplus$}; \node at (2*\d,-\h) {$\oplus$};
					\end{tikzpicture}		
				\end{center}			
			\end{exa}	  			 
			 Now let $\uom,\uom'\in \bword$ and consider $\Homb(F_{\uom}^\bullet,F_{\uom'}^\bullet)$. This decomposes as:
			 \begin{equation}\label{eq_morphsummandsubexpr}
			 	\bigoplus_{\boi,\boi'} \Hom_{\D}(B_{\boi},B'_{\boi'})\Tate{q'_{\boi'}-q_{\boi}},
			 \end{equation}
			 for an analogous definition of $B'_{\boi'}$ and $q'_{\boi'}$.
			 Each of the summands of \eqref{eq_morphsummandsubexpr} is a morphism space in the 
			 diagrammatic Hecke category 
			 $\D$, so its
			 elements are linear combinations of Soergel diagrams as described in \S\,\ref{subs_defcatHeck}.
			 Hence an element of $\Homb(F_{\uom}^\bullet,F_{\uom'}^\bullet)$ is also a $\kk$-linear combination of 
			 diagrams each of which lives in one of the summands above. To keep track of
			 this information we write it as a linear combination of \emph{dg diagrams}, or \emph{diagrams with patches}. 
			 A dg diagram is defined as follows, use Example \ref{exa_dgdiag} below as a reference.
			\begin{enumerate}
				\item We draw two boundary lines and we arrange the letters of $\uom$ on the bottom one
					and those of $\uom'$ on the top.
					We use different styles for the boundary line according to the sign of the 
					letters:
					\begin{itemize}
						\item the black line is for positive letters on the bottom boundary or negative 
							letters on the top boundary;
						\item the white line is for all the other ones.
					\end{itemize}
					The reason for which we invert the colors on the top boundary will be clear very soon.
					For the empty word we only draw a normal line.
				\item The letters of the starting and ending words correspond to the $1$'s in the 
					subexpressions
					$\boi$ and $\boi'$, so we cover with a \emph{patch} 
					\begin{tikzpicture}
						\pic[violet] at (0,0) {patch={1}};
					\end{tikzpicture}
					every letter corresponding to a $0$.
				\item A dg diagram is then a usual Soergel diagram whose boundary points are those
					which are not covered by patches.
			\end{enumerate}
			\begin{exa}\label{exa_dgdiag}
				Let $\uom=\sigma_{\re{s}}^3\sigma_{\bl{t}}^{-1}\sigma_{\gr{u}}^{-1}\sigma_{\bl{t}}$,
				and $\uom'=\sigma_{\re{s}}\sigma_{\bl{t}}^{-2}\sigma_{\gr{u}}^2$.
				The following is a dg diagram representing a morphism in the summand 
				$\Hom_{\D}(B_{011111},B'_{10101})$.
				\begin{equation}\label{eq_expatches}
					\begin{tikzpicture}[x=.7cm,baseline=.7cm]
						\node[anchor=north] (aa) at (0,0) {};
						\node[anchor=north] (a) at (1,0) {};
						\node[anchor=north] (b) at (2,0) {};
						\node[anchor=north] (c) at (3,0) {};
						\node[anchor=north] (d) at (4,0) {};
						\node[anchor=north] (e) at (5,0) {};
						\node[anchor=south] (A) at (.5,2) {};
						\node[anchor=south] (BB) at (1.5,2) {};
						\node[anchor=south] (B) at (2.5,2) {};
						\node[anchor=south] (CC) at (3.5,2) {};
						\node[anchor=south] (C) at (4.5,2) {};
						\node[font=\small,anchor=south] at (0,-.53) {$\sigma_{\re{s}}$};
						\node[font=\small,anchor=south] at (1,-.53) {$\sigma_{\re{s}}$};
						\node[font=\small,anchor=south] at (2,-.53) {$\sigma_{\re{s}}$};
						\node[font=\small,anchor=south] at (3,-.53) {$\sigma_{\bl{t}}^{-1}$};
						\node[font=\small,anchor=south] at (4,-.53) {$\sigma_{\gr{u}}^{-1}$};
						\node[font=\small,anchor=south] at (5,-.53) {$\sigma_{\bl{t}}$};
						\node[font=\small,anchor=south] at (.5,2.1) {$\sigma_{\re{s}}$};
						\node[font=\small,anchor=south] at (1.5,2.1) {$\sigma_{\bl{t}}^{-1}$};
						\node[font=\small,anchor=south] at (2.5,2.1) {$\sigma_{\bl{t}}^{-1}$};
						\node[font=\small,anchor=south] at (3.5,2.1) {$\sigma_{\gr{u}}$};
						\node[font=\small,anchor=south] at (4.5,2.1) {$\sigma_{\gr{u}}$};
						\draw[red] (b)..controls (2,.5) and (0.5,1.5)..(A);
						\draw[red] (a)--(1,.5);\fill[red] (1,.5) circle (1.5pt);
						\draw[blue](e)..controls (5,.5) and (2.5,1.5)..node[shape=coordinate] (mid) {} (B);
						\draw[blue] (c)..controls (3,.5) and (3.5,.6)..(mid);
						\draw[green](d)--(4,.5);\fill[green] (4,.5)circle(1.5pt);
						\draw[green](C)--(4.5,1.5);\fill[green] (4.5,1.5)circle(1.5pt);
						\node at (2.5,1) {$\alpha_{\gr{u}}$};
						\draw[positive] (-0.5,0)--(2.5,0);
						\draw[negative] (2.5,0)--(4.5,0);
						\draw[positive] (4.5,0)--(5.5,0);
						\draw[negative] (-0.5,2)--(1,2);
						\draw[positive] (1,2)--(3,2);
						\draw[negative] (3,2)--(5.5,2);
						\pic[red] at (0,0) {patch={1}};	\pic[blue] at (1.5,2) {patch={1}};
						\pic[green] at (3.5,2) {patch={1}};
					\end{tikzpicture}			
				\end{equation}			
			\end{exa}
			\begin{rmk}\label{rmk_ondgdiagrams}
				\begin{enumerate}
					\item We recover the cohomological degree $p$ of a morphism from the patches:
						each patch counts $+1$ on the white line and $-1$ on the black line.
					\item The polynomial degree of a diagram is $d-p$, if $d$ is its degree as a Soergel diagram 
					 	(as described in \S\,\ref{subs_defcatHeck}) and $p$ is its cohomological degree. This is due to the
					 	Tate twists in \eqref{eq_morphsummandsubexpr}.
				\end{enumerate}
			\end{rmk}
			We can compute the differential map on the diagram according to \eqref{eq_curlydiff}
			and \eqref{eq_diffonsubexpr}.
	 		Consider the following operations.
	 		\begin{dfn}
	 			We call \emph{dot-sprouting} the operation of replacing a patch on the black boundary 
	 			with a boundary dot and \emph{strand-uprooting} the operation of covering a white boundary 
	 			point with a patch and cutting the corresponding strand with a dot.
				\begin{align*}
					&\text{\emph{Dot-sprouting}:}&
					&\begin{tikzpicture}[x=.5cm,baseline=.2cm]
						\draw[positive] (0,0)--(1,0);
						\pic[violet] at (.5,0) {patch={1}};
					\end{tikzpicture}
					\rightsquigarrow
					\begin{tikzpicture}[x=.5cm,baseline=.2cm]
						\pic[violet,y=.5cm] at (.5,0) {dot={1cm}};
						\draw[positive] (0,0)--(1,0);
					\end{tikzpicture}&
					&\begin{tikzpicture}[x=.5cm,baseline=-.3cm]
						\draw[positive] (0,0)--(1,0);
						\pic[violet] at (.5,0) {patch={1}};
					\end{tikzpicture}
					\rightsquigarrow
					\begin{tikzpicture}[x=.5cm,baseline=-.3cm]
						\pic[violet,yscale=-1,y=.5cm] at (.5,0) {dot={1cm}};
						\draw[positive] (0,0)--(1,0);
					\end{tikzpicture} \\[1em]
					&\text{\emph{Strand-uprooting:}}&
					&\begin{tikzpicture}[x=.5cm,baseline=.3cm]
						\draw[violet] (.5,0)--(.5,.5);
						\draw[violet,dashed] (.5,.5)--(.5,1);
						\draw[negative] (0,0)--(1,0);
					\end{tikzpicture}
					\rightsquigarrow
					\begin{tikzpicture}[x=.5cm,baseline=.3cm]
						\fill[violet] (.5,.3) circle (1.5pt);
						\draw[violet] (.5,.3)--(.5,.5);
						\draw[violet,dashed] (.5,.5)--(.5,1);
						\draw[negative] (0,0)--(1,0);
						\pic[violet] at (.5,0) {patch={1}};
					\end{tikzpicture}&
					&\begin{tikzpicture}[x=.5cm,baseline=-.5cm,yscale=-1]
						\draw[violet] (.5,0)--(.5,.5);
						\draw[violet,dashed] (.5,.5)--(.5,1);
						\draw[negative] (0,0)--(1,0);
					\end{tikzpicture}
					\rightsquigarrow
					\begin{tikzpicture}[x=.5cm,baseline=-.5cm,yscale=-1]
						\fill[violet] (.5,.3) circle (1.5pt);
						\draw[violet] (.5,.3)--(.5,.5);
						\draw[violet,dashed] (.5,.5)--(.5,1);
						\draw[negative] (0,0)--(1,0);
						\pic[violet] at (.5,0) {patch={1}};
					\end{tikzpicture}
				\end{align*}					 	
				We emphasize that we only allow the dot-sprouting operation on the black boundary, 
				and the	strand-uprooting on the white boundary.	
	 		\end{dfn}
			Then the image of a diagram via the differential map 
			is the linear combination of all diagrams obtained by one of the above operations, 
			with coefficients $\pm 1$ according to \eqref{eq_diffonsubexpr}. 
			More precisely, when acting on the top boundary, if $k$ is the number 
			of patches on the left of the sprouted dot/uprooted strand, the sign is 
			$(-1)^k$. The same rule applies to the bottom boundary but we also need to multiply 
			by an additional $(-1)^{p+1}$ where $p$ is the cohomological degree of the 
			initial morphism. 
			\begin{exa}
			The image via the differential map of \eqref{eq_expatches} is the combination below.
			\begin{multline*}
				d \left(%
				\begin{tikzpicture}[x=.3cm,y=.7cm,baseline=.6cm]
					\node[red,anchor=north] (aa) at (0,0) {};
					\node[red,anchor=north] (a) at (1,0) {};
					\node[red,anchor=north] (b) at (2,0) {};
					\node[blue,anchor=north] (c) at (3,0) {};
					\node[green,anchor=north] (d) at (4,0) {};
					\node[blue,anchor=north] (e) at (5,0) {};
					\node[red,anchor=south] (A) at (.5,2) {};
					\node[red,anchor=south] (BB) at (1.5,2) {};
					\node[blue,anchor=south] (B) at (2.5,2) {};
					\node[green,anchor=south] (CC) at (3.5,2) {};
					\node[red,anchor=south] (C) at (4.5,2) {};
					\draw[red] (b)..controls (2,.5) and (0.5,1.5)..(A);
					\draw[red] (a)--(1,.5);\fill[red] (1,.5) circle (1.5pt);
					\draw[blue](e)..controls (5,.5) and (2.5,1.5)..node[shape=coordinate] (mid) {} (B);
					\draw[blue] (c)..controls (3,.5) and (3.5,.6)..(mid);
					\draw[green](d)--(4,.5);\fill[green] (4,.5)circle(1.5pt);
					\draw[green](C)--(4.5,1.5);\fill[green] (4.5,1.5)circle(1.5pt);
					\node at (2.5,1) {$\alpha_{\gr{u}}$};
					\draw[positive] (-0.5,0)--(2.5,0);
					\draw[negative] (2.5,0)--(4.5,0);
					\draw[positive] (4.5,0)--(5.5,0);
					\draw[negative] (-0.5,2)--(1,2);
					\draw[positive] (1,2)--(3,2);
					\draw[negative] (3,2)--(5.5,2);
					\pic[red,scale=.7,transform shape] at (0,0) {patch={1}};	
					\pic[blue,scale=.7,transform shape] at (1.5,2) {patch={1}};
					\pic[green,scale=.7,transform shape] at (3.5,2) {patch={1}};
				\end{tikzpicture}			
				\right)
				=
				\begin{tikzpicture}[x=.3cm,y=.7cm,baseline=.6cm]
					\node[red,anchor=north] (aa) at (0,0) {};
					\node[red,anchor=north] (a) at (1,0) {};
					\node[red,anchor=north] (b) at (2,0) {};
					\node[blue,anchor=north] (c) at (3,0) {};
					\node[green,anchor=north] (d) at (4,0) {};
					\node[blue,anchor=north] (e) at (5,0) {};
					\node[red,anchor=south] (A) at (.5,2) {};
					\node[red,anchor=south] (BB) at (1.5,2) {};
					\node[blue,anchor=south] (B) at (2.5,2) {};
					\node[green,anchor=south] (CC) at (3.5,2) {};
					\node[red,anchor=south] (C) at (4.5,2) {};
					\draw[red] (b)..controls (2,.5) and (0.5,1.2)..(.5,1.7);
					\fill[red] (.5,1.7) circle (1.5pt);
					\draw[red] (a)--(1,.5);\fill[red] (1,.5) circle (1.5pt);
					\draw[blue](e)..controls (5,.5) and (2.5,1.5)..node[shape=coordinate] (mid) {} (B);
					\draw[blue] (c)..controls (3,.5) and (3.5,.6)..(mid);
					\draw[green](d)--(4,.5);\fill[green] (4,.5)circle(1.5pt);
					\draw[green](C)--(4.5,1.5);\fill[green] (4.5,1.5)circle(1.5pt);
					\node at (2.5,1) {$\alpha_{\gr{u}}$};
					\draw[positive] (-0.5,0)--(2.5,0);
					\draw[negative] (2.5,0)--(4.5,0);
					\draw[positive] (4.5,0)--(5.5,0);
					\draw[negative] (-0.5,2)--(1,2);
					\draw[positive] (1,2)--(3,2);
					\draw[negative] (3,2)--(5.5,2);
					\pic[red,scale=.7,transform shape] at (0,0) {patch={1}};	
					\pic[blue,scale=.7,transform shape] at (1.5,2) {patch={1}};
					\pic[green,scale=.7,transform shape] at (3.5,2) {patch={1}};
					\pic[red,scale=.7,transform shape] at (.5,2) {patch={1}};	
				\end{tikzpicture}			
				+
				\begin{tikzpicture}[x=.3cm,y=.7cm,baseline=.6cm]
					\node[red,anchor=north] (aa) at (0,0) {};
					\node[red,anchor=north] (a) at (1,0) {};
					\node[red,anchor=north] (b) at (2,0) {};
					\node[blue,anchor=north] (c) at (3,0) {};
					\node[green,anchor=north] (d) at (4,0) {};
					\node[blue,anchor=north] (e) at (5,0) {};
					\node[red,anchor=south] (A) at (.5,2) {};
					\node[red,anchor=south] (BB) at (1.5,2) {};
					\node[blue,anchor=south] (B) at (2.5,2) {};
					\node[green,anchor=south] (CC) at (3.5,2) {};
					\node[red,anchor=south] (C) at (4.5,2) {};
					\draw[red] (b)..controls (2,.5) and (0.5,1.5)..(A);
					\draw[red] (a)--(1,.5);\fill[red] (1,.5) circle (1.5pt);
					\draw[blue](e)..controls (5,.5) and (2.5,1.5)..node[shape=coordinate] (mid) {} (B);
					\draw[blue] (c)..controls (3,.5) and (3.5,.6)..(mid);
					\draw[green](d)--(4,.5);\fill[green] (4,.5)circle(1.5pt);
					\draw[green](C)--(4.5,1.5);\fill[green] (4.5,1.5)circle(1.5pt);
					\draw[blue] (1.5,2)--(1.5,1.5);\fill[blue] (1.5,1.5) circle (1.5pt);
					\node at (2.5,1) {$\alpha_{\gr{u}}$};
					\draw[positive] (-0.5,0)--(2.5,0);
					\draw[negative] (2.5,0)--(4.5,0);
					\draw[positive] (4.5,0)--(5.5,0);
					\draw[negative] (-0.5,2)--(1,2);
					\draw[positive] (1,2)--(3,2);
					\draw[negative] (3,2)--(5.5,2);
					\pic[red,scale=.7,transform shape] at (0,0) {patch={1}};	
					\pic[green,scale=.7,transform shape] at (3.5,2) {patch={1}};
				\end{tikzpicture}			
				+
				\begin{tikzpicture}[x=.3cm,y=.7cm,baseline=.6cm]
					\node[red,anchor=north] (aa) at (0,0) {};
					\node[red,anchor=north] (a) at (1,0) {};
					\node[red,anchor=north] (b) at (2,0) {};
					\node[blue,anchor=north] (c) at (3,0) {};
					\node[green,anchor=north] (d) at (4,0) {};
					\node[blue,anchor=north] (e) at (5,0) {};
					\node[red,anchor=south] (A) at (.5,2) {};
					\node[red,anchor=south] (BB) at (1.5,2) {};
					\node[blue,anchor=south] (B) at (2.5,2) {};
					\node[green,anchor=south] (CC) at (3.5,2) {};
					\node[red,anchor=south] (C) at (4.5,2) {};
					\draw[red] (b)..controls (2,.5) and (0.5,1.5)..(A);
					\draw[red] (a)--(1,.5);\fill[red] (1,.5) circle (1.5pt);
					\draw[blue](e)..controls (5,.5) and (2.5,1.5)..node[shape=coordinate] (mid) {} (B);
					\draw[blue] (c)..controls (3,.5) and (3.5,.6)..(mid);
					\draw[green](d)--(4,.5);\fill[green] (4,.5)circle(1.5pt);
					\fill[green] (4.5,1.7) circle (1.5pt); \draw[green](4.5,1.7)--(4.5,1.2);\fill[green] (4.5,1.2)circle(1.5pt);
					\node at (2.5,1) {$\alpha_{\gr{u}}$};
					\draw[positive] (-0.5,0)--(2.5,0);
					\draw[negative] (2.5,0)--(4.5,0);
					\draw[positive] (4.5,0)--(5.5,0);
					\draw[negative] (-0.5,2)--(1,2);
					\draw[positive] (1,2)--(3,2);
					\draw[negative] (3,2)--(5.5,2);
					\pic[red,scale=.7,transform shape] at (0,0) {patch={1}};	
					\pic[blue,scale=.7,transform shape] at (1.5,2) {patch={1}};
					\pic[green,scale=.7,transform shape] at (4.5,2) {patch={1}};
					\pic[green,scale=.7,transform shape] at (3.5,2) {patch={1}};
				\end{tikzpicture}			
				+\\+
				\begin{tikzpicture}[x=.3cm,y=.7cm,baseline=.6cm]
					\node[red,anchor=north] (aa) at (0,0) {};
					\node[red,anchor=north] (a) at (1,0) {};
					\node[red,anchor=north] (b) at (2,0) {};
					\node[blue,anchor=north] (c) at (3,0) {};
					\node[green,anchor=north] (d) at (4,0) {};
					\node[blue,anchor=north] (e) at (5,0) {};
					\node[red,anchor=south] (A) at (.5,2) {};
					\node[red,anchor=south] (BB) at (1.5,2) {};
					\node[blue,anchor=south] (B) at (2.5,2) {};
					\node[green,anchor=south] (CC) at (3.5,2) {};
					\node[red,anchor=south] (C) at (4.5,2) {};
					\draw[red] (b)..controls (2,.5) and (0.5,1.5)..(A);
					\draw[red] (a)--(1,.5);\fill[red] (1,.5) circle (1.5pt);
					\draw[blue](e)..controls (5,.5) and (2.5,1.5)..node[shape=coordinate] (mid) {} (B);
					\draw[blue] (c)..controls (3,.5) and (3.5,.6)..(mid);
					\draw[green](d)--(4,.5);\fill[green] (4,.5)circle(1.5pt);
					\draw[green](C)--(4.5,1.5);\fill[green] (4.5,1.5)circle(1.5pt);
					\draw[red] (0,0)--(0,0.5);\fill[red] (0,0.5) circle (1.5pt);
					\node at (2.5,1) {$\alpha_{\gr{u}}$};
					\draw[positive] (-0.5,0)--(2.5,0);
					\draw[negative] (2.5,0)--(4.5,0);
					\draw[positive] (4.5,0)--(5.5,0);
					\draw[negative] (-0.5,2)--(1,2);
					\draw[positive] (1,2)--(3,2);
					\draw[negative] (3,2)--(5.5,2);
					\pic[blue,scale=.7,transform shape] at (1.5,2) {patch={1}};
					\pic[green,scale=.7,transform shape] at (3.5,2) {patch={1}};
				\end{tikzpicture}			
				-
				\begin{tikzpicture}[x=.3cm,y=.7cm,baseline=.6cm]
					\node[red,anchor=north] (aa) at (0,0) {};
					\node[red,anchor=north] (a) at (1,0) {};
					\node[red,anchor=north] (b) at (2,0) {};
					\node[blue,anchor=north] (c) at (3,0) {};
					\node[green,anchor=north] (d) at (4,0) {};
					\node[blue,anchor=north] (e) at (5,0) {};
					\node[red,anchor=south] (A) at (.5,2) {};
					\node[red,anchor=south] (BB) at (1.5,2) {};
					\node[blue,anchor=south] (B) at (2.5,2) {};
					\node[green,anchor=south] (CC) at (3.5,2) {};
					\node[red,anchor=south] (C) at (4.5,2) {};
					\draw[red] (b)..controls (2,.5) and (0.5,1.5)..(A);
					\draw[red] (a)--(1,.5);\fill[red] (1,.5) circle (1.5pt);
					\draw[blue](e)..controls (5,.5) and (2.5,1.5)..node[shape=coordinate] (mid) {} (B);
					\draw[blue] (3,.3)..controls (3,.55) and (3.5,.6)..(mid);
					\fill[blue] (3,.3) circle (1.5pt);
					\draw[green](d)--(4,.5);\fill[green] (4,.5)circle(1.5pt);
					\draw[green](C)--(4.5,1.5);\fill[green] (4.5,1.5)circle(1.5pt);
					\node at (2.5,1) {$\alpha_{\gr{u}}$};
					\draw[positive] (-0.5,0)--(2.5,0);
					\draw[negative] (2.5,0)--(4.5,0);
					\draw[positive] (4.5,0)--(5.5,0);
					\draw[negative] (-0.5,2)--(1,2);
					\draw[positive] (1,2)--(3,2);
					\draw[negative] (3,2)--(5.5,2);
					\pic[red,scale=.7,transform shape] at (0,0) {patch={1}};	
					\pic[blue,scale=.7,transform shape] at (1.5,2) {patch={1}};
					\pic[blue,scale=.7,transform shape] at (3,0) {patch={1}};
					\pic[green,scale=.7,transform shape] at (3.5,2) {patch={1}};
				\end{tikzpicture}			
				-
				\begin{tikzpicture}[x=.3cm,y=.7cm,baseline=.6cm]
					\node[red,anchor=north] (aa) at (0,0) {};
					\node[red,anchor=north] (a) at (1,0) {};
					\node[red,anchor=north] (b) at (2,0) {};
					\node[blue,anchor=north] (c) at (3,0) {};
					\node[green,anchor=north] (d) at (4,0) {};
					\node[blue,anchor=north] (e) at (5,0) {};
					\node[red,anchor=south] (A) at (.5,2) {};
					\node[red,anchor=south] (BB) at (1.5,2) {};
					\node[blue,anchor=south] (B) at (2.5,2) {};
					\node[green,anchor=south] (CC) at (3.5,2) {};
					\node[red,anchor=south] (C) at (4.5,2) {};
					\draw[red] (b)..controls (2,.5) and (0.5,1.5)..(A);
					\draw[red] (a)--(1,.5);\fill[red] (1,.5) circle (1.5pt);
					\draw[blue](e)..controls (5,.5) and (2.5,1.5)..node[shape=coordinate] (mid) {} (B);
					\draw[blue] (c)..controls (3,.5) and (3.5,.6)..(mid);
					\fill[green] (4,.3)circle (1.5pt);\draw[green](4,.3)--(4,.55);\fill[green] (4,.55)circle(1.5pt);
					\draw[green](C)--(4.5,1.5);\fill[green] (4.5,1.5)circle(1.5pt);
					\node at (2.5,1) {$\alpha_{\gr{u}}$};
					\draw[positive] (-0.5,0)--(2.5,0);
					\draw[negative] (2.5,0)--(4.5,0);
					\draw[positive] (4.5,0)--(5.5,0);
					\draw[negative] (-0.5,2)--(1,2);
					\draw[positive] (1,2)--(3,2);
					\draw[negative] (3,2)--(5.5,2);
					\pic[red,scale=.7,transform shape] at (0,0) {patch={1}};	
					\pic[blue,scale=.7,transform shape] at (1.5,2) {patch={1}};
					\pic[green,scale=.7,transform shape] at (4,0) {patch={1}};
					\pic[green,scale=.7,transform shape] at (3.5,2) {patch={1}};
				\end{tikzpicture}			
			\end{multline*}
			In fact, we have three possibilities on top (uprooting the first or the last strand, or sprouting a dot over the blue patch)
			and three for the bottom (sprouting a dot on the red patch, or uprooting the middle blue strand or the green strand).
			Here the initial morphism has cohomological degree 1, so the additional sign term $(-1)^{p+1}$ 
			for the bottom boundary	is just $+1$.			
			\end{exa}
			Two dg diagrams $D_1$ and $D_2$ can be composed if the arrangement of letters of the target of
			$D_1$ is the same as the source of $D_2$ (notice that, by our convention, this means that
			the black or white styles are complementary). The composition will be zero unless
			the arrangement of patches along the gluing line is the same in $D_1$ and in $D_2$,
			because the target summand $B_{\boi'}$ of $D_1$ has to match with the source summand
			of $B_{\boi}$ of $D_2$.
			
			Recall that the monoidal structure of $\Cdg(\D)$ respects the Koszul rule, so
			the tensor product of two dg diagrams $D_1$ and $D_2$ is obtained by gluing them horizontally and multiplying by
			the sign $(-1)^{pq}$ where $p$ is the cohomological degree of $D_2$ and $q$ is the number of bottom patches of $D_1$.
			
			The identity morphism of any $\Roudg{\uom}$ is $\sum_{\boi}\id_{B_{\boi}}$,
			so it is represented by the sum of all diagrams with the same arrangements of patches on top and on bottom and consisting 
			of parallel vertical lines. For example, if $\uom=\sigma_{\re{s}}^2\sigma_{\bl{t}}^{-1}$, then:
			\[
				\id_{\Roudg{\uom}}=
				\begin{tikzpicture}[x=.3cm,y=.6cm,baseline=.5cm]
					\draw[red] (1,0)--(1,2);\draw[red](2,0)--(2,2);\draw[blue](3,0)--(3,2);
					\draw[positive] (.5,0)--(2.5,0);\draw[negative] (.5,2)--(2.5,2);
					\draw[negative] (2.5,0)--(3.5,0);\draw[positive] (2.5,2)--(3.5,2);
				\end{tikzpicture}+
				\begin{tikzpicture}[x=.3cm,y=.6cm,baseline=.5cm]
					\draw[red](2,0)--(2,2);\draw[blue](3,0)--(3,2);
					\draw[positive] (.5,0)--(2.5,0);\draw[negative] (.5,2)--(2.5,2);
					\draw[negative] (2.5,0)--(3.5,0);\draw[positive] (2.5,2)--(3.5,2);
					\foreach \i in {0,2}{
						\pic[red,scale=.5,transform shape] at (1,\i) {patch={1}};
						}
				\end{tikzpicture}+
				\begin{tikzpicture}[x=.3cm,y=.6cm,baseline=.5cm]
					\draw[red] (1,0)--(1,2);\draw[blue](3,0)--(3,2);
					\draw[positive] (.5,0)--(2.5,0);\draw[negative] (.5,2)--(2.5,2);
					\draw[negative] (2.5,0)--(3.5,0);\draw[positive] (2.5,2)--(3.5,2);
					\foreach \i in {0,2}{
						\pic[red,scale=.5,transform shape] at (2,\i) {patch={1}};
						}
				\end{tikzpicture}+
				\begin{tikzpicture}[x=.3cm,y=.6cm,baseline=.5cm]
					\draw[red] (1,0)--(1,2);\draw[red](2,0)--(2,2);
					\draw[positive] (.5,0)--(2.5,0);\draw[negative] (.5,2)--(2.5,2);
					\draw[negative] (2.5,0)--(3.5,0);\draw[positive] (2.5,2)--(3.5,2);
					\foreach \i in {0,2}{
						\pic[blue,scale=.5,transform shape] at (3,\i) {patch={1}};
						}
				\end{tikzpicture}+
				\begin{tikzpicture}[x=.3cm,y=.6cm,baseline=.5cm]
					\draw[blue](3,0)--(3,2);
					\draw[positive] (.5,0)--(2.5,0);\draw[negative] (.5,2)--(2.5,2);
					\draw[negative] (2.5,0)--(3.5,0);\draw[positive] (2.5,2)--(3.5,2);
					\foreach \i in {0,2}{
						\pic[red,scale=.5,transform shape] at (1,\i) {patch={1}};
						\pic[red,scale=.5,transform shape] at (2,\i) {patch={1}};
						}
				\end{tikzpicture}+
				\begin{tikzpicture}[x=.3cm,y=.6cm,baseline=.5cm]
					\draw[red](2,0)--(2,2);
					\draw[positive] (.5,0)--(2.5,0);\draw[negative] (.5,2)--(2.5,2);
					\draw[negative] (2.5,0)--(3.5,0);\draw[positive] (2.5,2)--(3.5,2);
					\foreach \i in {0,2}{
						\pic[red,scale=.5,transform shape] at (1,\i) {patch={1}};
						\pic[blue,scale=.5,transform shape] at (3,\i) {patch={1}};
						}
				\end{tikzpicture}+
				\begin{tikzpicture}[x=.3cm,y=.6cm,baseline=.5cm]
					\draw[red] (1,0)--(1,2);
					\draw[positive] (.5,0)--(2.5,0);\draw[negative] (.5,2)--(2.5,2);
					\draw[negative] (2.5,0)--(3.5,0);\draw[positive] (2.5,2)--(3.5,2);
					\foreach \i in {0,2}{
						\pic[red,scale=.5,transform shape] at (2,\i) {patch={1}};
						\pic[blue,scale=.5,transform shape] at (3,\i) {patch={1}};
						}
				\end{tikzpicture}+
				\begin{tikzpicture}[x=.3cm,y=.6cm,baseline=.5cm]
					\draw[positive] (.5,0)--(2.5,0);\draw[negative] (.5,2)--(2.5,2);
					\draw[negative] (2.5,0)--(3.5,0);\draw[positive] (2.5,2)--(3.5,2);
					\foreach \i in {0,2}{
						\pic[red,scale=.5,transform shape] at (1,\i) {patch={1}};
						\pic[red,scale=.5,transform shape] at (2,\i) {patch={1}};
						\pic[blue,scale=.5,transform shape] at (3,\i) {patch={1}};
						}
				\end{tikzpicture}.
			\]
			One can check that this is a closed degree
			zero morphism which is the neutral element of composition.
		\subsection{DG double leaves}
			By Theorem \ref{thm_doubleleaves}, each summand 
			of \eqref{eq_morphsummandsubexpr} is a free 
			$R$-module with a basis given by (a choice of) double leaves maps. 
			Collecting all these bases 
			together, we get a basis for the whole morphism space. 
			Recall that once we have fixed $\boi$ and $\boi'$ (that is, the source and target summands)
			the double leaves maps are labeled by pairs of subexpressions $\boe$ and $\boe'$ of $\uom_{\boi}$ and 
			of $\uom'_{\boi'}$ respectively. 
			It is convenient to introduce a uniform labeling. 

			An \emph{enriched subexpression}
			of a braid word $\uom\in\bword$ is an element of $\{\emptyset,0,1\}^{\ell(\uom)}$.			
			If $\uw$ is the Coxeter projection of $\uom$, then $\uw^\epsilon$ denotes the element
			of $W$ expressed by $\epsilon$, by considering each $\emptyset$ as $0$.

			We can label all the double leaves maps in a basis as above by pairs $(\epsilon,\epsilon')$
			of enriched subexpressions, where each entry $\epsilon_k$ of $\epsilon$ (resp.\ $\epsilon_k'$ of $\epsilon'$)
			is either $\emptyset$, when there is a patch, or the corresponding entry of $\boe$ (resp.\ $\boe'$).
			\begin{exa}
			Notice that the morphism \eqref{eq_expatches}, 
			if we forget the $\alpha_u$, is actually a double leaves map. The source and the target summands
			are determined by the subexpressions $\boi=011111$ and $\boi'=10101$, whereas the two light leaves maps correspond to
			the subexpressions $\boe=01100$ of $\underline{sstut}$ and 110 of $\underline{stu}$. We can determine the corresponding enriched subexpressions:
			\[
				\begin{array}{c|cccccc}
					\boi		& 0			& 1	& 1	& 1	& 1	& 1 \\
					\boe		& \cdot		& 0	& 1	& 1	& 0	& 0 \\
					\hline
					\epsilon	& \emptyset	& 0	& 1	& 1 & 0	& 0
				\end{array}
				\qquad
				\begin{array}{c|ccccc}
					\boi'		& 1	& 0			& 1 & 0			& 1\\
					\boe'		& 1	& \cdot		& 1	& \cdot		& 0 \\
					\hline
					\epsilon'	& 1	& \emptyset	& 1	& \emptyset	& 0
				\end{array}			
			\]
			\end{exa}			
			One can define Bruhat strolls and decorations for enriched subexpressions
			by simply considering $\emptyset$ as a 0. The exact same definition of path dominance order 
			thus gives that of \emph{path dominance preorder} $\lesssim$
			on enriched subexpressions. The relation is not antisymmetric any more (the Bruhat stroll does not see the difference between $\emptyset$ and 0) so the 
			relation defined by $\epsilon\sim\zeta$ if and only if $\epsilon\lesssim\zeta$ and $\zeta\lesssim\epsilon$ is an equivalence relation that we call
			\emph{path equivalence}. Notice that $\epsilon\sim\zeta$ implies that if $\epsilon_k\neq\zeta_k$ then $\epsilon_k,\zeta_k\in\{0,\emptyset\}$.

			One can apply localization to any morphism $\phi$ in \eqref{eq_morphsummandsubexpr}.
			First one chooses a component in $\Hom(B_{\boi},B_{\boi'})$, then
			one splits it in $\D_Q$, as explained in \S \ref{subs_localiz}, into 
			morphisms $Q_{\bof}\rightarrow Q_{\bof'}$ corresponding to subexpressions $\bof$ and $\bof'$ of the subwords $\uom_{\boi}$ and
			$\uom_{\boi'}$ respectively. 
			Again it is convenient to use enriched subexpressions to encode these two choices at once.
			If $\eta$ (resp.\ $\eta'$) is the enriched subexpression of $\uom$ (resp.\ $\uom'$) corresponding with
			$\boi$ and $\bof$ (resp.\ $\boi'$ and $\bof'$) then let $Q_\eta:=Q_{\bof}$ 
			(resp.\ $Q_{\eta'}:=Q_{\bof'}$) and $\phi^{\eta,\eta'}:Q_{\eta}\rightarrow Q_{\eta'}$
			be the morphism obtained by first taking the component of $\phi$ in $\Hom(B_{\boi},B_{\boi'})$ and
			then localizing with respect to $\bof$ and $\bof'$. 

			We have the following analogue of Proposition \ref{pro_uppertriang}.
			\begin{pro}\label{pro_uppertriangext}
				Let $\eta$ and $\eta'$ be subexpressions of $\uom$ and $\uom'$ respectively.
				Let $\uw$ and $\uw'$ be the respective Coxeter projections of the latter.
				The component $(L_{\epsilon,\epsilon'})^{\eta,\eta'}\in \Hom(Q_{\eta},Q_{\eta'})$ obtained by localization satisfies:
				\begin{enumerate}
					\item $(L_{\epsilon,\epsilon'})^{\eta,\eta'}=0$ unless $\uw^{\eta}= (\uw')^{\eta'}$, and both $\eta\lesssim \epsilon$ and $\eta'\lesssim\epsilon'$;
					\item $(L_{\epsilon,\epsilon'})^{\eta,\eta'}\neq 0$ if $\epsilon=\eta$ and $\epsilon'=\eta'$.
				\end{enumerate}				
			\end{pro}
			\begin{proof}
				It is enough to notice that $(L_{\epsilon,\epsilon'})^{\eta,\eta'}=0$ unless the $\epsilon$ and $\eta$ have the same $\emptyset$'s, namely
				$\epsilon_k=\emptyset \Leftrightarrow \eta_k=\emptyset$, and the same condition holds for $\epsilon'$ and $\eta'$. In this case
				$\eta\lesssim\epsilon$ is the same as $\bof\le\boe$ for the corresponding usual subexpressions, so we can apply Proposition \ref{pro_uppertriang}.
			\end{proof}
			We do not currently know of a general combinatorial method to predict the coefficients of
			$d(L_{\epsilon,\epsilon'})$ in the basis of double leaves. 
			Nevertheless we can prove the following key result.
			\begin{thm}[Path monotonicity of the differential]\label{thm_monotonicity}
				The differential decreases the path dominance preorder on double leaves. 
				More precisely, let $\{L_{\epsilon,\epsilon'}\}$ denote a set of double 
				leaves maps labeled as above. If we write:
				\[
					d(L_{\epsilon,\epsilon'})= \sum_{\zeta,\zeta'} a_{\zeta,\zeta'} L_{\zeta,\zeta'},
				\]
				then $a_{\zeta,\zeta'}=0$ unless $\zeta\lesssim \epsilon$ and $\zeta'\lesssim \epsilon'$.
			\end{thm}
			\begin{proof}
				We claim that the component	$(d(L_{\epsilon,\epsilon'}))^{\eta,\eta'}$ vanishes
				unless $\eta\lesssim \epsilon$ and $\eta'\lesssim\epsilon'$. 
				Write $d(L_{\epsilon,\epsilon'})=\sum_i\phi_i$ where the $\phi_i$ are
				the terms obtained by either uprooting or sprouting. We analyze each $\phi_i$ separately.
				Without loss of generality we can suppose that the operation happens on a bottom boundary
				point corresponding with $\epsilon_k$.
				\begin{enumerate}
					\item If $\epsilon_k=\emptyset$ and $\eta_k=1$, we are performing the following operations:
						\[
							\begin{tikzpicture}[x=.5cm,baseline=.2cm]
								\node at (-.75,.2) {\dots};\node at (2.25,.2) {\dots};
								\draw[gray] (-1.65,0)--(-1.65,.8);					
								\draw[gray] (-1.35,0)--(-1.35,.8);					
								\draw[gray] (-.15,0)--(-.15,.8);					
								\draw[gray] (1.65,0)--(1.65,.8);					
								\draw[gray] (2.85,0)--(2.85,.8);					
								\draw[gray] (3.15,0)--(3.15,.8);					
								\draw[gray,dashed] (-1.65,1.2)--(-1.65,.8);					
								\draw[gray,dashed] (-1.35,1.2)--(-1.35,.8);					
								\draw[gray,dashed] (-.15,1.2)--(-.15,.8);					
								\draw[gray,dashed] (1.65,1.2)--(1.65,.8);					
								\draw[gray,dashed] (2.85,1.2)--(2.85,.8);					
								\draw[gray,dashed] (3.15,1.2)--(3.15,.8);					
								\draw[positive] (-1.8,0)--(-1.5,0);\draw[negative] (-1.5,0)--(-1.2,0);\draw[positive] (-1.2,0)--(-.9,0);								
								\draw[negative] (-.9,0)--(-.6,0);\draw[positive] (-.6,0)--(-.3,0);\draw[negative] (-.3,0)--(0,0);								
								\draw[positive] (0,0)--(1.5,0);
								\draw[negative] (1.5,0)--(1.8,0);\draw[positive] (1.8,0)--(2.1,0);\draw[negative] (2.1,0)--(2.4,0);
								\draw[positive] (2.4,0)--(2.7,0);\draw[negative] (2.7,0)--(3,0);\draw[positive] (3,0)--(3.3,0);
								\pic[red] at (.75,0) {patch={1}};
							\end{tikzpicture}
							\rightsquigarrow
							\begin{tikzpicture}[x=.5cm,baseline=.2cm]
								\pic[red,y=.5cm] at (.75,0) {dot={1cm}};
								\node at (-.75,.2) {\dots};\node at (2.25,.2) {\dots};
								\draw[gray] (-1.65,0)--(-1.65,.8);					
								\draw[gray] (-1.35,0)--(-1.35,.8);					
								\draw[gray] (-.15,0)--(-.15,.8);					
								\draw[gray] (1.65,0)--(1.65,.8);					
								\draw[gray] (2.85,0)--(2.85,.8);					
								\draw[gray] (3.15,0)--(3.15,.8);					
								\draw[gray,dashed] (-1.65,1.2)--(-1.65,.8);					
								\draw[gray,dashed] (-1.35,1.2)--(-1.35,.8);					
								\draw[gray,dashed] (-.15,1.2)--(-.15,.8);					
								\draw[gray,dashed] (1.65,1.2)--(1.65,.8);					
								\draw[gray,dashed] (2.85,1.2)--(2.85,.8);					
								\draw[gray,dashed] (3.15,1.2)--(3.15,.8);					
								\draw[positive] (-1.8,0)--(-1.5,0);\draw[negative] (-1.5,0)--(-1.2,0);\draw[positive] (-1.2,0)--(-.9,0);								
								\draw[negative] (-.9,0)--(-.6,0);\draw[positive] (-.6,0)--(-.3,0);\draw[negative] (-.3,0)--(0,0);								
								\draw[positive] (0,0)--(1.5,0);
								\draw[negative] (1.5,0)--(1.8,0);\draw[positive] (1.8,0)--(2.1,0);\draw[negative] (2.1,0)--(2.4,0);
								\draw[positive] (2.4,0)--(2.7,0);\draw[negative] (2.7,0)--(3,0);\draw[positive] (3,0)--(3.3,0);
							\end{tikzpicture}
							\rightsquigarrow
							\begin{tikzpicture}[x=.5cm,baseline=.2cm]
								\draw[red,localized] (.75,0)--(.75,.5);\draw[red] (.75,.5)--(.75,.8);					
								\fill[red] (.75,.8) circle (1.5pt);
								\draw[gray,localized] (-.15,0)--(-.15,.5); \draw[gray] (-.15,.5)--(-.15,.8);					
								\fill[gray] (-1.35,.5) circle (1.5pt); \draw[gray] (-1.35,.5)--(-1.35,.8);					
								\draw[gray,localized] (-1.65,0)--(-1.65,.5);\draw[gray] (-1.65,.5)--(-1.65,.8);
								\fill[gray] (1.65,.5) circle (1.5pt);\draw[gray] (1.65,.5)--(1.65,.8);					
								\draw[gray,localized] (2.85,0)--(2.85,.5);\draw[gray] (2.85,.5)--(2.85,.8);
								\fill[gray] (3.15,.5) circle(1.5pt);\draw[gray] (3.15,.5)--(3.15,.8);					
								\draw[gray,dashed] (-1.65,1.2)--(-1.65,.8);					
								\draw[gray,dashed] (-1.35,1.2)--(-1.35,.8);					
								\draw[gray,dashed] (-.15,1.2)--(-.15,.8);					
								\draw[gray,dashed] (1.65,1.2)--(1.65,.8);					
								\draw[gray,dashed] (2.85,1.2)--(2.85,.8);					
								\draw[gray,dashed] (3.15,1.2)--(3.15,.8);					
								\node at (-.75,.2) {\dots};\node at (2.25,.2) {\dots};
								\draw[gray] (-2.1,0)--(3.6,0);
							\end{tikzpicture}
							=0
						\]						
						where the last equality follows from the first relation in \eqref{eq_localrel}.
					\item If $\epsilon_k=\emptyset$ and $\eta_k=0$, the first operation is the same, but then we have:
						\[
							\begin{tikzpicture}[x=.5cm,baseline=.2cm]
								\pic[red,y=.5cm] at (.75,0) {dot={1cm}};
								\node at (-.75,.2) {\dots};\node at (2.25,.2) {\dots};
								\draw[gray] (-1.65,0)--(-1.65,.8);					
								\draw[gray] (-1.35,0)--(-1.35,.8);					
								\draw[gray] (-.15,0)--(-.15,.8);					
								\draw[gray] (1.65,0)--(1.65,.8);					
								\draw[gray] (2.85,0)--(2.85,.8);					
								\draw[gray] (3.15,0)--(3.15,.8);					
								\draw[gray,dashed] (-1.65,1.2)--(-1.65,.8);					
								\draw[gray,dashed] (-1.35,1.2)--(-1.35,.8);					
								\draw[gray,dashed] (-.15,1.2)--(-.15,.8);					
								\draw[gray,dashed] (1.65,1.2)--(1.65,.8);					
								\draw[gray,dashed] (2.85,1.2)--(2.85,.8);					
								\draw[gray,dashed] (3.15,1.2)--(3.15,.8);					
								\draw[positive] (-1.8,0)--(-1.5,0);\draw[negative] (-1.5,0)--(-1.2,0);\draw[positive] (-1.2,0)--(-.9,0);								
								\draw[negative] (-.9,0)--(-.6,0);\draw[positive] (-.6,0)--(-.3,0);\draw[negative] (-.3,0)--(0,0);								
								\draw[positive] (0,0)--(1.5,0);
								\draw[negative] (1.5,0)--(1.8,0);\draw[positive] (1.8,0)--(2.1,0);\draw[negative] (2.1,0)--(2.4,0);
								\draw[positive] (2.4,0)--(2.7,0);\draw[negative] (2.7,0)--(3,0);\draw[positive] (3,0)--(3.3,0);
							\end{tikzpicture}
							\rightsquigarrow
							\begin{tikzpicture}[x=.5cm,baseline=.2cm]
								\draw[gray,dashed] (-1.65,1.2)--(-1.65,.8);					
								\draw[gray,dashed] (-1.35,1.2)--(-1.35,.8);					
								\draw[gray,dashed] (-.15,1.2)--(-.15,.8);					
								\draw[gray,dashed] (1.65,1.2)--(1.65,.8);					
								\draw[gray,dashed] (2.85,1.2)--(2.85,.8);					
								\draw[gray,dashed] (3.15,1.2)--(3.15,.8);					
								\fill[red] (.75,.4) circle (1.5pt); \draw[red] (.75,.4)--(.75,.8);					
								\fill[red] (.75,.8) circle (1.5pt);
								\draw[gray,localized] (-.15,0)--(-.15,.5); \draw[gray] (-.15,.5)--(-.15,.8);					
								\fill[gray] (-1.35,.5) circle (1.5pt); \draw[gray] (-1.35,.5)--(-1.35,.8);					
								\draw[gray,localized] (-1.65,0)--(-1.65,.5);\draw[gray] (-1.65,.5)--(-1.65,.8);
								\fill[gray] (1.65,.5) circle (1.5pt);\draw[gray] (1.65,.5)--(1.65,.8);					
								\draw[gray,localized] (2.85,0)--(2.85,.5);\draw[gray] (2.85,.5)--(2.85,.8);
								\fill[gray] (3.15,.5) circle(1.5pt);\draw[gray] (3.15,.5)--(3.15,.8);					
								\node at (-.75,.2) {\dots};\node at (2.25,.2) {\dots};
								\draw[gray] (-2.1,0)--(3.6,0);
							\end{tikzpicture}
							=w(\alpha_{\re{s}})
							\begin{tikzpicture}[x=.5cm,baseline=.2cm]
								\draw[gray,dashed] (-1.65,1.2)--(-1.65,.8);					
								\draw[gray,dashed] (-1.35,1.2)--(-1.35,.8);					
								\draw[gray,dashed] (-.15,1.2)--(-.15,.8);					
								\draw[gray,dashed] (1.65,1.2)--(1.65,.8);					
								\draw[gray,dashed] (2.85,1.2)--(2.85,.8);					
								\draw[gray,dashed] (3.15,1.2)--(3.15,.8);					
								\draw[gray,localized] (-.15,0)--(-.15,.5); \draw[gray] (-.15,.5)--(-.15,.8);					
								\fill[gray] (-1.35,.5) circle (1.5pt); \draw[gray] (-1.35,.5)--(-1.35,.8);					
								\draw[gray,localized] (-1.65,0)--(-1.65,.5);\draw[gray] (-1.65,.5)--(-1.65,.8);
								\fill[gray] (1.65,.5) circle (1.5pt);\draw[gray] (1.65,.5)--(1.65,.8);					
								\draw[gray,localized] (2.85,0)--(2.85,.5);\draw[gray] (2.85,.5)--(2.85,.8);
								\fill[gray] (3.15,.5) circle(1.5pt);\draw[gray] (3.15,.5)--(3.15,.8);					
								\node at (-.75,.2) {\dots};\node at (2.25,.2) {\dots};
								\draw[gray] (-2.1,0)--(3.6,0);
								\node[font=\small] at (-.9,-.2) {$\underbrace{\phantom{aaaaa}}_{w\in W}$};
							\end{tikzpicture}
						\]
						where we have slid the polynomial $\alpha_s$ across the squiggly lines to the left, using
						the third relation in \eqref{eq_localrel}.
						Now consider the enriched subexpression $\tilde{\eta}$ which is the same as $\eta$ except 
						$\tilde{\eta}_k=\emptyset$, in particular $\tilde{\eta}\sim\eta$. The remaining part of the diagram is nothing but 
						$(L_{\epsilon,\epsilon'})^{\tilde{\eta},\eta'}$. Hence $\phi_i^{\eta,\eta'}=w(\alpha_s)(L_{\epsilon,\epsilon'})^{\tilde{\eta},\eta'}$.
						By Proposition \ref{pro_uppertriangext} (\textit{i}), this is zero unless $\tilde{\eta}\lesssim\epsilon$ and $\eta'\lesssim\epsilon'$.
					\item If $\epsilon_k\neq \emptyset$, we can suppose $\eta_k=\emptyset$ (otherwise $\phi_i^{\eta,\eta'}=0$). In this case we are just computing 
						$(L_{\epsilon,\epsilon'})^{\tilde{\eta},\eta'}$ with $\tilde{\eta}$ equal to $\eta$ except that $\tilde{\eta}_k=0$ (in particular $\tilde{\eta}\sim\eta$). 
						Then we can conclude $\eta\sim \tilde{\eta}\lesssim \epsilon$ by Proposition \ref{pro_uppertriangext} (\textit{i}).
				\end{enumerate}				
				This proves the claim. Now, suppose that there are $\zeta$ and $\zeta'$ with 
				$a_{\zeta,\zeta'}\neq 0$ and such that $\zeta$ does not satisfy $\zeta\lesssim \epsilon$ (the argument for $\zeta'$ is the same).
				We can take $\zeta$ maximal with these properties. Then localizing at $\eta=\zeta$ and $\eta'=\zeta'$ gives a contradiction:
				the left hand side is zero by the claim and the right hand side is nonzero by Proposition \ref{pro_uppertriangext} (\textit{ii}) and the maximality of $\zeta$.
			\end{proof}
			By an abuse of language, we will say that two double leaves $L_{\epsilon,\epsilon'}$ and $L_{\zeta,\zeta'}$ are path equivalent
			if $\epsilon\sim \zeta$ and $\epsilon'\sim\zeta'$.
			\begin{rmk}\label{rmk_firstfiltr}
				Let $\{\LL_i\}$ be the set of path equivalence classes of double leaves, numbered in such a way that
				if $L_{\epsilon,\epsilon'}\in\LL_i$ and $L_{\zeta,\zeta'}\in \LL_j$ then $\epsilon\lesssim\zeta$ and $\epsilon'\lesssim\zeta'$ implies $i\le j$.
				Then by Theorem \ref{thm_monotonicity} the spans $N_i$ of all double leaves in $\cup_{j\le i} \LL_j$ are dg submodules forming a filtration.
			\end{rmk}
			One can refine path equivalence by taking decorations into account.
			We say that $\epsilon$ and $\zeta$ are \emph{strongly path equivalent}, and write
			$\epsilon\approx\zeta$, if $\epsilon\sim\zeta$ and, whenever $\epsilon_k\neq\zeta_k$, we have $\epsilon_k,\zeta_k\in\{ U0,U\emptyset \}$
			and they correspond to black boundary points. A strong path equivalence class is \emph{trivial} if it
			consists of a single element.
			\begin{rmk}\label{rmk_finfiltr}
				Consider the span of a (non strong) path equivalence class of double leaves.
				This forms a dg subquotient of the morphism space. Here the induced differential 
				can only decrease the number of white $U0$, black $D\emptyset$ and white $D0$ 
				by turning them into $U\emptyset$, $D0$ and $D\emptyset$ respectively. 
				This implies that one can number all the 
				strong path equivalence classes in such a way to obtain a refinement of the filtration, obtained in Remark \ref{rmk_firstfiltr},
				of the whole morphism space by dg submodules.			
			\end{rmk}
			The following technical result, which is a corollary of Theorem \ref{thm_monotonicity} and 
			Remark \ref{rmk_finfiltr} will be an important tool later.
			\begin{cor}\label{cor_N}
				Consider a morphism space $\Homb(\Roudg{\uom},\Roudg{\uom'})$ and
				let $N$ be a dg submodule spanned by a set of double leaves maps 
				which is a union of non trivial strong path equivalence classes.
				Then $N$ is contractible.
			\end{cor}
			\begin{proof}
				We consider the filtration of Remark \ref{rmk_finfiltr}.
				The subquotients $N_{i+1}/N_{i}$ are spanned by the double leaves maps from a 
				single strong path equivalence class. 
				In each subquotient the differential acts only on the black $U\emptyset$'s. 
				Then this dg module is isomorphic to $(R\overset{\id}{\rightarrow} R)^{\otimes n_i}$ 
				for some $n_i\ge 1$, as the class is not trivial. 
				Hence it is contractible. But then the whole of $N$ is contractible.
			\end{proof}
		\begin{exa}
			Here is an example of a subquotient as above:
			\begin{center}
				\begin{tikzpicture}[every node/.style={inner sep=0.8cm}]
					\node (a) at (.75,0.5) {}; \node (b1) at (3.25,2) {}; \node (b2) at (3.25,0.5) {}; \node (b3) at (3.25,-1) {};
					\node (c1) at (5.75,2) {}; \node (c2) at (5.75,0.5) {}; \node (c3) at (5.75,-1) {};
					\node (d) at (8.25,0.5) {};
					\draw[-latex] (a)--(b1);
					\draw[-latex] (a)--(b2);
					\draw[-latex] (a)--(b3);
					\draw[-latex] (b1)--(c1);
					\draw[-latex] (b1)--(c2);
					\draw[-latex] (b2)--(c1);
					\draw[-latex] (b2)--(c3);
					\draw[-latex] (b3)--(c2);
					\draw[-latex] (b3)--(c3);
					\draw[-latex] (c1)--(d);
					\draw[-latex] (c2)--(d);
					\draw[-latex] (c3)--(d);
					\begin{scope}[x=.3cm,y=.5cm]
						\draw[red] (1,0)..controls (1,.5).. (2,1)--(2,2); \draw[red] (3,0)..controls (3,.5)..(2,1);
						\draw[blue] (4,2)--(4,1.5);\fill[blue] (4,1.5) circle (1.5pt);
						\draw[positive] (0.5,0)--(2.5,0);\draw[negative] (2.5,0)--(3.5,0); \draw[positive] (3.5,0)--(4.5,0);
						\draw[negative] (0.5,2)--(2.5,2);\draw[positive] (2.5,2)--(3.5,2); \draw[negative] (3.5,2)--(4.5,2);
						\pic[blue,scale=.8,transform shape] at (2,0) {patch={1}};
						\pic[green,scale=.8,transform shape] at (4,0) {patch={1}};
						\pic[green,scale=.8,transform shape] at (3,2) {patch={1}};
					\end{scope}
					\begin{scope}[xshift=2.5cm,yshift=1.5cm,x=.3cm,y=.5cm]
						\draw[red] (1,0)..controls (1,.5).. (2,1)--(2,2); \draw[red] (3,0)..controls (3,.5)..(2,1);
						\draw[blue] (4,2)--(4,1.5);\fill[blue] (4,1.5) circle (1.5pt);
						\draw[blue] (2,0)--(2,.5);\fill[blue] (2,0.5) circle (1.5pt);
						\draw[positive] (0.5,0)--(2.5,0);\draw[negative] (2.5,0)--(3.5,0); \draw[positive] (3.5,0)--(4.5,0);
						\draw[negative] (0.5,2)--(2.5,2);\draw[positive] (2.5,2)--(3.5,2); \draw[negative] (3.5,2)--(4.5,2);
						\pic[green,scale=.8,transform shape] at (4,0) {patch={1}};
						\pic[green,scale=.8,transform shape] at (3,2) {patch={1}};
					\end{scope}
					\begin{scope}[xshift=2.5cm,yshift=0,x=.3cm,y=.5cm]
						\draw[red] (1,0)..controls (1,.5).. (2,1)--(2,2); \draw[red] (3,0)..controls (3,.5)..(2,1);
						\draw[blue] (4,2)--(4,1.5);\fill[blue] (4,1.5) circle (1.5pt);
						\draw[green] (4,0)--(4,0.5);\fill[green] (4,0.5) circle (1.5pt);
						\draw[positive] (0.5,0)--(2.5,0);\draw[negative] (2.5,0)--(3.5,0); \draw[positive] (3.5,0)--(4.5,0);
						\draw[negative] (0.5,2)--(2.5,2);\draw[positive] (2.5,2)--(3.5,2); \draw[negative] (3.5,2)--(4.5,2);
						\pic[blue,scale=.8,transform shape] at (2,0) {patch={1}};
						\pic[green,scale=.8,transform shape] at (3,2) {patch={1}};
					\end{scope}
					\begin{scope}[xshift=2.5cm,yshift=-1.5cm,x=.3cm,y=.5cm]
						\draw[red] (1,0)..controls (1,.5).. (2,1)--(2,2); \draw[red] (3,0)..controls (3,.5)..(2,1);
						\draw[blue] (4,2)--(4,1.5);\fill[blue] (4,1.5) circle (1.5pt);
						\draw[green] (3,2)--(3,1.5);\fill[green] (3,1.5) circle (1.5pt);
						\draw[positive] (0.5,0)--(2.5,0);\draw[negative] (2.5,0)--(3.5,0); \draw[positive] (3.5,0)--(4.5,0);
						\draw[negative] (0.5,2)--(2.5,2);\draw[positive] (2.5,2)--(3.5,2); \draw[negative] (3.5,2)--(4.5,2);
						\pic[blue,scale=.8,transform shape] at (2,0) {patch={1}};
						\pic[green,scale=.8,transform shape] at (4,0) {patch={1}};
					\end{scope}
					\begin{scope}[xshift=5cm,yshift=1.5cm,x=.3cm,y=.5cm]
						\draw[red] (1,0)..controls (1,.5).. (2,1)--(2,2); \draw[red] (3,0)..controls (3,.5)..(2,1);
						\draw[blue] (4,2)--(4,1.5);\fill[blue] (4,1.5) circle (1.5pt);
						\draw[blue] (2,0)--(2,0.5);\fill[blue] (2,0.5) circle (1.5pt);
						\draw[green] (4,0)--(4,0.5);\fill[green] (4,0.5) circle (1.5pt);
						\draw[positive] (0.5,0)--(2.5,0);\draw[negative] (2.5,0)--(3.5,0); \draw[positive] (3.5,0)--(4.5,0);
						\draw[negative] (0.5,2)--(2.5,2);\draw[positive] (2.5,2)--(3.5,2); \draw[negative] (3.5,2)--(4.5,2);
						\pic[green,scale=.8,transform shape] at (3,2) {patch={1}};
					\end{scope}
					\begin{scope}[xshift=5cm,yshift=0,x=.3cm,y=.5cm]
						\draw[red] (1,0)..controls (1,.5).. (2,1)--(2,2); \draw[red] (3,0)..controls (3,.5)..(2,1);
						\draw[blue] (4,2)--(4,1.5);\fill[blue] (4,1.5) circle (1.5pt);
						\draw[blue] (2,0)--(2,0.5);\fill[blue] (2,0.5) circle (1.5pt);
						\draw[green] (3,2)--(3,1.5);\fill[green] (3,1.5) circle (1.5pt);
						\draw[positive] (0.5,0)--(2.5,0);\draw[negative] (2.5,0)--(3.5,0); \draw[positive] (3.5,0)--(4.5,0);
						\draw[negative] (0.5,2)--(2.5,2);\draw[positive] (2.5,2)--(3.5,2); \draw[negative] (3.5,2)--(4.5,2);
						\pic[green,scale=.8,transform shape] at (4,0) {patch={1}};
					\end{scope}
					\begin{scope}[xshift=5cm,yshift=-1.5cm,x=.3cm,y=.5cm]
						\draw[red] (1,0)..controls (1,.5).. (2,1)--(2,2); \draw[red] (3,0)..controls (3,.5)..(2,1);
						\draw[blue] (4,2)--(4,1.5);\fill[blue] (4,1.5) circle (1.5pt);
						\draw[blue] (4,0)--(4,0.5);\fill[blue] (4,0.5) circle (1.5pt);
						\draw[blue] (3,2)--(3,1.5);\fill[blue] (3,1.5) circle (1.5pt);
						\draw[positive] (0.5,0)--(2.5,0);\draw[negative] (2.5,0)--(3.5,0); \draw[positive] (3.5,0)--(4.5,0);
						\draw[negative] (0.5,2)--(2.5,2);\draw[positive] (2.5,2)--(3.5,2); \draw[negative] (3.5,2)--(4.5,2);
						\pic[blue,scale=.8,transform shape] at (2,0) {patch={1}};
					\end{scope}
					\begin{scope}[xshift=7.5cm,x=.3cm,y=.5cm]
						\draw[red] (1,0)..controls (1,.5).. (2,1)--(2,2); \draw[red] (3,0)..controls (3,.5)..(2,1);
						\draw[blue] (4,2)--(4,1.5);\fill[blue] (4,1.5) circle (1.5pt);
						\draw[blue] (2,0)--(2,0.5);\fill[blue] (2,0.5) circle (1.5pt);
						\draw[green] (4,0)--(4,0.5);\fill[green] (4,0.5) circle (1.5pt);
						\draw[green] (3,2)--(3,1.5);\fill[green] (3,1.5) circle (1.5pt);
						\draw[positive] (0.5,0)--(2.5,0);\draw[negative] (2.5,0)--(3.5,0); \draw[positive] (3.5,0)--(4.5,0);
						\draw[negative] (0.5,2)--(2.5,2);\draw[positive] (2.5,2)--(3.5,2); \draw[negative] (3.5,2)--(4.5,2);
					\end{scope}
				\end{tikzpicture}				
			\end{center}
			In this case it is isomorphic to $(R\overset{\id}{\rightarrow}R)^{\otimes 3}$.
		\end{exa}
		\begin{rmk}\label{rmk_blackU0}
			If one has a dg submodule $N$ as in the Corollary, which is spanned by a union of strong path equivalence classes,
			a necessary and sufficient condition for these classes to be non trivial is that each 
			double leaves map contains at least one black $U0$ or one black $U\emptyset$. In fact, suppose that
			a double leaves map contains a black $U0$, then in the same class there has to be the double leaves map with that 
			symbol replaced by $U\emptyset$, and vice versa.
		\end{rmk}
		\begin{rmk}
			One could define strong path equivalence by replacing $U$ with $D$ and ``black'' with
			``white'' and one would obtain an analogous corollary.
		\end{rmk}
			We now prove Proposition \ref{pro_braidRou} in a completely combinatorial way.
		\subsection{Inverse relation}		
		The first property, namely $F_s F_s^{-1}\cong \un$, can be easily proven via \emph{Gaussian 
		elimination} for complexes (see \cite{BarNat}). One can find explicit homotopy equivalences
		in \cite{EK}. Here we will describe these equivalences in terms of diagrams with patches.
		\begin{proof}[Proof of Proposition \ref{pro_braidRou} \ref{item_propinvers}]
			For any $s\in S$, consider:
			\begin{align}\label{eq_inversebraid}
				&\begin{tikzcd}[ampersand replacement=\&]
					F_sF_{s}^{-1} \ar[r,shift left,"\eta^-"] \& \un[0]\ar[l,shift left,"\epsilon^+"]
				\end{tikzcd},&
				&\begin{tikzcd}[ampersand replacement=\&]
					F_{s}^{-1}F_s \ar[r,shift left,"\eta^+"] \& \un[0]\ar[l,shift left,"\epsilon^-"]				
				\end{tikzcd},&
			\end{align}
			where
			\begin{align}
				&\epsilon^+=
				\begin{tikzpicture}[x=.5cm,y=.4cm,baseline=.3cm]
					\draw[gray](0.5,0)--(2.5,0); \draw[negative](0.5,2)--(1.5,2);\draw[positive](1.5,2)--(2.5,2);
					\pic[red,scale=.8,transform shape] at (1,2) {patch={1}}; \pic[red,scale=.8,transform shape] at (2,2) {patch={1}};
				\end{tikzpicture}
				+
				\begin{tikzpicture}[x=.5cm,y=.4cm,baseline=.3cm]
					\draw[red] (1,2) ..controls (1,.5) and (2,.5).. (2,2);
					\draw[gray](0.5,0)--(2.5,0); \draw[negative](0.5,2)--(1.5,2);\draw[positive](1.5,2)--(2.5,2);
				\end{tikzpicture}&
				&\eta^+=
				\begin{tikzpicture}[x=.5cm,y=.4cm,baseline=.3cm]
					\draw[gray](0.5,2)--(2.5,2); \draw[negative](0.5,0)--(1.5,0);\draw[positive](1.5,0)--(2.5,0);
					\pic[red,scale=.8,transform shape] at (1,0) {patch={1}}; \pic[red,scale=.8,transform shape] at (2,0) {patch={1}};
				\end{tikzpicture}
				+
				\begin{tikzpicture}[x=.5cm,y=.4cm,baseline=.3cm]
					\draw[red] (1,0) ..controls (1,1.5) and (2,1.5).. (2,0);
					\draw[gray](0.5,2)--(2.5,2); \draw[negative](0.5,0)--(1.5,0);\draw[positive](1.5,0)--(2.5,0);
				\end{tikzpicture}\\[1em]
				&\epsilon^-=
				\begin{tikzpicture}[x=.5cm,y=.4cm,baseline=.3cm]
					\draw[gray](0.5,0)--(2.5,0); \draw[positive](0.5,2)--(1.5,2);\draw[negative](1.5,2)--(2.5,2);
					\pic[red,scale=.8,transform shape] at (1,2) {patch={1}}; \pic[red,scale=.8,transform shape] at (2,2) {patch={1}};
				\end{tikzpicture}
				-
				\begin{tikzpicture}[x=.5cm,y=.4cm,baseline=.3cm]
					\draw[red] (1,2) ..controls (1,.5) and (2,.5).. (2,2);
					\draw[gray](0.5,0)--(2.5,0); \draw[positive](0.5,2)--(1.5,2);\draw[negative](1.5,2)--(2.5,2);
				\end{tikzpicture}&
				&\eta^-=
				\begin{tikzpicture}[x=.5cm,y=.4cm,baseline=.3cm]
					\draw[gray](0.5,2)--(2.5,2); \draw[positive](0.5,0)--(1.5,0);\draw[negative](1.5,0)--(2.5,0);
					\pic[red,scale=.8,transform shape] at (1,0) {patch={1}}; \pic[red,scale=.8,transform shape] at (2,0) {patch={1}};
				\end{tikzpicture}
				-
				\begin{tikzpicture}[x=.5cm,y=.4cm,baseline=.3cm]
					\draw[red] (1,0) ..controls (1,1.5) and (2,1.5).. (2,0);
					\draw[gray](0.5,2)--(2.5,2); \draw[positive](0.5,0)--(1.5,0);\draw[negative](1.5,0)--(2.5,0);
				\end{tikzpicture}&
			\end{align}
			One can easily obtain $\eta^-\epsilon^+=\id_{\un[0]}$
			from the rules of composition of dg diagrams. On the other hand we have:
			\[
				\id_{F_{s}F_{s}^{-1}}-\epsilon^+\eta^- = %
				\begin{tikzpicture}[x=.5cm,y=.4cm,baseline=.3cm]
					\draw[red](1,0)--(1,2);\draw[red] (2,0)--(2,2);
					\draw[positive] (0.5,0)--(1.5,0);\draw[negative] (1.5,0)--(2.5,0);
					\draw[negative] (0.5,2)--(1.5,2);\draw[positive] (1.5,2)--(2.5,2);
				\end{tikzpicture}
				+
				\begin{tikzpicture}[x=.5cm,y=.4cm,baseline=.3cm]
					\draw[red](1,0)--(1,2);
					\draw[positive] (0.5,0)--(1.5,0);\draw[negative] (1.5,0)--(2.5,0);
					\draw[negative] (0.5,2)--(1.5,2);\draw[positive] (1.5,2)--(2.5,2);
					\pic[red,scale=.8,transform shape] at (2,0) {patch={1}}; \pic[red,scale=.8,transform shape] at (2,2) {patch={1}};
				\end{tikzpicture}
				+
				\begin{tikzpicture}[x=.5cm,y=.4cm,baseline=.3cm]
					\draw[red](2,0)--(2,2);
					\draw[positive] (0.5,0)--(1.5,0);\draw[negative] (1.5,0)--(2.5,0);
					\draw[negative] (0.5,2)--(1.5,2);\draw[positive] (1.5,2)--(2.5,2);
					\pic[red,scale=.8,transform shape] at (1,0) {patch={1}}; \pic[red,scale=.8,transform shape] at (1,2) {patch={1}};
				\end{tikzpicture}
				-
				\begin{tikzpicture}[x=.5cm,y=.4cm,baseline=.3cm]
					\draw[red](1,2) ..controls (1,.5) and (2,.5).. (2,2);
					\draw[positive] (0.5,0)--(1.5,0);\draw[negative] (1.5,0)--(2.5,0);
					\draw[negative] (0.5,2)--(1.5,2);\draw[positive] (1.5,2)--(2.5,2);
					\pic[red,scale=.8,transform shape] at (1,0) {patch={1}}; \pic[red,scale=.8,transform shape] at (2,0) {patch={1}};
				\end{tikzpicture}				
				+
				\begin{tikzpicture}[x=.5cm,y=.4cm,baseline=.3cm]
					\draw[red](1,0) ..controls (1,1.5) and (2,1.5).. (2,0);
					\draw[positive] (0.5,0)--(1.5,0);\draw[negative] (1.5,0)--(2.5,0);
					\draw[negative] (0.5,2)--(1.5,2);\draw[positive] (1.5,2)--(2.5,2);
					\pic[red,scale=.8,transform shape] at (1,2) {patch={1}}; \pic[red,scale=.8,transform shape] at (2,2) {patch={1}};
				\end{tikzpicture}				
				+
				\begin{tikzpicture}[x=.5cm,y=.4cm,baseline=.3cm]
					\draw[red](1,0) ..controls (1,1) and (2,1).. (2,0);
					\draw[red](1,2) ..controls (1,1) and (2,1).. (2,2);
					\draw[positive] (0.5,0)--(1.5,0);\draw[negative] (1.5,0)--(2.5,0);
					\draw[negative] (0.5,2)--(1.5,2);\draw[positive] (1.5,2)--(2.5,2);
				\end{tikzpicture}.
			\]
			We want to show that this is exact, and therefore null-homotopic.
			Notice that, by \eqref{sliding}, \eqref{associativity} and \eqref{dotline}, we have:
			\begin{align*}
				&\begin{tikzpicture}[x=.5cm,y=.5cm,baseline=.4cm]
					\draw[red] (0.5,0)--(0.5,2);\draw[red] (1.5,0)--(1.5,2);
					\draw[positive] (0.1,0)--(1,0); \draw[negative] (1,0)--(1.9,0); 
					\draw[negative] (0.1,2)--(1,2); \draw[positive] (1,2)--(1.9,2); 
				\end{tikzpicture}
				=
				\begin{tikzpicture}[x=.5cm,y=.5cm,baseline=.4cm]
					\draw[red] (0.5,0)--(0.5,2);\draw[red] (1.5,0)--(1.5,2);\draw[red](0.5,1)--(1.5,1);
					\node[inner sep=2] at (1,1.4) {\small $\delta_s$};
					\draw[positive] (0.1,0)--(1,0); \draw[negative] (1,0)--(1.9,0); 
					\draw[negative] (0.1,2)--(1,2); \draw[positive] (1,2)--(1.9,2); 
				\end{tikzpicture}
				-
				\begin{tikzpicture}[x=.6cm,y=.5cm,baseline=.4cm]
					\draw[red] (0.4,0)--(0.4,2);\draw[red] (1.6,0)--(1.6,2);\draw[red](0.4,1)--(1.6,1);
					\node[inner sep=2] at (1,0.6) {\small $s(\delta_s)$};
					\draw[positive] (0.1,0)--(1,0); \draw[negative] (1,0)--(1.9,0); 
					\draw[negative] (0.1,2)--(1,2); \draw[positive] (1,2)--(1.9,2); 
				\end{tikzpicture}
				=
				\begin{tikzpicture}[x=.5cm,y=.5cm,baseline=.4cm]
					\draw[red](0.5,0)--(1,1);\draw[red](1.5,0)--(0.5,2); \draw[red](1.5,2)--(1.5,1.5);\fill[red](1.5,1.5) circle (1.5pt);
					\draw[positive] (0.1,0)--(1,0); \draw[negative] (1,0)--(1.9,0); 
					\draw[negative] (0.1,2)--(1,2); \draw[positive] (1,2)--(1.9,2); 
				\end{tikzpicture}
				+
				\begin{tikzpicture}[x=.5cm,y=.5cm,baseline=.4cm]
					\draw[red] (0.6,0)--(0.6,2);\draw[red] (1.4,0)--(1.4,2);\draw[red](0.6,1)--(1.4,1);
					\node[anchor=west] at (1.4,1) {\small $s(\delta_s)$};
					\draw[positive] (0.1,0)--(1,0); \draw[negative] (1,0)--(1.9,0); 
					\draw[negative] (0.1,2)--(1,2); \draw[positive] (1,2)--(1.9,2); 
				\end{tikzpicture}
				+
				\begin{tikzpicture}[x=.5cm,y=.5cm,baseline=.4cm]
					\draw[red](1.5,2)--(1,1);\draw[red](1.5,0)--(0.5,2); \draw[red](0.5,0)--(0.5,0.5);\fill[red](0.5,0.5) circle (1.5pt);
					\draw[positive] (0.1,0)--(1,0); \draw[negative] (1,0)--(1.9,0); 
					\draw[negative] (0.1,2)--(1,2); \draw[positive] (1,2)--(1.9,2); 
				\end{tikzpicture}
				-
				\begin{tikzpicture}[x=.5cm,y=.5cm,baseline=.4cm]
					\draw[red] (0.5,0)--(0.5,2);\draw[red] (1.5,0)--(1.5,2);\draw[red](0.5,1)--(1.5,1);
					\node[inner sep=2] at (0.2,1) {\small $\delta_s$};
					\draw[positive] (0.1,0)--(1,0); \draw[negative] (1,0)--(1.9,0); 
					\draw[negative] (0.1,2)--(1,2); \draw[positive] (1,2)--(1.9,2); 
				\end{tikzpicture}\\[1em]
				&\begin{tikzpicture}[x=.5cm,y=.5cm,baseline=.4cm]
					\draw[red](0.5,0) ..controls (0.5,1) and (1.5,1).. (1.5,0); \draw[red](0.5,2) ..controls (0.5,1) and (1.5,1).. (1.5,2);
					\draw[positive] (0.1,0)--(1,0); \draw[negative] (1,0)--(1.9,0); 
					\draw[negative] (0.1,2)--(1,2); \draw[positive] (1,2)--(1.9,2); 
				\end{tikzpicture}
				=
				-
				\begin{tikzpicture}[x=.5cm,y=.5cm,baseline=.4cm]
					\draw[red] (0.5,0)--(1,.7)--(1,1.3)--(0.5,2);\draw[red] (1.5,0)--(1,.7);\draw[red](1.5,2)--(1,1.3);
					\node[anchor=west] at (1,1) {\small $s(\delta_s)$};
					\draw[positive] (0.1,0)--(1,0); \draw[negative] (1,0)--(1.9,0); 
					\draw[negative] (0.1,2)--(1,2); \draw[positive] (1,2)--(1.9,2); 
				\end{tikzpicture}
				+
				\begin{tikzpicture}[x=.5cm,y=.5cm,baseline=.4cm]
					\draw[red] (0.5,0)--(1,.7)--(1,1.3)--(0.5,2);\draw[red] (1.5,0)--(1,.7);\draw[red](1.5,2)--(1,1.3);
					\node[inner sep=2] at (0.5,1) {\small $\delta_s$};
					\draw[positive] (0.1,0)--(1,0); \draw[negative] (1,0)--(1.9,0); 
					\draw[negative] (0.1,2)--(1,2); \draw[positive] (1,2)--(1.9,2); 
				\end{tikzpicture}	
				=-
				\begin{tikzpicture}[x=.5cm,y=.5cm,baseline=.4cm]
					\draw[red] (0.6,0)--(0.6,2);\draw[red] (1.4,0)--(1.4,2);\draw[red](0.6,1)--(1.4,1);
					\node[anchor=west] at (1.4,1) {\small $s(\delta_s)$};
					\draw[positive] (0.1,0)--(1,0); \draw[negative] (1,0)--(1.9,0); 
					\draw[negative] (0.1,2)--(1,2); \draw[positive] (1,2)--(1.9,2); 
				\end{tikzpicture}
				+
				\begin{tikzpicture}[x=.5cm,y=.5cm,baseline=.4cm]
					\draw[red] (0.5,0)--(0.5,2);\draw[red] (1.5,0)--(1.5,2);\draw[red](0.5,1)--(1.5,1);
					\node[inner sep=2] at (0.2,1) {\small $\delta_s$};
					\draw[positive] (0.1,0)--(1,0); \draw[negative] (1,0)--(1.9,0); 
					\draw[negative] (0.1,2)--(1,2); \draw[positive] (1,2)--(1.9,2); 
				\end{tikzpicture}				
			\end{align*}
			Hence the difference $\id_{F_{s}F_{s}^{-1}}-\epsilon^+\eta^-$ above becomes:
			\[
				\begin{tikzpicture}[x=.5cm,y=.4cm,baseline=.3cm]
					\draw[red](1,2)--(2,0);\draw[red] (1,0)--(1.5,1);\draw[red] (2,2)--(2,1.3);
					\fill[red](2,1.3) circle (1.5pt);
					\draw[positive] (0.5,0)--(1.5,0);\draw[negative] (1.5,0)--(2.5,0);
					\draw[negative] (0.5,2)--(1.5,2);\draw[positive] (1.5,2)--(2.5,2);
				\end{tikzpicture}
				+
				\begin{tikzpicture}[x=.5cm,y=.4cm,baseline=.3cm]
					\draw[red](1,2)--(2,0);\draw[red] (2,2)--(1.5,1);\draw[red] (1,0)--(1,0.7);
					\fill[red](1,0.7) circle (1.5pt);
					\draw[positive] (0.5,0)--(1.5,0);\draw[negative] (1.5,0)--(2.5,0);
					\draw[negative] (0.5,2)--(1.5,2);\draw[positive] (1.5,2)--(2.5,2);
				\end{tikzpicture}
				+
				\begin{tikzpicture}[x=.5cm,y=.4cm,baseline=.3cm]
					\draw[red](1,0)--(1,2);
					\draw[positive] (0.5,0)--(1.5,0);\draw[negative] (1.5,0)--(2.5,0);
					\draw[negative] (0.5,2)--(1.5,2);\draw[positive] (1.5,2)--(2.5,2);
					\pic[red,scale=.8,transform shape] at (2,0) {patch={1}}; \pic[red,scale=.8,transform shape] at (2,2) {patch={1}};
				\end{tikzpicture}
				+
				\begin{tikzpicture}[x=.5cm,y=.4cm,baseline=.3cm]
					\draw[red](2,0)--(2,2);
					\draw[positive] (0.5,0)--(1.5,0);\draw[negative] (1.5,0)--(2.5,0);
					\draw[negative] (0.5,2)--(1.5,2);\draw[positive] (1.5,2)--(2.5,2);
					\pic[red,scale=.8,transform shape] at (1,0) {patch={1}}; \pic[red,scale=.8,transform shape] at (1,2) {patch={1}};
				\end{tikzpicture}
				-
				\begin{tikzpicture}[x=.5cm,y=.4cm,baseline=.3cm]
					\draw[red](1,2) ..controls (1,.5) and (2,.5).. (2,2);
					\draw[positive] (0.5,0)--(1.5,0);\draw[negative] (1.5,0)--(2.5,0);
					\draw[negative] (0.5,2)--(1.5,2);\draw[positive] (1.5,2)--(2.5,2);
					\pic[red,scale=.8,transform shape] at (1,0) {patch={1}}; \pic[red,scale=.8,transform shape] at (2,0) {patch={1}};
				\end{tikzpicture}				
				+
				\begin{tikzpicture}[x=.5cm,y=.4cm,baseline=.3cm]
					\draw[red](1,0) ..controls (1,1.5) and (2,1.5).. (2,0);
					\draw[positive] (0.5,0)--(1.5,0);\draw[negative] (1.5,0)--(2.5,0);
					\draw[negative] (0.5,2)--(1.5,2);\draw[positive] (1.5,2)--(2.5,2);
					\pic[red,scale=.8,transform shape] at (1,2) {patch={1}}; \pic[red,scale=.8,transform shape] at (2,2) {patch={1}};
				\end{tikzpicture}				
			\]
			and one can check that this is:
			\[
				d\left(%
				\begin{tikzpicture}[x=.5cm,y=.4cm,baseline=.3cm]
					\draw[red](1,0)--(1.5,1);
					\draw[red](2,0)--(1,2);
					\draw[positive] (0.5,0)--(1.5,0);\draw[negative] (1.5,0)--(2.5,0); 
					\draw[negative] (0.5,2)--(1.5,2);\draw[positive] (1.5,2)--(2.5,2); 
					\pic[red,scale=.8,transform shape] at (2,2) {patch={1}};
				\end{tikzpicture}
				+
				\begin{tikzpicture}[x=.5cm,y=.4cm,baseline=.3cm]
					\draw[red](2,0)--(1,2);
					\draw[red](2,2)--(1.5,1);
					\draw[positive] (0.5,0)--(1.5,0);\draw[negative] (1.5,0)--(2.5,0); 
					\draw[negative] (0.5,2)--(1.5,2);\draw[positive] (1.5,2)--(2.5,2); 
					\pic[red,scale=.8,transform shape] at (1,0) {patch={1}};
				\end{tikzpicture}
				\right)
			\]
			One can do similar computations for the pair $(\eta^+,\epsilon^-)$. 		
			\end{proof}
		\subsection{Braid relation}\label{subs_braidrel}
		We now pass to the braid relation. 
		Let $s,t\in S$ such that $st$ has order $m<\infty$ in $W$. We introduce the following notations:
			\begin{align*}
				&w(s,t):=\underbrace{sts\dots}_{m\,\text{times}}=\underbrace{tst\dots}_{m\,\text{times}}\in W,& &&\\
				&\uw(s,t):=\underbrace{sts\dots}_{m\,\text{times}}\in \cword,& &\uom(s,t):=\underbrace{\sigma_s\sigma_t\sigma_s\dots}_{m\,\text{times}}\in \bword, & \\
				&\uw(t,s):=\underbrace{tst\dots}_{m\,\text{times}}\in \cword,& &\uom(t,s):=\underbrace{\sigma_t\sigma_s\sigma_t\dots}_{m\,\text{times}}\in \bword. &
			\end{align*}
		We will reduce the dg morphism space 
		$\Homb(F_{\uom(s,t)}^\bullet,F_{\uom(t,s)}^\bullet)$ by describing a large contractible dg submodule, whose
		quotient is just a copy of $R$ in degree $0$. 
		This will give a canonical homotopy equivalence 
		$F_{\uom(s,t)}^\bullet\rightarrow F_{\uom(t,s)}^\bullet$.		
		\begin{rmk}\label{rmk_sesBS}
			In the Hecke category we can consider the objects $B_{\uw(s,t)}$ and $B_{\uw(t,s)}$ and a double 
			leaves basis for the morphism space $H$ between the two. 
			Consider the subspace $N$ of morphisms factoring through shorter words.
			This is spanned, as a left $R$-module, 
			by all the double leaves maps except the one consisting of the 
			$(s,t)$-ar, which we call $\beta$. This corresponds to the pair of subexpressions 
			consisting of only 1's. 
			Then we have a short exact sequence of left $R$-modules:
			\[
				0\longrightarrow N \longrightarrow H \longrightarrow R \longrightarrow 0.
			\]
			If we see $R$ as the standard $R$-bimodule $R_{w(s,t)}$, 
			where the left structure is the usual one and
			the right structure is twisted by the element $w(s,t)\in W$, then this is actually a 
			short exact sequence of $R$-bimodules. In fact 
			$N$ is stable by left and right multiplication by polynomials and
			multiplying the $(s,t)$-ar on the left by $f\in R$ is the same, modulo lower terms, 
			as multiplying it by $w(s,t)(f)$
			on the right (as one can see via localization).
		\end{rmk}
		\begin{dfn}
			Let $N_{s,t}\subset \Homb(F^\bullet_{\uom(s,t)},F^\bullet_{\uom(t,s)})$ 
			be the span of all homogeneous morphisms factoring through words with less than 
			$m$ symbols. This in particular contains all morphisms with starting or ending subexpressions
			different from 11\dots 1.
		\end{dfn}
		\begin{pro}\label{pro_sesbraid}
			We have: 
			\begin{enumerate}
				\item the space $N_{s,t}$ is a dg submodule of $\Homb(F_{\uom(s,t)}^\bullet,F_{\uom(t,s)}^\bullet)$;
				\item the quotient is $R[0]$, in other words we have a short exact sequence of dg modules:
					\begin{equation}\label{eq_sesbraid}
						0\longrightarrow N_{s,t}\longrightarrow \Hom^{\bullet}(F_{\uom(s,t)}^\bullet,F_{\uom(t,s)}^\bullet)\longrightarrow R[0]\longrightarrow 0 ;					
					\end{equation}
				\item the dg-module $N_{s,t}$ is contractible, so the quotient map above
					is a homotopy equivalence.
			\end{enumerate}
		\end{pro}
		\begin{proof}
			\begin{enumerate}
				\item This is easy by the definition of the differential map: if a morphism factors through
					a shorter word then any morphism obtained after the sprouting of a dot or the 
					uprooting of a strand will still factor through that shorter word.
				\item\label{item_proofbraid} One can write a basis of double leaves maps for the whole morphism space such 
					that the only one which does not factor through shorter words is the $(s,t)$-ar. 
				\item By \ref{item_proofbraid} all the double leaves maps different from the $(s,t)$-ar
					are a basis of $N_{s,t}$. 
					This basis is clearly a union of strong path equivalence classes 
					(the class of the $(s,t)$-ar being trivial). Furthermore any enriched 
					subexpression different
					from 11\dots 1 of a reduced word (such as $\uw(s,t)$)
					has to contain at least one $U0$ or $U\emptyset$. In fact it is sufficient 
					to take the first entry which is not 1 from the 
					left: this is a 0 or a $\emptyset$ that has to be decorated with a 
					$U$ by reducedness. Furthermore, the bottom boundary is entirely black, 
					so we can conclude by Corollary \ref{cor_N} and Remark \ref{rmk_blackU0}.
%
			\end{enumerate}
		\end{proof}
		\begin{rmk}
			In the same way as in Remark \ref{rmk_sesBS}, one can prove that \eqref{eq_sesbraid} 
			is a short exact sequence of dg $R$-bimodules.
		\end{rmk}
		Now, taking cohomology at zero in the short exact sequence \eqref{eq_sesbraid} gives 
		an isomorphism:
		\[
			\Hom_{\Kb(\D)}(F_{\uom(s,t)}^\bullet, F_{\uom(t,s)}^\bullet)\overset{\sim}{\longrightarrow} R.
		\]
		This is given by taking the coefficient of the 
		$(s,t)$-ar in any closed representative in $\Cdg(\D)$. 
		Take $\gamma_{s,t}$ to be the morphism corresponding to $1\in R$.
		The following proves Proposition \ref{pro_braidRou} \ref{item_propbraid}.
		\begin{pro}\label{pro_gammast}
			The morphisms $\gamma_{s,t}$ and $\gamma_{t,s}$ are mutually inverse homotopy equivalences.
		\end{pro}
		\begin{proof}
			Notice that combining relations \eqref{2mtriv} and \eqref{eq_2mdot}, we obtain the following 
			equality (the picture is for $m_{s,t}$ even):
			\[
				\begin{tikzpicture}[baseline=.8cm,scale=.6, transform shape]
					\draw (0,0) -- (3,0); \draw (0,3) -- (3,3);
					\draw[red] (.5,0)--(1.5,1) ..controls (2.5,1.5).. (1.5,2)--(.5,3);
					\draw[red] (2,0)--(1.5,1) ..controls (1,1.5).. (1.5,2)--(2,3);
					\draw[blue] (2.5,0)--(1.5,1) ..controls (.5,1.5).. (1.5,2)--(2.5,3);
					\draw[blue] (1,0)--(1.5,1) ..controls (2,1.5).. (1.5,2)--(1,3);
					\node at (1.5,.3) {...};\node at (1.5,2.7) {...};\node at (1.5,1.5) {...};
				\end{tikzpicture}
				=
				\begin{tikzpicture}[baseline=.8cm,scale=.6]
					\draw (0,0) -- (3,0); \draw (0,3) -- (3,3);
					\draw[red](.5,0)--(.5,3); \draw[red](.5,1.5)--(1.5,1.5);\draw[blue](1,0)--(1,3);
					\draw[red](2,0)--(2,3);\draw[blue](2.5,1.5)--(1.5,1.5); \draw[blue](2.5,0)--(2.5,3);
					\node at (1.5,.3) {...};\node at (1.5,2.7) {...};
					\node[draw,circle,fill=white] at (1.5,1.5) {$\JW$};
				\end{tikzpicture}			
			\]
			One can show that this is the identity,
			modulo morphisms that factor through shorter Coxeter words: see \cite[Claim 7.1]{EW}. 
			Now we can repeat the argument of Proposition \ref{pro_sesbraid} for $\Homb(F_{\uom(s,t)}^\bullet,F_{\uom(s,t)}^\bullet)$
			and obtain that, up to a polynomial, the only nonzero morphism in $\Kb(\D)$ is the identity.
			Then the compsition above has to be homotopy equivalent to the identity.
		\end{proof}
		Let us now consider two braid words $\uom_1$ and $\uom_2$, such that $\uom_2$ is obtained from 
		$\uom_1$ by applying one of the relations of $B_W$. By appropriately tensoring with identity to 
		the left and to the right the morphisms \eqref{eq_inversebraid} or those in Proposition \ref{pro_gammast},
		we get a homotopy equivalence $\gamma_{\uom_1}^{\uom_2}$ 
		corresponding to the relation used.
		
		If more generally $\uom_2$ is obtained from $\uom_1$ via a sequence of relations then one gets
		a homotopy equivalence $\gamma_{\uom_1}^{\uom_2}$ given by the corresponding composition of 
		morphisms as above, which, a priori, depends on the choice of relations used.
		\begin{rmk}
			The morphisms $\gamma_{\uom_1}^{\uom_2}$ above are, up to sign, 
			the canonical isomorphisms defined by Rouquier \cite[\S\,9.3]{Rou_cat}. In fact consider a 
			braid word $\uom$. We have a homotopy equivalence 
			$\Hom^\bullet (F_{\uom}^\bullet,F_{\uom}^\bullet)\simeq R[0]$ compatible 
			with the ring structure (because null-homotopic morphisms form a two-sided ideal). Hence any loop
			of relations transforming $\uom$ into itself gives a homotopy equivalence that corresponds
			to an invertible element of $R$. Now taking $\Bbbk=\ZZ$ we deduce that this element is $\pm 1$.
%

			Another argument for canonicity is given in \cite[\S\,4.2]{Elias}.
		\end{rmk}
		\subsection{Explicit homotopy equivalences}
		The above gives a recipe for an explicit diagrammatic description of $\gamma_{s,t}$.
		In type $A$ such a description can be found in a paper by Elias and Krasner \cite{EK}.
		
		Let $\mathbb{L}_{s,t}$ be a basis of double leaves for the space $N_{s,t}$ and take 
		$\mathbb{L}_{s,t}^0$ to be the subset of maps of cohomological degree 0.
		Notice that, as the bottom line is entirely black and 
		the top line is entirely white, the cohomological degree can be zero only if the number of patches on bottom is the same as on top
		(see Remark \ref{rmk_ondgdiagrams} (\textit{i})). One can also restrict further to those maps
		of polynomial degree 0. Let $\beta$ be the $(s,t)$-ar morphism.
		Then, by \S\,\ref{subs_braidrel}, we know that:
		\[
			\gamma_{s,t}=\beta+\sum_{L\in \mathbb{L}_{s,t}^0} a_L L,
		\]
		and that it is the unique morphism of complexes up to homotopy satisfying these properties. Then
		it is sufficient to find any choice of coefficients $a_L$ such that:
		\[
			d\big( \beta+\sum_{L\in \mathbb{L}_{s,t}^0} a_L L \big)=0.
		\]
		So this reduces to a linear algebra problem. Here are the cases $m=2,3$.
		\begin{exa}
			For $m_{s,t}=2$ we obtain:
			\begin{equation}
				\gamma_{s,t}=%
				\begin{tikzpicture}[x=.5cm,y=.4cm,baseline=.3cm]
					\draw[red] (1,0)--(2,2);\draw[blue] (2,0)--(1,2);
					\pic at (0,0) {twoframe};
				\end{tikzpicture}+
				\begin{tikzpicture}[x=.5cm,y=.4cm,baseline=.3cm]
					\def\scapat{.8}
					\draw[red] (1,0)--(2,2);
					\pic at (0,0) {twoframe};
					\pic[blue] at (2,0) {patch={\scapat}};
					\pic[blue] at (1,2) {patch={\scapat}};
				\end{tikzpicture}
				+
				\begin{tikzpicture}[x=.5cm,y=.4cm,baseline=.3cm]
					\def\scapat{.8}
					\draw[blue] (2,0)--(1,2);
					\pic at (0,0) {twoframe};
					\pic[red] at (1,0) {patch={\scapat}};
					\pic[red] at (2,2) {patch={\scapat}};
				\end{tikzpicture}%
				-
				\begin{tikzpicture}[x=.5cm,y=.4cm,baseline=.3cm]
					\def\scapat{.8}
					\pic at (0,0) {twoframe};
					\pic[red] at (1,0) {patch=\scapat};
					\pic[red] at (2,2) {patch=\scapat};
					\pic[blue] at (1,2) {patch=\scapat};
					\pic[blue] at (2,0) {patch=\scapat};
				\end{tikzpicture}
			\end{equation}
			For $m_{s,t}=3$, we get:
			\begin{multline}
				\gamma_{s,t}=
				\begin{tikzpicture}[x=.35cm,y=.4cm,baseline=.3cm]
					\draw[red] (2,0)--(2,1)--(1,2);\draw[red] (3,2)--(2,1);
					\draw[blue] (1,0)--(2,1)--(3,0);\draw[blue] (2,2)--(2,1);
					\pic at (0,0) {threeframe};
				\end{tikzpicture}			
				+
				\begin{tikzpicture}[x=.35cm,y=.4cm,baseline=.3cm]
					\draw[red] (1,2)--(2,0);
					\draw[blue] (2,2)--(3,0);
					\pic at (0,0) {threeframe};
					\pic[blue,scale=.8,transform shape] at (1,0) {patch={1}};
					\pic[red,scale=.8,transform shape] at (3,2) {patch={1}};
				\end{tikzpicture}			
				+
				\begin{tikzpicture}[x=.35cm,y=.4cm,baseline=.3cm]
					\draw[red] (3,2)--(2,0);
					\draw[blue] (2,2)--(1,0);
					\pic at (0,0) {threeframe};
					\pic[blue,scale=.8,transform shape] at (3,0) {patch={1}};
					\pic[red,scale=.8,transform shape] at (1,2) {patch={1}};
				\end{tikzpicture}			
				+
				\begin{tikzpicture}[x=.35cm,y=.4cm,baseline=.3cm]
					\draw[red] (1,2)--(1,1.2);\fill[red] (1,1.2) circle (1.5pt);
					\draw[blue] (2,2)--(2,1)--(1,0);\draw[blue] (2,1)--(3,0);
					\pic at (0,0) {threeframe};
					\pic[red,scale=.8,transform shape] at (2,0) {patch={1}};
					\pic[red,scale=.8,transform shape] at (3,2) {patch={1}};
				\end{tikzpicture}			
				+
				\begin{tikzpicture}[x=.35cm,y=.4cm,baseline=.3cm]
					\draw[red] (3,2)--(3,1.2);\fill[red] (3,1.2) circle (1.5pt);
					\draw[blue] (2,2)--(2,1)--(1,0);\draw[blue] (2,1)--(3,0);
					\pic at (0,0) {threeframe};
					\pic[red,scale=.8,transform shape] at (2,0) {patch={1}};
					\pic[red,scale=.8,transform shape] at (1,2) {patch={1}};
				\end{tikzpicture}+\\
				+a
				\begin{tikzpicture}[x=.35cm,y=.4cm,baseline=.3cm]
					\draw[blue] (1,0)--(1,0.8);\fill[blue] (1,0.8) circle (1.5pt);
					\draw[red] (2,0)--(2,1)--(1,2);\draw[red] (2,1)--(3,2);
					\pic at (0,0) {threeframe};
					\pic[blue,scale=.8,transform shape] at (2,2) {patch={1}};
					\pic[blue,scale=.8,transform shape] at (3,0) {patch={1}};
				\end{tikzpicture}			
				+(1-a)
				\begin{tikzpicture}[x=.35cm,y=.4cm,baseline=.3cm]
					\draw[blue] (3,0)--(3,0.8);\fill[blue] (3,0.8) circle (1.5pt);
					\draw[red] (2,0)--(2,1)--(1,2);\draw[red] (2,1)--(3,2);
					\pic at (0,0) {threeframe};
					\pic[blue,scale=.8,transform shape] at (2,2) {patch={1}};
					\pic[blue,scale=.8,transform shape] at (1,0) {patch={1}};
				\end{tikzpicture}				
				+
				\begin{tikzpicture}[x=.35cm,y=.4cm,baseline=.3cm]
					\draw[red] (1,2)..controls (1,1) and (3,1).. (3,2);
					\draw[blue] (1,0)..controls (1,1) and (3,1).. (3,0);
					\pic at (0,0) {threeframe};
					\pic[red,scale=.8,transform shape] at (2,0) {patch={1}};
					\pic[blue,scale=.8,transform shape] at (2,2) {patch={1}};
				\end{tikzpicture}			
				-
				\begin{tikzpicture}[x=.35cm,y=.4cm,baseline=.3cm]
					\draw[blue] (2,2)--(1,0);
					\pic at (0,0) {threeframe};
					\pic[blue,scale=.8,transform shape] at (3,0) {patch={1}};
					\pic[red,scale=.8,transform shape] at (2,0) {patch={1}};
					\pic[red,scale=.8,transform shape] at (1,2) {patch={1}};
					\pic[red,scale=.8,transform shape] at (3,2) {patch={1}};
				\end{tikzpicture}
				-
				\begin{tikzpicture}[x=.35cm,y=.4cm,baseline=.3cm]
					\draw[blue] (2,2)--(3,0);
					\pic at (0,0) {threeframe};
					\pic[blue,scale=.8,transform shape] at (1,0) {patch={1}};
					\pic[red,scale=.8,transform shape] at (2,0) {patch={1}};
					\pic[red,scale=.8,transform shape] at (1,2) {patch={1}};
					\pic[red,scale=.8,transform shape] at (3,2) {patch={1}};
				\end{tikzpicture}+\\
				-a
				\begin{tikzpicture}[x=.35cm,y=.4cm,baseline=.3cm]
					\draw[red] (1,2)--(2,0);
					\pic at (0,0) {threeframe};
					\pic[blue,scale=.8,transform shape] at (2,2) {patch={1}};
					\pic[red,scale=.8,transform shape] at (3,2) {patch={1}};
					\pic[blue,scale=.8,transform shape] at (1,0) {patch={1}};
					\pic[blue,scale=.8,transform shape] at (3,0) {patch={1}};
				\end{tikzpicture}			
				-(1-a)
				\begin{tikzpicture}[x=.35cm,y=.4cm,baseline=.3cm]
					\draw[red] (3,2)--(2,0);
					\pic at (0,0) {threeframe};
					\pic[blue,scale=.8,transform shape] at (2,2) {patch={1}};
					\pic[red,scale=.8,transform shape] at (1,2) {patch={1}};
					\pic[blue,scale=.8,transform shape] at (1,0) {patch={1}};
					\pic[blue,scale=.8,transform shape] at (3,0) {patch={1}};
				\end{tikzpicture}			
				+
				\begin{tikzpicture}[x=.35cm,y=.4cm,baseline=.3cm]
					\pic at (0,0) {threeframe};
					\pic[blue,scale=.8,transform shape] at (1,0) {patch={1}};
					\pic[blue,scale=.8,transform shape] at (3,0) {patch={1}};
					\pic[blue,scale=.8,transform shape] at (2,2) {patch={1}};
					\pic[red,scale=.8,transform shape] at (2,0) {patch={1}};
					\pic[red,scale=.8,transform shape] at (1,2) {patch={1}};
					\pic[red,scale=.8,transform shape] at (3,2) {patch={1}};
				\end{tikzpicture}
			\end{multline}
			where $a$ can be any value in $\kk$. Notice that:
			\[
				d\bigg(a
				\begin{tikzpicture}[x=.35cm,y=.4cm,baseline=.3cm]
					\draw[red] (2,0)--(2,1)--(1,2);\draw[red] (2,1)--(3,2);
					\pic at (0,0) {threeframe};
					\pic[blue,scale=.8,transform shape] at (2,2) {patch={1}};
					\pic[blue,scale=.8,transform shape] at (1,0) {patch={1}};
					\pic[blue,scale=.8,transform shape] at (3,0) {patch={1}};
				\end{tikzpicture}			
				\bigg)=
				a
				\begin{tikzpicture}[x=.35cm,y=.4cm,baseline=.3cm]
					\draw[blue] (1,0)--(1,0.8);\fill[blue] (1,0.8) circle (1.5pt);
					\draw[red] (2,0)--(2,1)--(1,2);\draw[red] (2,1)--(3,2);
					\pic at (0,0) {threeframe};
					\pic[blue,scale=.8,transform shape] at (2,2) {patch={1}};
					\pic[blue,scale=.8,transform shape] at (3,0) {patch={1}};
				\end{tikzpicture}			
				-a
				\begin{tikzpicture}[x=.35cm,y=.4cm,baseline=.3cm]
					\draw[blue] (3,0)--(3,0.8);\fill[blue] (3,0.8) circle (1.5pt);
					\draw[red] (2,0)--(2,1)--(1,2);\draw[red] (2,1)--(3,2);
					\pic at (0,0) {threeframe};
					\pic[blue,scale=.8,transform shape] at (2,2) {patch={1}};
					\pic[blue,scale=.8,transform shape] at (1,0) {patch={1}};
				\end{tikzpicture}				
				-a
				\begin{tikzpicture}[x=.35cm,y=.4cm,baseline=.3cm]
					\draw[red] (1,2)--(2,0);
					\pic at (0,0) {threeframe};
					\pic[blue,scale=.8,transform shape] at (2,2) {patch={1}};
					\pic[red,scale=.8,transform shape] at (3,2) {patch={1}};
					\pic[blue,scale=.8,transform shape] at (1,0) {patch={1}};
					\pic[blue,scale=.8,transform shape] at (3,0) {patch={1}};
				\end{tikzpicture}			
				+a
				\begin{tikzpicture}[x=.35cm,y=.4cm,baseline=.3cm]
					\draw[red] (3,2)--(2,0);
					\pic at (0,0) {threeframe};
					\pic[blue,scale=.8,transform shape] at (2,2) {patch={1}};
					\pic[red,scale=.8,transform shape] at (1,2) {patch={1}};
					\pic[blue,scale=.8,transform shape] at (1,0) {patch={1}};
					\pic[blue,scale=.8,transform shape] at (3,0) {patch={1}};
				\end{tikzpicture}			
			\]
			so the morphism above is actually unique up to homotopy.
		\end{exa}
			\subsection{Rouquier formula}\label{subs_roufor}
				As a byproduct of Corollary \ref{cor_N}, we can now prove the so-called Rouquier formula 
				(conjectured in \cite{Rou_der}, and proved in \cite{LibWil}, and in \cite{Maki}).
				\begin{cor}\label{cor_Roufor}
					Let $w,v\in W$ and let $\uw$ and $\uv$ be reduced words expressing them.
					Let $\uom$ be the positive word lift of $\uw$ and $\underline{\nu}$ 
					be the negative word lift for $\uv$. Then:
					\[
						\Hom^\bullet(F_{\uom}^\bullet,F_{\underline{\nu}}^\bullet)\simeq	\begin{cases}
																								R[0] & \text{if $w=v$,} \\
																								0 & \text{otherwise.}
																							\end{cases}
					\] 
				\end{cor}
				\begin{proof}
					Suppose that $\ell(w)\ge\ell(v)$ (the other case being analogous).
					Let $N$ be the dg submodule of morphisms factoring through shorter words.
					Any enriched subexpression of $\uw$ corresponding to a shorter word has to contain 
					at least one 0 or $\emptyset$. The first such symbol from the left has to be decorated by $U$ by reducedness. 
					As we took the positive lift of $w$, this is also black.
					Then Corollary \ref{cor_N} and Remark \ref{rmk_blackU0} imply that $N$ is contractible.

					Now, if $w\neq v$ then $N$ is the whole morphism space. Instead, if $w=v$,
					the only dg light leaves map not factoring through
					shorter words is the one with subexpressions $U1\dots U1$ for both $w$ and $v$ 
					(which gives a morphism corresponding to a certain braid move from $w$ to $v$).
				\end{proof}
		\printbibliography
\end{document}